\documentclass[a4paper,12pt,reqno]{amsart}
\usepackage[a4paper,top=3cm,bottom=2cm,left=3cm,right=3cm,marginparwidth=1.75cm]{geometry}
\usepackage[english]{babel}
\usepackage{amsfonts, amsmath, amssymb, amsthm, mathtools}
\usepackage{bm}
\usepackage{tikz-cd}
\usepackage{microtype}
\usetikzlibrary{arrows}

\numberwithin{equation}{section}
\usepackage{hyperref} 

\newcommand{\N}{\mathbb{N}}
\newcommand{\Z}{\mathbb{Z}}
\newcommand{\Q}{\mathbb{Q}}
\newcommand{\R}{\mathbb{R}}
\newcommand{\C}{\mathbb{C}}
\newcommand{\F}{\mathbb{F}}

\DeclareMathOperator{\Br}{Br}
\DeclareMathOperator{\HH}{H}
\DeclareMathOperator{\Norm}{\mathbf{N}}
\DeclareMathOperator{\Spec}{Spec}

\newcommand{\Mod}[1]{\ (\mathrm{mod}\ #1)}
\newcommand{\PowerRes}[3]{\left(\frac{#1}{#2}\right)_{#3}}
\newcommand{\Jacobi}[2]{\left(\frac{#1}{#2}\right)}

\newtheorem{theorem}{Theorem}[section]
\newtheorem{proposition}[theorem]{Proposition}
\newtheorem{corollary}[theorem]{Corollary}
\newtheorem{lemma}[theorem]{Lemma}

\theoremstyle{definition}
\newtheorem{definition}[theorem]{Definition}

\newtheorem{remark}[theorem]{Remark}

\title{Diagonal quartic surfaces with a Brauer-Manin obstruction}
\begin{document}
\author{Tim Santens}
\email{Tim.Santens@kuleuven.be}
\address{KU Leuven, Departement wiskunde, Celestijnenlaan 200B, 3001 Leuven,
 Belgium}
\begin{abstract}
    In this paper we investigate the quantity of diagonal quartic surfaces $a_0 X_0^4 + a_1 X_1^4 + a_2 X_2^4 +a_3 X_3^4 = 0$ which have a Brauer-Manin obstruction to the Hasse principle. We are able to find an asymptotic formula for the quantity of such surfaces ordered by height. The proof uses a generalization of a method of Heath-Brown on sums over linked variables. We also show that there exists no uniform formula for a generic generator in this family.
\end{abstract}
\thanks{The author is supported by a grant of FWO-Vlaanderen (Research Foundation Flanders) with grant number 11I0621N}
\maketitle
\tableofcontents
\section{Introduction}
This paper is concerned with failures of the Hasse principle. A $\Q$-variety $X$ satisfies the \emph{Hasse principle} if $X(\mathbb{A}_\Q) \neq \emptyset$ implies that $X(\Q) \neq \emptyset$. There are multiple known \emph{obstructions} which explain why varieties fail the Hasse principle. The most important of these being the \emph{Brauer-Manin obstruction}. This obstruction consists of constructing a subset $X(\Q) \subset X(\mathbb{A}_\Q)^{\Br} \subset X(\mathbb{A}_\Q)$ called the \emph{Brauer-Manin set}.
We say that $X$ has a Brauer-Manin obstruction to the Hasse principle if $X(\mathbb{A}_\Q) \neq \emptyset$ but $X(\mathbb{A}_\Q)^{\Br} = \emptyset$.

In this paper we will look at K3 surfaces. These lie at the boundary of our current understanding of the Hasse principle. It was conjectured by Skorobogatov \cite{Skorobogatov2009diagonal} that the Brauer-Manin obstruction is the only obstruction for K3 surfaces. This conjecture is only known for certain Kummer varieties \cite{Harpaz2016Hasse, Harpaz2019SecondDescent} when assuming the truth of some big conjectures in number theory (finiteness of Shafarevich-Tate groups) and wide open in general. 

Recently Gvirtz, Loughran and Nakahara \cite{gvirtz2021quantitative} investigated how often diagonal degree $2$ K3 surfaces have a Brauer-Manin obstruction and were able to obtain the correct order of magnitude. We will instead look at another family of K3 surfaces, namely the diagonal quartic surfaces. These are defined by the equation
\begin{equation*}
    X_{\mathbf{a}}: a_0 X_0^4 + a_1 X_1^4 + a_2 X_2^4 + a_3 X_3^4 = 0 \subset \mathbb{P}_{\Q}^3
\end{equation*} 
for $\mathbf{a} = (a_0, a_1, a_2, a_3)$. Unlike the case of degree $2$ K3 surfaces the Brauer group has already been computed. The algebraic part by Bright \cite{bright2002computations} and the transcendental part by Gvirtz, Ieronymou and Skorobogatov \cite{ieronymou2015odd, gvirtz2021cohomology}. However, in our case there is no uniform description of the relevant algebras. When compared with the case of degree $2$ K3 surfaces this causes us significant difficulties. But the method used to resolve these issues also allows us, without too much extra effort, to find an asymptotic formula. This is a stronger type of result than the correct order of magnitude in \cite{gvirtz2021quantitative}. We hope that this method can also be applied to find asymptotic formulas for other families.

First note that due to \cite[Theorem 1.3]{bright2016failures} a positive proportion of these surfaces have local points everywhere. This positive proportion can be computed as a product of local densities.
It has been shown by Bright \cite{bright2011brauer} that $0 \%$ of these surfaces have a Brauer-Manin obstruction to the Hasse principle. In this paper we are able to improve this result by proving the following asymptotic formula.
\begin{theorem}
There exists a constant $A > 0$ such that
\begin{equation*}
    \# \left\{ \mathbf{a} \in \Z_{\neq 0}^4: \begin{array}{l}
        |a_i| \leq T, \emph{for all } i \in \{0,1,2,3\},   \\
        X_{\mathbf{a}} \emph{ has a BM-obstruction to the HP} 
    \end{array} \right\} \sim A T^2 (\log T)^{\frac{33}{8}}.
\end{equation*}
\label{Main theorem}
\end{theorem}

We will actually prove the following finer results, denote the set in the left-hand side by $N^{\Br}(T)$, we can then define the subsets
\begin{equation*}
    \begin{split}
        N_{ = \square}^{\Br}(T) &:= \{ \mathbf{a} \in N^{\Br}(T):  a_0 a_1 a_2 a_3 \in \Q^{\times 2} \}\\
        N_{ = -\square}^{\Br}(T) &:= \{ \mathbf{a} \in N^{\Br}(T): a_0 a_1 a_2 a_3 \in -\Q^{\times 2} \}\\
        N_{ \neq \pm \square}^{\Br}(T) &:= \{ \mathbf{a} \in N^{\Br}(T): a_0 a_1 a_2 a_3 \not \in \pm \Q^{\times 2} \}.
    \end{split}
\end{equation*}
Because $N^{\Br}(T) = N_{ = \square}^{\Br}(T) \sqcup N_{ = -\square}^{\Br}(T) \sqcup N_{ \neq \pm \square}^{\Br}(T)$ it suffices to prove the following theorem.
\begin{theorem}
There exist constants $0 < A, B$ such that 
\begin{align*}
    \#N_{ = \square}^{\Br}(T) = O(T^2 (\log T)^2) \\ \#N_{ = -\square}^{\Br}(T) \sim A T^2 (\log T)^{\frac{33}{8}} \\ \#N_{ \neq \pm \square}^{\Br}(T) \sim B T^2 (\log T)^{\frac{45}{16}}.
\end{align*}
\label{Main theorem split into cases}
\end{theorem}
\begin{remark}
It is geometrically interesting to instead count points $[a_0: a_1: a_2: a_3] \in \mathbb{P}^3_{\Q}(\Q)$ of bounded height such that $X_{\mathbf{a}}$ has a Brauer-Manin obstruction. This is equivalent to counting only those $\mathbf{a}$ such that $\gcd(a_0, a_1, a_2, a_3) = 1$ in Theorems \ref{Main theorem} and \ref{Main theorem split into cases}. Our method is easily able to deal with this variant. We explain in Remark \ref{Case when all a are coprime} which minor modification one has to make to prove this variant.
\label{Remark on points in projective space of bounded height}
\end{remark}

We recall that a subset of $X(\Q)$ is \emph{thin} if it is a finite union of sets of the form $f(Y(\Q))$ for $Y$ irreducible and $f: Y \to X$ finite and not birational. For example, the subset $\{ \mathbf{a} \in \mathbb{A}^4(\Q) : a_0 a_1 a_2 a_3 \in - \Q^{\times 2}\}$ is thin. Theorem~\ref{Main theorem split into cases} implies that $100 \%$ of the elements of $N^{\Br}(T)$ lie in this thin set. It has been understood for some time that in Manin-type conjectures one is required to remove a thin set. So the exponent of the logarithm which would be reasonable to geometrically interpret is $45/16$. We will give such an interpretation in \S1.1.

Somewhat unusually we do not prove the lower bound by providing an explicit example of a family of varieties which have a Brauer-Manin obstruction to the Hasse principle. Our method is actually only able to provide such a family of order of magnitude $T^2$. It seems very difficult to provide an explicit example of a family of varieties with a Brauer-Manin obstruction of the correct order of magnitude, due to the non-existence of a uniform description of the generator of $\Br(X_{\mathbf{a}})/\Br(\Q)$ in the generic case. We will make precise in \S \ref{Section uniform formula} what we mean by a uniform generator and prove, using ideas of Uematsu \cite{uematsu2014cubic,uematsu2016quadrics}, that it indeed does not exist.

An important concept in the proof is that of \emph{prolific algebras} as introduced by Bright \cite{bright2015bad}. Indeed, the proof proceeds by first showing that $50 \%$ of the surfaces where a certain algebra $\mathcal{A}$, which depends on certain choices but defines a canonical element of $\Br(X_{\mathbf{a}})/ \Br(\Q)$, is locally nowhere prolific have a Brauer-Manin obstruction. And then later finding an asymptotic formula for the amount of surfaces where $\mathcal{A}$ is nowhere locally prolific.

The fact that the thin set where $a_0 a_1 a_2 a_3 \in - \Q^{\times 2}$ dominates can also be interpreted using this notion. Namely, $\mathcal{A}$ being nowhere locally prolific and $X_{\mathbf{a}}$ being locally soluble is more common on this thin set than outside of it.

To prove this $50\%$ result we will proceed roughly as follows. Given any choice of $\mathcal{A} \in \Br(X_{\mathbf{a}})$ and a tuple $\mathbf{u} = (u_0, u_1, u_2, u_3) \in \Q^{\times}$ we construct a related algebra $\mathcal{A}_{\mathbf{u}} \in \Br(X_{\mathbf{u}^2 \mathbf{a}})$ where $\mathbf{u}^2 \mathbf{a}= (u_0^2 a_0, u_1^2 a_1, u_2^2 a_2, u_3^2 a_3)$. It is possible to compute the invariants of $\mathcal{A}_{\mathbf{u}}$ in terms of those of $\mathcal{A}$. We use this to show that as long as $\mathcal{A}$ is locally nowhere prolific half of the surfaces $X_{\mathbf{a} \mathbf{u}^2}$ have a Brauer-Manin obstruction  as $\mathbf{u}$ varies. Note that we do not need to know what the invariants of $\mathcal{A}$ actually are.
\subsection{Structure}
We will now give an overview of the structure of the paper. In the second section we will recall the notion of cyclic algebras and the Brauer-Manin obstruction. We will also describe a slight modification of a method by Bright \cite{bright2015bad} which is used to decide whether an algebra is prolific.

The third section starts by constructing two Brauer elements $\mathcal{A}, \mathcal{B}$ which we will use in our arguments. We then apply the method in the previous section to compute when they are prolific. We also look at how the invariants of the Brauer element $\mathcal{A}$ change as we multiply the coefficients $\mathbf{a}$ by a square.

The goal of the fourth section is to generalize a method of Heath-Brown \cite{heathbrown1993selmer} on sums over \emph{linked} pairs of variables. We will in particular need to deal with variables which are linked by quartic residue symbols and not just by Jacobi symbols. This generalization is proven in \S 4.2. To prove this in the desired generality, and also for other uses, we prove an uniform asymptotic formula for sums over \emph{frobenian multiplicative functions} as introduced by Loughran and Matthiesen \cite{loughran2019frobenian} in \S 4.1. 

In this section we will also prove a lemma which can be interpreted as a form of the hyperbola method for toric varieties \cite{Pieropan2020Hyperbola} for functions $f(x_1, \cdots, x_n)$ which can be written as a product of single-variable functions $f_i(x_i)$. The advantage of this new form is that it can also be applied when $\sum_{x_i \leq T} f_i(x_i)$ is not of order of magnitude a power of $T$. These analytic results are of independent interest. 

In the fifth section we count $N^{\Br}(T)$. We first show that most of the Brauer-Manin obstructions are caused by $\mathcal{A}$. We then prove that $50\%$ of the surfaces where $\mathcal{A}$ is nowhere locally prolific have a Brauer-Manin obstruction. We will then interpret the amount of surfaces where $\mathcal{A}$ is nowhere locally prolific as a sum over a large amount of linked variables and use the method developed in the third section to reduce the number of variables. We finish by computing an asymptotic formula for the resulting simpler sum using our form of the hyperbola method.

The last section is devoted to proving that there exists no uniform formula for $\mathcal{A}$. We will do this using ideas of Uematsu \cite{uematsu2014cubic, uematsu2016quadrics}. This section is independent of the rest of the paper.

\subsection{An interpretation of the exponents} We will give a geometric interpretation for the exponents in Theorem~\ref{Main theorem split into cases}. What follows will be quite speculative and we will give no proofs of the claims we make.

Let $\pi: \mathfrak{X} \to \mathbb{P}_{\Q}^3$ be the universal diagonal quartic, i.e. $\mathfrak{X} \subset \mathbb{P}_{\Q}^3 \times \mathbb{P}_{\Q}^3$ is defined by the equation
\begin{equation*}
    a_0 X_0^4 + a_1 X_1^4 + a_2 X_2^4 + a_3 X_3^4 = 0.
\end{equation*}
Proposition \ref{No Brauer-Manin obstruction if p divides one coefficient} can be interpreted as saying that $N^{\Br}(T)$ is a subset of the set of weak Campana points of the Campana orbifold $(\mathbb{P}^3_{\Q}, \frac{1}{2}D)$ where $D = \sum_{i =0}^3 \frac{1}{2} \{a_i = 0 \}$ as defined in \cite[Definition 3.3]{Pieropan2021Campana}. This is also what one should expect from arguments like in \cite{Bright2018Obstruction}. There is as of yet no Manin-type conjecture for the asymptotic behavior of weak Campana points. But these are in some sense Campana points on the limit of all log blow-ups. To be precise let $B \to \mathbb{P}_{\Q}^3$ be a smooth log blow-up of the log pair $(\mathbb{P}_{\Q}^3, D)$ and let $\Tilde{D} \subset B$ be the strict transform of $D$. The set of weak Campana points on $(\mathbb{P}^3_{\Q}, \frac{1}{2}D)$ is the intersection of the images of the Campana points on the orbifolds $(B, \frac{1}{2} \Tilde{D})$ as $B \to \mathbb{P}_{\Q}^3$ ranges over all smooth log blow-ups. We are counting Campana points with respect to a height function determined by the hyperplane class $H \in \text{Pic}(\mathbb{P}_{\Q}^3)$.

Blow up $\mathbb{P}^3_{\Q}$ first at the disjoint subvarieties $\{a_0 = a_1 = 0\}, \{a_2 = a_3 = 0\}$, then blow it up at the strict transforms of $\{ a_0 = a_2 = 0\}, \{a_1 = a_3 = 0\}$ and then blow it up at the strict transforms of $\{a_0 = a_3 = 0\}, \{a_1 = a_2 = 0\}.$ Call the resulting (toric) variety $B$. Denote by $E_{k \ell} \subset B$ the (strict transform) of the exceptional divisor corresponding to $\{a_k = a_\ell = 0\}$. 

One can check that the order of magnitude of the set of Campana points on the orbifold $(B,\frac{1}{2} \Tilde{D})$ will not change if we do more log blow-ups $B' \to B$. The order of magnitude in this case is $T^2 (\log T)^6$ where $6 = \text{rk}(\text{Pic}(B)) - 1$. We will get a further saving on the exponent $6$ from two parts. Firstly that the surfaces we count have to be locally soluble, and secondly due to properties of the Brauer group.

First note that the base change $\mathfrak{X} \times_{\mathbb{P}^3_{\Q}} B$ is not smooth. Let $\mathfrak{Y}$ be the blow-up of $\mathfrak{X} \times_{\mathbb{P}^3_{\Q}} B$ at the subvarieties $\mathfrak{X} \times_{\mathbb{P}^3_{\Q}} E_{k \ell} \cap \{X_m = X_n = 0\}$ for all $\{k, \ell, m, n\} = \{0, 1, 2, 3\}$. This variety is smooth. We will interpret the exponent of the logarithm as coming from $f: \mathfrak{Y} \to B$.

The first issue is that the surfaces we count have to be locally soluble. As in \cite{Loughran2016Fibrations} we should get for each divisor $D \subset B$ a saving in the exponent of the logarithm equal to $1$ minus the $\delta$-invariant of the set of geometrically irreducible components of the fiber $\mathfrak{Y}_{\Q(D)}$ as a $\text{Gal}(\overline{\Q(D)}/\Q(D))$-set. This $\delta$-invariant is the probability that an element of $\text{Gal}(\overline{\Q(D)}/\Q(D))$ fixes any element of the set. The only divisors whose fiber is not geometrically irreducible are the $E_{k \ell}$. The \'etale algebra corresponding to this set of irreducible components is
\begin{equation*}
    A = \Q(E_{k \ell})[\sqrt[4]{ -a_k/a_\ell}] \times \Q(E_{k \ell})[\sqrt[4]{ -a_m/a_n}].
\end{equation*}
One computes that the $\delta$-invariant is $\frac{13}{32}$ so the saving in the logarithm is $\frac{19}{32}$.

The other savings should come from some sort of relative version of the computation in Lemma~\ref{Prop if p divides 2 coefficients}. An issue when trying to do this relative computation is that the relative version of $\mathcal{A}$ does not live in $\Br(\mathfrak{Y}_{\Q(B)})$, but only in $H^1(\Q(B), \text{Pic}(\mathfrak{Y}_{\overline{\Q(B)}}))$. A relative residue map has been defined in \cite[(5.9)]{bright2016failures}
\begin{equation*}
    \HH^1(\Q(B), \text{Pic}(\mathfrak{Y}_{\overline{\Q(B)}})) \to H^0(\Q(E_{k \ell}), \HH^1(\mathfrak{Y}^{\text{sm}}_{\overline{\Q(E_{k \ell})}}, \Q/\Z)).
\end{equation*} But in this case $\mathfrak{Y}^{\text{sm}}_{\overline{\Q(E_{k \ell})}}$ is a disjoint union of affine planes so the right-hand side is trivial. This relative residue map will thus not be particularly useful. 

We see from Lemma~\ref{Prop if p divides 2 coefficients} that the savings should be equal to $1$ minus the chance that an element of $\text{Gal}(\overline{\Q(E_{k \ell})}/\Q(E_{k \ell}))$ fixes a geometrically irreducible component of the fiber $\mathfrak{Y}_{\Q(E_{k \ell})}$ and either fixes $\sqrt{a_k a_m/a_{\ell} a_n}$ or does not fix $\sqrt{-1}$. In other words it has to fix at least one of $\sqrt{\pm a_k a_m/ a_{\ell} a_n}$. The saving should thus be equal to $1$ minus the $\delta$-invariant of the $\text{Gal}(\overline{\Q(E_{k \ell})}/\Q(E_{k \ell}))$-set corresponding to the \'etale algebra
\begin{equation*}
    A[\sqrt{a_k a_m/a_{\ell} a_n}] \times A[\sqrt{-a_k a_m/a_{\ell} a_n}]
\end{equation*}
A computation shows that this $\delta$-invariant is $\frac{15}{32}$ so the saving is $\frac{17}{32}$ for each $E_{k \ell}$. The expected power of the logarithm is $6 - 6 \frac{17}{32} = \frac{45}{16}$ which is exactly what we wanted.

One can interpret the order of magnitude $T^2 (\log T)^{\frac{33}{8}}$ coming from the thin set $\{a_0 a_1 a_2 a_3 \in - \Q^{\times 2 }\}$ in a similar way, but instead of starting with $\mathbb{P}_{\Q}^3$ one has to consider (a desingularization) of the double cover $C = \{a_0 a_1 a_2 a_3 = -b^2\} \subset \mathbb{P}_{\Q}^3 \times \mathbb{A}_{\Q}^1$. 

The leading constant in Theorem \ref{Main theorem split into cases} (if one counts $\mathbf{a} = [a_0:a_1: a_2: a_3] \in \mathbb{P}^3_{\Q}(\Q)$ of bounded height as in Remark \ref{Remark on points in projective space of bounded height}) is exactly half the leading constant of the analogous counting problem where one counts the $\mathbf{a}$ such that $\mathcal{A}$ for $X_{\mathbf{a}}$ is nowhere locally prolific. This latter counting problem should have a leading constant of the same type as the counting problem in the conjecture \cite[Conj. 1.6]{Loughran2016Fibrations}.

The known cases of this conjecture \cite{Loughran2018Number, Loughran2020Zero-loci, Sofos2021Density} suggest that the constant should take the Peyre-type form
\begin{equation*}
    \frac{\alpha(B, \frac{1}{2}\Tilde{D}) \# (\text{SubBr}/ \Br \Q) \tau( U^{\text{SubBr}})} {\gamma}.
\end{equation*}
Where 
\begin{enumerate}
    \item $\alpha((B, \frac{1}{2}\Tilde{D}), H)$ is the alpha constant of the orbifold 
    $(B, \frac{1}{2}\Tilde{D})$ as defined in \cite[\S 3.3]{Pieropan2021Campana}.
    \item The group $\text{SubBr} \subset \Br \Q(\mathbb{P}^3)$ is the \emph{subordinate Brauer group} which should be defined analogously to \cite[Def. 2.8]{Loughran2018Number}.
    \item The set $U \subset \mathbb{P}^3_{\Q}(\mathbb{A}_{\Q})$ is the set of all $(\mathbf{a}_v)_v$ such that for each place $v$ the algebra $\mathcal{A}$ is not prolific on $X_{\mathbf{a}_v}$. 
    \item The subset $U^{\text{SubBr}} \subset U$ is the set of those elements of $U$ which are orthogonal to $\text{SubBr}$ with respect to the Brauer-Manin pairing.
    \item The \emph{Tamagawa number} $ \tau( U^{\text{SubBr}})$ is the volume of $U^{\text{SubBr}}$ with respect to a certain Tamagawa measure $\tau$. 
    \item The denominator $\gamma$ is a product of values of the Gamma function.
\end{enumerate}

In the published cases of this conjecture \cite{Loughran2018Number, Loughran2020Zero-loci, Sofos2021Density} the number $\gamma$ single value of the Gamma function. But in unplubished work Rome and Sofos \cite{Rome2022Number} found a case of the conjecture in which $\gamma$ is $\pi^{\frac{3}{2}} = \Gamma(\frac{1}{2})^3$. It is unknown how exactly this $\gamma$ depends on the geometry of the situation.

In our case the subordinate Brauer group must certainly be contained in $\Br B \setminus \bigcup_{k, \ell} E_{k \ell}$ because the $E_{k \ell}$ are the only divisors for which we have a non-trivial $\delta$-invariant. But $B \setminus \bigcup_{k \neq  \ell} E_{k \ell} \cong \mathbb{P}^3_{\Q} \setminus \bigcup_{k \neq \ell} \{a_k = a_{\ell} = 0\}$. And $\{a_k = a_{\ell} = 0\}$ has codimension $2$ so $\Br B \setminus \bigcup_{k, \ell} E_{k \ell} \cong \Br \mathbb{P}^{3}_{\Q} = \Br \Q$. We thus have $\# (\text{SubBr}/ \Br \Q) = 1$ and $U^{\text{SubBr}} = U$.

It is unclear how to interpret the necessary convergence factors in the Tamagawa measure. In \cite[\S5.7.2]{Loughran2018Number} a virtual Artin character is used, but this has no obvious analogue in our case.

It is conjectured by Lehman and Tanimoto \cite[\S1.2]{Lehman2017Thin} that the non-Zariski dense part of the thin set in Manin's conjecture comes from finite maps which are \'etale in codimension $1$. The map $C \to \mathbb{P}^3_{\Q}$ defining the thin set is ramified at the divisors $\{a_i = 0 \}$. But the map of orbifolds $(C \times_{\mathbb{P}^3_{\Q}} B, 0) \to (B, \frac{1}{2} \Tilde{D})$ is \'etale in codimension $1$ in the sense that the map $C \times_{\mathbb{P}^3_{\Q}} B \to B$ is unramified at all divisors except those in $\Tilde{D}$ and that the ramification degree at every divisor in $\Tilde{D}$ is equal to $2$. Our result thus agrees with the orbifold version of this conjecture.
\subsection{Notation and Conventions}
For every field $K$ fix a separable closure $\overline{K}$.

Given a number field $K$ we let $\Omega_K$ be the set of its places. Given a place $v \in \Omega_K$ we denote by $K_v$ the associated local field and $\mathcal{O}_v$ the valuation ring. If $v$ is finite we let $\F_v$ be the residue field.

We will identify $\Z/n \Z$ with the subgroup of $\Q/\Z$ generated by $\frac{1}{n}$.

For any group $G$ and number $n \in \N$ we will use $G[n]$ to denote the $n$-torsion. If $G$ is discrete and $p$ is a prime number then let $G(p')$ be the prime to $p$ torsion. If $G$ naturally has a profinite topology, the only such $G$ we will consider are \'etale fundamental groups, then let $G(p')$ be the maximal prime to $p$ quotient.

All cohomology is \'etale/Galois cohomology and all fundamental groups are \'etale fundamental groups. We omit the base point if it is not relevant. For any abelian group $B$, we identify elements of $H^{1}(X, B)$ with the corresponding maps $\pi_1(X) \to B$. 

For any number field $K$, we let $I_K$ be the monoid of integral ideals of $K$. For any extension $L/K$ of number fields and ideal $\mathfrak{a} \in I_L$ we let $\Norm_{L/K}(\mathfrak{a})$ be the relative norm. We write $\Norm(\cdot) := \Norm_K(\cdot) := \Norm_{K/\Q}(\cdot)$ for the absolute norm.

For any ideal $\mathfrak{n}$ of $K$ and $z \in \C$ we define $\tau_z(\mathfrak{n})$ as the coefficients of the Dirichlet series $\zeta_K(s)^z$. It is the unique multiplicative function such that for all $k \in \N$
\begin{equation*}
    \tau_z(\mathfrak{p}^k) = \binom{z + k - 1}{k}.
\end{equation*}
In particular, for an integer $n$, $\tau(n) := \tau_2(n)$ is the number of divisors of $n$. We will frequently use that for $z \geq 0$ 
\begin{equation*}
    \sum_{\Norm(\mathfrak{n}) \leq N} \tau_z(\mathfrak{n}) \ll_{z,K} N (\log N)^{z -1}.
\end{equation*}
We will also use the divisor bound, i.e. that for all $\varepsilon > 0$ we have $\tau_{z}(\mathfrak{n}) \ll_{z, K, \varepsilon} \Norm(\mathfrak{n})^{\varepsilon}$.

For any number $n$ we define $\omega(n) := \sum_{p \mid n} 1$ as the amount of prime ideals dividing $\mathfrak{n}$, not counted with multiplicity.

For any number $n  \in \N$ define $\text{rad}(n) := \prod_{p \mid n} p$ as the largest square-free number dividing $n$.

A prime $\mathfrak{p}$ of $K$ is \emph{totally split} if it completely splits for the extension $K/\Q$. This is equivalent to $\Norm_K(\mathfrak{p})$ being a prime number.

For any ideal $\mathfrak{m}$ a Hecke character $\chi$ of finite modulus $\mathfrak{m}$ is a character of the group $\{\mathfrak{a} \in I_K : (\mathfrak{a}, \mathfrak{m}) = 1\}/\{ (\alpha) \in I_K : \alpha \equiv 1 \Mod{ \mathfrak{m}}\}$. The conductor of $\chi$ is the minimal ideal $\mathfrak{q} \mid \mathfrak{m}$ such that there exists a Hecke character $\chi'$ of modulus $\mathfrak{q}$ such that $\chi(\mathfrak{a}) = \chi'(\mathfrak{a})$ for all $\mathfrak{a}$ coprime to $\mathfrak{m}$.
\section{Algebraic preliminaries}
\subsection{Cyclic algebras}
We recall the notion of cyclic algebras. Let $K$ be a field, $n \in \N$, $\chi \in H^{1}(K, \Z/n\Z)$ and $b \in K^{\times}$. Then $b$ induces an element in $H^{1}(K, \mu_n ) \cong K^\times /K^{\times n}$ by Kummer theory. The associated cyclic algebra is the cup product 
\begin{equation*}
    (\chi, b) := \chi \cup b \in H^{2}(K, \mu_n) \subset \Br(K).
\end{equation*}
If $\mu_n \subset K$ then, after fixing a primitive $n$-th root of unity $\zeta$, we can assign to each $a \in K^{\times}$ a character $\chi_a$ via the induced isomorphism $\mu_n \cong \Z/n \Z$. The induced cyclic algebra is 
\begin{equation*}
    (a,b)_\zeta := (\chi_a, b).
\end{equation*}
We denote the quaternion algebra $(a,b) := (a,b)_{-1}$. For more information and another construction of cyclic algebras we refer the reader to \cite[\S 2.5, \S 4.7]{gille2006central}.

In the special case when $K$ is a local field corresponding to a place $v$ of a number field $k$ we have the invariant map $\text{inv}_v:\Br(K) \to \Q/\Z$ coming from local class field theory. This map is always an injection and it is an isomorphism if $v$ is finite. We define the quadratic Hilbert symbol as
\begin{equation*}
    (a,b)_{v} := \text{inv}_v((a,b)) \in \Z/2\Z.
\end{equation*}
This is the usual quadratic Hilbert symbol, but taking values in $\Z/2\Z$ instead of $\mu_2$.
\subsection{Brauer-Manin obstruction}
For any scheme $X$, its (cohomological) \emph{Brauer group} is $\Br(X) := H^2(X, \mathbb{G}_m)$. Let $X$ be a variety over a number field $k$. We recall \cite[\S 12.3]{colliot2021brauer} that there exists a Brauer-Manin pairing 
\begin{equation*}
    X(\mathbb{A}_k) \times \Br(X) \to \Q/\Z: ((x_v)_{v},A) \to \sum_{v \in \Omega_k} \text{inv}_v( A(x_v))
\end{equation*}
where $X(\mathbb{A}_k)$ is the set of adelic points of $X$. The left kernel of this pairing is called the \emph{Brauer-Manin set} $X(\mathbb{A}_k)^{\Br}$, it contains the set of rational points $X(k)$, the right kernel contains the constant Brauer elements $\Br_0(X) := \text{Im}(\Br(k) \to \Br(X))$. If $X(\mathbb{A}_k) \neq \emptyset$ but $X(\mathbb{A}_k)^{\Br}= \emptyset$ we say that $X$ has a \emph{Brauer-Manin obstruction to the Hasse principle}.

A homomorphism of groups $B \to \Br(X)$ induces a pairing
\begin{equation*}
     X(\mathbb{A}_k) \times B \to \Q/\Z.
\end{equation*}
We call the left kernel of this pairing the \emph{Brauer-Manin set induced by $B$} and denote it by $X(\mathbb{A}_k)^B$. We have $X(\mathbb{A}_k)^{\Br} \subseteq X(\mathbb{A}_k)^B$ with equality if $B \to \Br(X)/\Br_0(X)$ is surjective. For $A \in \Br(X)$ let $\langle A \rangle \subset \Br(X)$ be the subgroup generated by $A$. We write $X(\mathbb{A}_k)^{A} := X(\mathbb{A}_k)^{\langle A \rangle}$ and call this the \emph{Brauer-Manin set induced by $A$}

The \emph{algebraic Brauer group} of $X$ is $\Br_1(X) := \text{Ker}(\Br(X) \to \Br(X_{\overline{k}}))$. The quotient $\Br(X)/ \Br_1(X)$ is the \emph{transcendental Brauer group} of $X$.
\subsection{Prolific algebras}
\label{Section criterion}
To decide whether a quartic diagonal surface has a Brauer-Manin obstruction we will use a method developed by Bright \cite{bright2015bad}. We will want to slightly modify this method. In the original one needs to know a priori that certain algebras are non-constant. We will give a variation of the method where this will not be necessary (but will follow a posteriori).

Let $X$ be a smooth variety over a number field $k$ and $v$ a finite place of $k$. Let $p$ be the characteristic of $\F_v$ and let $\mathcal{X}$ be a smooth $\mathcal{O}_v$-integral model of $X$, i.e. $\mathcal{X}$ is a smooth finite type separated $\mathcal{O}_v$-scheme and there is a fixed isomorphism of the generic fiber $\mathcal{X}_k$ with $X$. We will also assume that the special fiber $\mathcal{X}_{\F_v}$ is connected and thus irreducible. Any integral model $\mathcal{X}$ can be turned into such a model by taking an open subscheme of the smooth locus $\mathcal{X}^{\text{sm}}$. Lastly, we will assume that $\mathcal{X}(\F_v) \neq \emptyset$. Then $\mathcal{X}_{\F_v}$ has to be geometrically irreducible and $X(k_v) \supset \mathcal{X}(\mathcal{O}_v) \neq \emptyset$ by Hensel's lemma. Thus $\Br_0(X_{k_v}) = \Br(k_v)$.

The following is a slight variation of \cite[Definition 7.1]{bright2015bad}.
\begin{definition}
We say that a morphism $B \xrightarrow{\alpha} \Br(X)$ is \emph{prolific at} $v$ if the map
\begin{equation*}
    \text{inv}^B_v:  X(k_v) \to B^\vee := \text{Hom}(B, \Q/ \Z) : (b, P) \to \text{inv}_v(b(P))
\end{equation*}
is surjective. An element $A \in \Br(X)$ of order $n$ is said to be prolific if $\Z/ n \Z \to \Br(X): \frac{1}{n} \to A$ is prolific.
\end{definition}
\begin{remark}
If $B$ is prolific then $B \to \Br(X)/\Br_0(X)$ is injective. Indeed, even the map $B \to \Br(X_{k_v})/ \Br(k_v)$ has to be injective. This is because the composition of $\text{inv}^B_v$ with $B^\vee \to \alpha^{-1}(\Br(k_v))^\vee$ always has to be constant, so $\alpha^{-1}(\Br(k_v))^\vee$ is $0$ if $B$ is prolific. In particular, if $B$ is prolific in our sense then $\alpha(B) \subset \Br(X)/\Br_0(X)$ is prolific in the sense of Bright.
\end{remark}
The crucial property of prolific $B$ is the following \cite[Proposition 7.3]{bright2015bad}.
\begin{proposition}
If $B$ is prolific then it induces no Brauer-Manin obstruction to the Hasse principle.
\label{Prolific implies that the evaluation map is surjective}
\end{proposition}

Let $ \alpha: B \to \Br(X)$ be a map of abelian groups with $B$ finite and $p \nmid |B|$. We will study whether $B$ is prolific by relating $\text{inv}^B_v$ to another map. 

We can compose $\alpha$ with the residue map $\partial_v: \Br(X)(p') \to \HH^1(\mathcal{X}_{\F_v}, \Q/\Z)(p')$ coming from the Gysin sequence \cite[Theorem~3.7.1]{colliot2021brauer}. This gives us a morphism $\partial_v \alpha: B \to \HH^1(\mathcal{X}_{\F_v}, \Q/\Z)(p')$. By duality and since $p \nmid |B|$ this corresponds to a map $\pi_1(\mathcal{X}_{\F_v}) \to B^\vee$ which we will denote by $\partial^v \alpha \in \HH^1(\mathcal{X}_{\F_v}, B^\vee)$. 

The following lemma gives the relation between $\partial^v \alpha$ and $\text{inv}_v^A$. For $x \in \mathcal{X}(\mathcal{O}_v)$ use the notation $x \bmod{ \mathfrak{m}_v}$ for the image of $x$ under the composition \begin{equation*}
    \mathcal{X}(\mathcal{O}_v) \to \mathcal{X}(\mathbb{F}_v) = \mathcal{X}_{\F_v}(\F_v).
\end{equation*}
For any $\F_v$-variety $Z$ and $z \in Z(\F_v)$ we let $\text{Frob}_z$ be the image of the Frobenius $\text{Frob}_v$ under the map $z_*: \text{Gal}(\overline{\F}_v/\F_v) \to \pi_1(Z)$. This is only well-defined up to conjugacy. But $B^\vee$ is abelian so $ \partial^v \alpha( \text{Frob}_{x \bmod{ \mathfrak{m}_v}})$ is well-defined. 
\begin{lemma}
For all $x \in \mathcal{X}(\mathcal{O}_v) \subset X(k_v)$ we have
\begin{equation*}
    \emph{inv}^B_v(x) = \partial^v \alpha( \emph{Frob}_{x \bmod{ \mathfrak{m}_v}}).
\end{equation*}
\label{Relation invariant and partial}
\end{lemma}
\begin{proof}
We have to show that for all $b \in B$
\begin{equation*}
        \text{inv}_v(\alpha(b)(x)) = \partial_v \alpha(b)(\text{Frob}_{x \bmod{ \mathfrak{m}_v}}).
\end{equation*}
This follows from functoriality of the residue map \cite[Theorem 3.7.4]{colliot2021brauer}, in particular the commutativity of the following diagram
\begin{equation*}
    \begin{tikzcd}
\Br(X)(p') \arrow[d, "x"] \arrow[r, "\partial_v"] & \HH^1(\mathcal{X}_{\F_v}, \Q/ \Z)(p') \arrow[d, "x \bmod{ \mathfrak{m}_v}"]          &       \\
\Br(K_v)(p') \arrow[r, "\partial_v"]              & \HH^1(\mathbb{F}_v, \Q/\Z)(p') \arrow[rr, "\varphi \to \varphi(\text{Frob}_v)", "\cong"'] & & \Q/\Z(p').
\end{tikzcd}
\end{equation*}
The composition of the bottom row is $\text{inv}_v$ by combining \cite[Theorem 1.4.10, Lemma~2.3.3]{colliot2021brauer} and the definition of $\text{inv}_v$.
\end{proof}
This lemma will allow us to relate whether $b^\vee \in B^\vee $ lies in the image of $\text{inv}^B_v$ to the existence of points on certain varieties over $\F_v$. 

Evaluation at the Frobenius induces an isomorphism $H^{1}(\F_v, B^\vee) \cong B^\vee$. Denote the character corresponding to $b^\vee \in B^\vee$ under this isomorphism by $\chi_{b^\vee} \in H^{1}(\F_v, B^\vee)$. Consider now $ \partial^v \alpha - \chi_{b^\vee} \in \HH^1( \mathcal{X}_{\F_v}, B^\vee)$. This corresponds to a $B^\vee$-torsor which we will denote $\pi_{b^\vee}: Y_{b^\vee} \to \mathcal{X}_{\F_v}$. Over $\overline{\F}_v$ all of the $Y_{b^\vee}$ are isomorphic. Denote this $\overline{\F}_v$-variety $\overline{Y}$.
\begin{corollary}
For $x \in \mathcal{X}^{\emph{sm}} (\mathcal{O}_v)$ we have $\emph{inv}^B_v(x) = b^\vee$ if and only if $x \bmod \mathfrak{m}_v \in \pi_{b^\vee}(Y_{b^\vee}(\F_v))$.
Such an $x$ exists if and only if $Y_{b^\vee}(\F_v) \neq \emptyset.$
\end{corollary}
\begin{proof}
Let $P := x \bmod \mathfrak{m}_v$. We have $P \in  \pi_{b^\vee}(Y_{b^\vee}(\F_v))$ if and only if $\pi_{b^\vee}^{ -1}(P)$ has a $\F_v$-point. This fiber is the $B^\vee$-torsor corresponding to the composition 
\begin{equation*}
    \pi_1(\F_v) = \text{Gal}(\overline{\F}_v/\F_v) \xrightarrow{P_*} \pi_1(\mathcal{X}_{\F_v}) \xrightarrow{\partial^v \alpha - \chi_{b^\vee} } B^\vee.
\end{equation*}
A torsor over a field has a rational point if and only if it splits. So this torsor has a $\F_v$-point if and only if $(\partial^v \alpha - \chi_{b^\vee})(\text{Frob}_P) = 0$. In other words $\partial^v \alpha(\text{Frob}_P) = b^\vee$ and we are done by Lemma~\ref{Relation invariant and partial}
\end{proof}

We are now interested in when $Y_{b^\vee}(\F_v) \neq \emptyset$. A necessary condition is that one (and thus any) connected component is geometrically connected. This condition is sufficient as long as $|\F_v|$ is sufficiently large with respect to the (compactly supported) Betti numbers of $\overline{Y}$ by the Weil conjectures \cite{deligne80weil}. The following lemma describes when this happens.
\begin{lemma}
Applying the snake lemma to the following commutative diagram with exact rows
\begin{equation*}
\begin{tikzcd}
0 \arrow[r] & \pi_1(\mathcal{X}_{\overline{\F}_v}) \arrow[r] \arrow[d, "\overline{\partial^v \alpha}"] & \pi_1(\mathcal{X}_{\F_v}) \arrow[r] \arrow[d, "\partial^v \alpha - \chi_{b^\vee}"] & \emph{Gal}(\overline{\F}_v/ \F_v) \arrow[r] \arrow[d] & 0 \\
0 \arrow[r] & B^\vee \arrow[r] & B^\vee \arrow[r] & 0  \arrow[r] & 0
\end{tikzcd}
\end{equation*}
defines a morphism $\delta^{b^\vee}: \emph{Gal}(\overline{\F}_v/ \F_v) \to \emph{coker}(\overline{\partial^v \alpha})$.
It is equal to $\delta^{b^\vee} = \delta^0 - \chi_{b^\vee}$.

In particular, the connected components of $Y_{b^\vee}$ are geometrically connected if and only if $\delta_{b^\vee} = 0$.
\label{Connected components are geometrically connected}
\end{lemma}
\begin{proof}
The first statement follows from standard homological algebra. To prove the second statement we note that the set of connected components of $Y_{b^\vee}$ has a transitive $B^\vee$-action with stabilizer $\text{im}(\partial^v \alpha - \chi_{b^\vee})$, analogously the set of connected components of $\overline{Y}$ has a transitive $B^{\vee}$ action with stabilizer $\text{im}(\overline{\partial^v \alpha})$. Every connected component is thus geometrically connected if and only if the natural map $\text{im}(\overline{\partial^v \alpha}) \to \text{im}(\partial^v \alpha - \chi_{b^\vee})$ is an isomorphism. The snake lemma implies that this happens if and only if $\delta^{b^\vee} = 0$.
\end{proof}
\begin{remark}
\leavevmode
\begin{enumerate}
    \item The morphism $\delta^{b^\vee}$ is 0 if and only 
    \begin{equation*}
        \delta^{b^\vee}(\text{Frob}_v) = \delta^0(\text{Frob}_v) - b^\vee = 0 \in \text{coker}(\overline{\partial^v \alpha}).
    \end{equation*} 
    The set of such $b^\vee$ is thus an $\text{im}(\overline{\partial^v \alpha})$-coset of $B^\vee$.
    \label{Delta and Frobenius}
    \item If $\overline{Y}$ is connected then $\overline{\partial^v \alpha}$ is surjective so $\delta^{b^\vee}$ will always be 0.
    \item Dually, the boundary map $\delta_{b^\vee}$ which one gets by applying the snake lemma to the following diagram has to be zero.
    \begin{equation*}
        \begin{tikzcd}[column sep= small]
            0 \arrow[r] & 0 \arrow[r] \arrow[d]        & B \arrow[r] \arrow[d, "\partial_v \alpha - b^\vee"]     & B \arrow[r] \arrow[d, "\overline{\partial_v \alpha}"]             & 0 \\
            0 \arrow[r] & {H^1(\F_v, \Q/\Z)} \arrow[r] & {H^1(\mathcal{X}_{\F_v}, \Q/\Z)} \arrow[r] & {H^1(\mathcal{X}_{\overline{\F}_v}, \Q/\Z)} \arrow[r] & 0
        \end{tikzcd}
    \end{equation*}
    Here $b^\vee$ is identified with the composition $B \xrightarrow{b^\vee} \Q/\Z \cong \HH^1(\F_v, \Q/\Z).$
\end{enumerate}
\label{Remarks about delta}
\end{remark}
If $\mathcal{X}_{\F_v}$ is the smooth locus of a cone over a curve we can say more. What follows will be analogous to \cite[Theorem~6.5]{bright2015bad}. We remark that the proof written there is incomplete. It is stated that fibrations with fiber $\mathbb{A}^1_{\F_v}$ induce isomorphisms on fundamental groups. This is already false for the fibration $\mathbb{A}^1_{\F_v} \to \Spec(\F_v)$ due to the existence of Artin-Schreier extensions of $\mathbb{A}^1_{\F_v}$. The statement is probably true for the maximal prime to $p$ quotient but we will just need the following special case.
\begin{lemma}
    Let $K$ be a field of characteristic $p$. Let $X \subset \mathbb{P}^{n+1}_K$ be the cone over a smooth variety $Y \subset \mathbb{P}^n_K$.
    The projection $X^{\emph{sm}} \to Y$ induces an isomorphism $\pi_1(X^{\emph{sm}})(p') \cong \pi_1(Y)(p')$
    \label{Fundamental group of cone}
\end{lemma}
\begin{proof}
Let $P \in X$ be the cone point. Consider the blow-up $X'$ at this point with exceptional divisor $E$. It is a $\mathbb{P}^1$-bundle over $Y$ and $X^{\text{sm}} = X \setminus P = X' \setminus E$. So $X^{\text{sm}}$ is the complement of a relative normal crossing divisor $E$ in the smooth proper $Y$-variety $X'$. The desired statement follows from \cite[Examples XIII.4.4]{SGA1}.
\end{proof}
In particular, if $\mathcal{X}_{\F_v}$ is the smooth locus of a cone over a smooth proper curve $C$ then $\pi_1(\mathcal{X}_{\F_v})(p') \to \pi_1(C)(p')$ is an isomorphism. This induces an isomorphism $H^1(C, \Q/\Z)(p') \cong \HH^1(\mathcal{X}_{\F_v}, \Q/\Z)(p')$.
\begin{theorem}
Assume that $\mathcal{X}_{\F_v}$ is the smooth locus of a cone over a smooth proper curve $C$ of genus $g$ and
\begin{equation}
    |\F_v| > (g' + \sqrt{g'^2 -1})^2 \text{ where } g' = |\emph{im}(\overline{\partial^v \alpha})| (g-1) + 1.
    \label{p has to be sufficiently large}
\end{equation}
The image of $\emph{inv}^B_v$ contains $b^\vee$ if and only if $\delta^{0}(\emph{Frob}_v) = b^\vee \in \emph{coker}(\overline{\partial^v \alpha}).$
\label{Image inv for cone over curve}
\end{theorem}
\begin{proof}
By Lemma~\ref{Connected components are geometrically connected} and Remark \ref{Remarks about delta}.\ref{Delta and Frobenius} it suffices to show that if all of the connected components of $Y_{b^\vee}$ are geometrically connected then it has a $\F_v$-point. By Lemma~\ref{Fundamental group of cone} there exists an \'etale cover $D \to C$ such that $Y_{b^\vee} = D \times_C \mathcal{X}_{\F_v}$. We will show that $D(\F_v) \neq \emptyset$. Every connected component of $D$ is a connected \'etale $\text{im}(\overline{\partial^v \alpha})$-torsor over $C$. By the Hurwitz formula this is a curve of genus $g'$. The Weil conjectures \cite{deligne80weil} together with condition \eqref{p has to be sufficiently large} give us what we want.
\end{proof}
\begin{corollary}
Assume that $B \to H^{1}(C_{\overline{\F}_v}, \Q/\Z)(p')$ is injective, equivalently, that the associated \'etale cover of $C_{\overline{\F}_v}$ is connected, in the situation of Theorem~\ref{Image inv for cone over curve}. Then $B$ is prolific and thus induces no Brauer-Manin obstruction to the Hasse principle.
\label{No Brauer-Manin obstruction for cone over curve}
\end{corollary}
\begin{proof}
Dually, the map $\overline{\partial^v \alpha}$ is surjective so $\text{coker}(\overline{\partial^v \alpha}) = 0$ and the condition in the theorem is automatically satisfied.
\end{proof}
We end this subsection with a lemma which will allow us to compute $\partial_v$ for cyclic algebras.
\begin{lemma}
Let $\chi \in  H^{1}(k(X), \Z/ n \Z)$ and $a \in H^{1}(k(X), \mu_n)$ such that $(\chi, a) \in \Br(X) \subset \Br(k(X))$. Let $v \not \in S$ be a finite place of $k$ coprime to $n$ such that $\mathcal{X}_{\F_v}$ is irreducible. In this case the valuation $v$ can be extended to $k(X)$. Assume that $\chi$ is unramified at $\mathcal{X}_{\F_v}$. Let $\chi|_{\F_v} \in H^{1}(\F_v(\mathcal{X}_{\F_v}), \Z/n\Z)$ be the induced character. Then
\begin{equation}
\partial_v((\chi, a)) = v(a) \chi|_{\F_v}.
    \label{Equation residue of cyclic algebras}
\end{equation}
\label{Lemma~residue of cyclic algebras}
\end{lemma}
\begin{proof}
This is a combination of the second bullet point of \cite[p.37]{colliot2021brauer} and applying \cite[Lemma~2.33]{colliot2021brauer} to compare the different notions of residue.
\end{proof}
\section{Computations for the Brauer-Manin obstruction}
Let $X_{\mathbf{a}} \subset \mathbb{P}^3_{\Q}$ be the smooth surface defined by the equation 
\begin{equation}
    a_0 X_0^4 + a_1 X_1^4 + a_2 X_2^4 + a_3 X_3^4 = 0
    \label{Defining equation}
\end{equation} for $\mathbf{a} := (a_0, a_1, a_2, a_3) \in (\Q^{\times})^4$. Write $\theta_{\mathbf{a}} := a_0 a_1 a_2 a_3$. Two such surfaces are \emph{equivalent} if one is obtained from the other by permuting the coefficients, multiplying the coefficients by fourth powers and multiplying them by a common factor. There is an obvious isomorphism between any two equivalent surfaces which permutes the coordinates and multiplies them by a constant.

\subsection{Constructing Brauer group elements}
Assume that $X_{\mathbf{a}}$ has local points everywhere. The following proposition gives two elements of $\Br(X_{\mathbf{a}})$. The first of these is due to Bright \cite[Lemma~2.1]{bright2011brauer} and the second is, written somewhat differently, due to Swinnerton-Dyer \cite[(43)]{swinnerton2000quartic}. We will give a different argument that these Brauer elements are unramified.
\begin{proposition}
\leavevmode
\begin{enumerate}
    \item Let $Y_{\mathbf{a}}$ be the smooth quadric surface defined by 
    \begin{equation*}
        a_0 Y_0^2 + a_1 Y_1^2 + a_2 Y_2^2 + a_3 Y_3^2 = 0
    \end{equation*} and let $\alpha: X_{\mathbf{a}} \to Y_{\mathbf{a}}$ the morphism given by $Y_i = X_i^2$. Because $X_{\mathbf{a}}$ has local points everywhere, $Y_{\mathbf{a}}$ must have local points everywhere and thus has a rational point by Hasse-Minkowski. Let $f$ be an equation defining the tangent plane at this rational point, then
    \begin{equation*}
        \mathcal{A} := \left(\frac{\alpha^*f}{X_3^2}, \theta_{\mathbf{a}}\right) \in \Br(X_{\mathbf{a}}).
    \end{equation*}
    \item If $\theta_{\mathbf{a}}\in \Q^{\times 2}$ let $\{k,\ell,m,n\} = \{0,1,2,3\}$ and let $Z_{\mathbf{a}}^{k \ell m}$ be the quadric curve defined by 
    \begin{equation*}
    a_k \Norm_{\Q( \sqrt{-\frac{a_\ell}{a_k}})}\left(Z_k + \sqrt{-\frac{a_\ell}{a_k}} Z_\ell\right)+ a_m Z_m^2 = a_k Z_k^2 + a_\ell Z_\ell^2 + a_m Z_m^2 = 0.
    \end{equation*}
    For any choice of $\sqrt{\theta_{\mathbf{a}}}$ let $\gamma_{\sqrt{\theta_{\mathbf{a}}}}: X_{\mathbf{a}} \to Z_{\mathbf{a}}^{k \ell m}$ be the morphism defined by 
    \begin{equation*}
    \begin{split}
        &Z_k = X_k^2 X_m^2 - \frac{\sqrt{\theta_{\mathbf{a}}}}{a_k a_m} X_\ell^2 X_n^2 \\
        &Z_\ell = \frac{\sqrt{\theta_{\mathbf{a}}}}{a_\ell a_m} X_k^2 X_n^2 + X_\ell^2 X_m^2 \\
        &Z_m = X_m^4 + \frac{a_n}{a_m} X_n^4.
    \end{split}
    \end{equation*} A priori this is only a rational map, but it can be extended to a morphism. By Hasse-Minkowski $Z_{\mathbf{a}}^{k \ell m}$ has a rational point. Let $h$ be an equation defining the tangent line at this point, then
    \begin{equation*}
        \mathcal{B}_{k \ell m ; \sqrt{\theta_{\mathbf{a}}}} := \left( \gamma_{\sqrt{\theta_{\mathbf{a}}}}^*\frac{ h}{Z_m}, -a_k a_m \sqrt{\theta_{\mathbf{a}}}\right) \in \Br(X_{\mathbf{a}}).
    \end{equation*}
\end{enumerate}
These algebras are independent of the choices made as elements of $\Br(X_{\mathbf{a}})/\Br(\Q)$.
\end{proposition}
\begin{proof}
By the Grothendieck purity theorem \cite[III: Theorem 6.1]{grothendieck1968groupe} it suffices to prove that each of these Brauer elements is unramified at all the divisors of $X_{\mathbf{a}}$. Note that $\mathcal{A}$ is the pullback along $\alpha$ of the following Brauer element described in \cite[\S 5.8]{colliot2009brauer}
\begin{equation}
    \left( \frac{f}{Y_3}, a_0 a_1 a_2 a_3 \right) \in \Br(Y_{\mathbf{a}} \setminus \{Y_3 = 0 \}).
\label{A as a pullback}
\end{equation}
This is independent of the choice of $f$ as an element of $\Br(Y_{\mathbf{a}} \setminus \{Y_3 = 0 \})/\Br(\Q)$ which shows that $\mathcal{A}$ is well-defined as an element of $\Br(\Q(X_{\mathbf{a}}))/\Br(\Q)$. It follows that $\mathcal{A}$ is unramified except possibly at the divisor $\{X_3 = 0\}$. But $f$ is ramified at $\{X_3 = 0\}$ of order 2 so by \cite[Proposition 8.2]{garibaldi2003cohomological} $\mathcal{A}$ is also unramified at this divisor. 

We now show that $\gamma_{\sqrt{\theta_{\mathbf{a}}}}$ is indeed a morphism. To prove this we may work over the algebraic closure $\overline{\Q}$ by Galois descent. The given equations define $\gamma$ except when $X_m^2 = \sqrt{-\frac{a_n}{a_m}} X_n^2$ and $X_k^2 = -\frac{\sqrt{\theta_{\mathbf{a}}}}{a_k a_n} \sqrt{-\frac{a_n}{a_m}} X_\ell^2$ for any choice of the square root $\sqrt{-\frac{a_n}{a_m}}$. Note that this includes the cases $X_k = X_\ell = 0$, $X_n = X_m = 0$. Multiplying all the coordinates by $(X_k^2 - \frac{\sqrt{\theta_{\mathbf{a}}}}{a_k a_n} \sqrt{-\frac{a_n}{a_m}} X_\ell^2)/(X_m^2 - \sqrt{-\frac{a_n}{a_m}} X_n^2)$ gives coordinates which are defined in this case. Indeed
\begin{equation}
    \begin{split}
        \frac{X_k^2 X_m^2 - \frac{\sqrt{\theta_{\mathbf{a}}}}{a_k a_m} X_\ell^2 X_n^2}{X_m^2 - \sqrt{-\frac{a_n}{a_m}} X_n^2} &= X_k^2 + X_n^2\frac{\sqrt{-\frac{a_n}{a_m}}X_k^2 - \frac{\sqrt{\theta_{\mathbf{a}}}}{a_k a_m}  X_\ell^2 }{X_m^2 - \sqrt{-\frac{a_n}{a_m}}X_n^2} \\ &= X_k^2 + \frac{a_m}{a_k} \sqrt{-\frac{a_n}{a_m}}X_n^2 \frac{X_m^2 + \sqrt{-\frac{a_n}{a_m}}X_n^2}{ X_k^2 - \frac{\sqrt{\theta_{\mathbf{a}}}}{a_k a_n} \sqrt{-\frac{a_n}{a_m}} X_\ell^2 }
        \\ 
        \frac{\frac{\sqrt{\theta_{\mathbf{a}}}}{a_\ell a_m} X_k^2 X_n^2 + X_\ell^2 X_m^2}{X_m^2 - \sqrt{-\frac{a_n}{a_m}} X_n^2} &= X_\ell^2 + X_n^2\frac{\sqrt{-\frac{a_n}{a_m}} X_\ell^2 + \frac{\sqrt{\theta_{\mathbf{a}}}}{a_\ell a_m} X_k^2 }{X_m^2 - \sqrt{-\frac{a_n}{a_m}}X_n^2} \\ &= X_\ell^2 + \frac{\sqrt{\theta_{\mathbf{a}}}}{a_k a_\ell } X_n^2 \frac{X_m^2 + \sqrt{-\frac{a_n}{a_m}}X_n^2}{X_k^2 - \frac{\sqrt{\theta_{\mathbf{a}}}}{a_k a_n} \sqrt{- \frac{a_n}{a_m}}X_\ell^2 }
        \\
        \frac{X_m^4 + \frac{a_n}{a_m}X_n^4}{X_m^2 - \sqrt{-\frac{a_n}{a_m}} X_n^2} &= X_m^2 + \sqrt{-\frac{a_n}{a_m}} X_n^2.
        \end{split}
\label{Other equations for gamma}
\end{equation}
The divisor of $h$ is a square so $\mathcal{B}_{k \ell m; \sqrt{\theta_{\mathbf{a}}}}$ can only be ramified at the divisors contained in $\gamma^*\{Z_m = 0\}$. By the definition of $\gamma$ and \eqref{Other equations for gamma} this is the union of 
\begin{equation*}
    \{X_m^2 = \sqrt{-\frac{a_n}{a_m}} X_n^2, X_k^2 = \frac{\sqrt{\theta_{\mathbf{a}}}}{a_k a_n} \sqrt{-\frac{a_n}{a_m}} X_\ell^2\}
\end{equation*}
for the two different choices of $\sqrt{-\frac{a_m}{a_n}}$. So $-a_k a_m \sqrt{\theta_{\mathbf{a}}}$ is a square in the residue field.

If we have $2$ choices of tangent lines defined by $g_1, g_2$ then their quotient $g_1/g_2$ is a rational function with square divisor and is thus a square up to multiplication by a constant. The associated quaternion algebras thus differ by a constant algebra.
\end{proof}
\begin{remark}
The algebras $\mathcal{A}, \mathcal{B}_{k \ell m ; \sqrt{\theta_{\mathbf{a}}}}$ exist more generally over other fields $K$ as long as the quadrics $Y, Z_{\mathbf{a}}^{k \ell m}$ have $K$-points. For the same reason as over $\Q$ they are still well-defined as elements of $\Br(X_{\mathbf{a}})/ \Br(K)$. We will use this in what follows to see that $\mathcal{A}$ as an element of $\Br(X_{\Q_p})/\Br(\Q_p)$ can be computed by taking a tangent line at a $\Q_p$-point of $Y_{\mathbf{a}}$. It also implies that $\overline{\partial_v \mathcal{A}} \in \HH^1(\mathcal{X}_{\overline{\F}_p}, \Q/\Z)$ can be computed by taking the tangent line at a $K$-point of $Y_{\mathbf{a}}$ for an unramified $K/\Q_p$ extension. Analogously for $ \mathcal{B}_{k \ell m ; \sqrt{\theta_{\mathbf{a}}}}$.
\label{A, B_klm and change of fields}
\end{remark}
\subsection{Local computations}
\label{Subsection Local computation}
In this subsection we will compute the local invariants of $\mathcal{A}$ utilizing the methods of \S \ref{Section criterion}. Let us first describe when $X_{\mathbf{a}}$ has a $\Q_p$-point.
\begin{lemma}
Let $p$ be an odd prime number, then $X_{\mathbf{a}}( \Q_p) \neq \emptyset$ only if one of the following two conditions is true 
\begin{enumerate}
    \item There exist $k \neq \ell$ such that $-\frac{a_k}{a_\ell} \in \Q_p^4$.
    \item There exist pairwise distinct $k, \ell, m$ such that $v_p(a_k) \equiv v_p(a_\ell) \equiv v_p(a_m) \pmod{4}$.
\end{enumerate}
Conversely, $X_{\mathbf{a}}(\Q_p) \neq \emptyset$ if the first condition holds or if the second one holds and $p > 33$.
\label{Local solubility}
\end{lemma}
\begin{proof}
If $X_{\mathbf{a}}(\Q_p) \neq \emptyset$ then by fixing a solution and then taking an equivalent surface while appropriately changing the solution we may assume that $a_0 \in \Z_p^{\times}$, $a_1, a_2, a_3 \in \Z_p$ and that there exist $x_0, x_1, x_2, x_3 \in \Z_p^{\times}$ such that $a_0 x_0^4 + a_1 x_1^4 + a_2 x_2^4 + a_3 x_3^4 = 0$. By reducing modulo $p$ we see that at least one of the other coefficients is a $p$-adic unit. If there is exactly one other coefficient $a_k$ which is a $p$-adic unit then $-\frac{a_k}{a_0} \equiv (\frac{x_k}{x_0})^4 \not \equiv 0 \pmod{p}$, so by Hensel's lemma $-\frac{a_k}{a_0} \in \Q_p^{\times 4}$ and the first condition holds. Otherwise the second condition holds.

If the first condition holds then without loss of generality $k = 0, \ell = 1$ and $[1: \sqrt[4]{-a_1/a_0}: 0: 0] \in X_{\mathbf{a}}(\Q_p)$. For the second condition take an equivalent surface such that $a_k, a_\ell, a_m \in \Z_p^{\times}$. In this case the equation $a_k X_k^4 + a_\ell X_\ell^4 + a_m X_m^4 =  0 $ defines a smooth proper curve of genus $3$ which has a $\F_p$-point if $p > 33$ because of the Weil conjectures \cite{deligne80weil}. It thus has a $\Z_p$-point by Hensel's lemma.
\end{proof}

\begin{proposition}
Let $\{k, \ell, m, n\} = \{0,1,2,3\}$. If there exists a prime $p > 16897$ such that $v_p(a_n) = 1$ and $p \nmid a_k a_\ell a_m$, then $X_{\mathbf{a}}$ has no Brauer-Manin obstruction to the Hasse principle.
\label{No Brauer-Manin obstruction if p divides one coefficient}
\end{proposition}
\begin{proof}
This is the same as \cite[Corolllary~7.10]{bright2015bad} except that it holds for more primes. The proof is the same except that we can bound $|\Br(X_{\mathbf{a}})/\Br(\Q)|$ by $32$ because in this case $\Br_1(X_{\mathbf{a}}) = \Br(X_{\mathbf{a}})$ due to \cite[Theorem 1.1]{ieronymou2015odd} and \cite[Main theorem]{gvirtz2021cohomology}. And $|\Br_1(X_{\mathbf{a}})/\Br(\Q)| \leq 32$ by \cite[Appendix A]{bright2002computations}.
\end{proof}

\begin{lemma}
Let $\{k, \ell, m, n\} = \{0,1,2,3\}$ and $p > 97$ an odd prime. If for all $i \neq j \in \{0, 1, 2, 3 \}$ we have $a_i a_j \not = 1,-1,2,-2 \in \Q^{\times}/ \Q^{\times 2}$, $\theta_{\mathbf{a}}\in \Q^{\times 2}$, $v_p(a_n) \equiv 2 \pmod{4}$ and $p \nmid a_k a_\ell a_m$ then $X_{\mathbf{a}}$ has no Brauer-Manin obstruction to the Hasse principle.
\label{No Brauer-Manin obstruction for theta square}
\end{lemma}
\begin{proof}
In this situation $\Br_1(X_{\mathbf{a}}) = \Br(X_{\mathbf{a}})$ because of \cite[Theorem 1.1]{ieronymou2015odd} and \cite[Main theorem]{gvirtz2021cohomology}. The only possible case of \cite[Appendix A]{bright2002computations} is A131, hence $\Br(X_{\mathbf{a}})/\Br(\Q) \cong \Z/ 2\Z$. We will prove that $\mathcal{B}_{k \ell m; \sqrt{\theta_{\mathbf{a}}}}$ is prolific. This suffices because in that case $\mathcal{B}_{k \ell m; \sqrt{\theta_{\mathbf{a}}}}$ generates $\Br(X_{\mathbf{a}})/\Br(\Q)$.

We will apply Corollary~\ref{No Brauer-Manin obstruction for cone over curve}. Let $\mathcal{X}$ be the smooth locus of the integral model given by the same equation. The special fiber $\mathcal{X}_{\F_p}$ is the smooth locus of the cone over the curve $C: a_k X_k^4 + a_\ell X_\ell^4 + a_m X_m^4 = 0$. To compute $\partial_p \mathcal{B}_{k \ell m ; \sqrt{\theta_{\mathbf{a}}}}$ we may work over the maximal unramified extension of $\Q_p$ by Remark \ref{A, B_klm and change of fields}. We may thus assume that $a_k = a_\ell = a_m = 1$ and take $h = Z_k + \sqrt{-1} Z_m = 0$ as the equation defining a tangent line of $Z_{\mathbf{a}}^{k \ell m}$. Then 
\begin{equation*}
    \gamma^*_{\sqrt{\theta_{\mathbf{a}}}} \frac{h}{Z_m} \equiv \frac{X_k^2 X_m^2 + \sqrt{-1}X_m^4}{X_m^4} \equiv \frac{X_k^2}{X_m^2} + \sqrt{-1} \pmod{p}.
\end{equation*}
Lemma~\ref{Lemma~residue of cyclic algebras} implies that $\partial_p \mathcal{B}_{k \ell m ; \sqrt{\theta_{\mathbf{a}}}} = \frac{X_k^2}{X_m^2} + \sqrt{-1}$ since $v_p(\theta_{\mathbf{a}})$ is odd.

We check that the \'etale cover of $C$ defined by $Y^2 = X_k^2 + \sqrt{-1} X_m^2$ is geometrically irreducible using the \verb|Magma| function \verb|IsAbsolutelyIrreducible()|. The full script can be found on the authors webpage.
Hence $\mathcal{B}_{k \ell m ; \sqrt{\theta_{\mathbf{a}}}}$ is prolific by Corollary~\ref{No Brauer-Manin obstruction for cone over curve}.
\end{proof}

We say that the algebra $\mathcal{A}$ is \emph{normalized} if the coefficients of $f$ in the definition are integral and coprime. It is clear that we may always assume that $\mathcal{A}$ is normalized. More generally, for any prime $p$ we will say that $\mathcal{A}$ is $p$-\emph{normalized} if all the coefficients are $p$-adic integers and at least one of them is a unit. So $f$ and $\alpha^* f$ will be non-zero polynomials when reduced modulo $p$ which we will denote 
$\Tilde{f}$, respectively $\alpha^* \Tilde{f}$. These normalised algebras are useful since we are then able to write more explicit formulas for the invariant.
\begin{lemma}
Let $\{k, \ell, m, n\} = \{0,1,2,3\}$. Let $p$ be an odd prime such that $p \nmid a_k a_\ell a_m$, $v_p(a_n) = 0,2$, $X_{\mathbf{a}}(\Q_p) \neq \emptyset$. If $\mathcal{A}$ is $p$-normalized, then $\emph{inv}_p(\mathcal{A}(\cdot)) = 0$.
\label{Lemma if p divides at most one coefficient}
\end{lemma}
\begin{proof}
Let $\mathcal{X}$ be the integral model given by the same equation. It follows from looking at primitive solutions that $\mathcal{X}^{\text{sm}}(\Z_p) = X_{\mathbf{a}}(\Q_p)$. The special fiber $\mathcal{X}^{\text{sm}}_{\F_v}$ is irreducible, so $v_p$ can be extended to $\Q_p(X)$. By Lemma~\ref{Relation invariant and partial} it suffices to show that $\partial_p \mathcal{A} = 0$. Since $\mathcal{A}$ is normalized $\alpha^* f$ will be non-zero when reduced modulo $p$. The reduction $\alpha^* \Tilde{f}$ can not be zero on the special fiber since it has smaller degree than the defining equation so $v_p(\frac{\alpha^* f}{X_3^2}) = 0$. We are done by applying Lemma~\ref{Lemma~residue of cyclic algebras}.
\end{proof}
\begin{lemma}
Let $\{k, \ell, m, n\} = \{0,1,2,3\}$. Let $p$ be an odd prime such that $v_p(a_m), v_p(a_n) \in \{1,3\}$, $p \nmid a_k a_\ell$ and $X_{\mathbf{a}}(\Q_p) \neq \emptyset$. 

If $\theta_{\mathbf{a}}\not \in \Q_p^{\times 2}$ and $-\frac{a_k}{a_\ell} \in \Q_p^{\times 4}$ then $\mathcal{A} \not \in \Br_0(X)$. If on top of this $\mathcal{A}$ is $p$-normalized then for all $P \in X_{\mathbf{a}}(\Q_p)$
\begin{equation*}
    \emph{inv}_p(\mathcal{A}(P)) = 
    \begin{cases}
    0 \emph{ if } \alpha^*\Tilde{f}(P) = 0 \\
    \frac{1}{2} \emph{ if } \alpha^*\Tilde{f}(P) \neq 0.
    \end{cases}
\end{equation*}
\label{Computation invariant if p divides 2 coefficients}

In particular, if $p \equiv 1 \pmod{4}$ then $\mathcal{A}$ is prolific and if $p \equiv 3 \pmod{4}$ then $\emph{inv}_p(\mathcal{A}(\cdot)) $ is constant.
\end{lemma}
\begin{proof}
Because $-\frac{a_k}{a_\ell} \in \Q_p^{\times 4}$ and $\theta_{\mathbf{a}}\not \in \Q_p^2$ we must have $-\frac{a_m}{a_n} \not \in \Q_p^{\times 2}$. Let $\mathcal{X}$ be the integral model defined by the same equation. By looking at primitive solutions modulo $p^4$ we see that $\mathcal{X}^{\text{sm}}(\Z_p) = X_{\mathbf{a}}(\Q_p)$. Let $\mathcal{Y}$ be the integral model of $Y_{\mathbf{a}}$ given by the same equation. For the same reason $\mathcal{Y}^{\text{sm}}(\Z_p) = Y_{\mathbf{a}}(\Q_p)$. 

Take a $b \in \Z_p^\times$ such that $-\frac{a_k}{a_\ell} = b^4$. The special fiber $\mathcal{X}^{\text{sm}}_{\F_p}$ is equal to $\{X_k^4 = b^4 X_\ell^4\} \setminus \{X_k = 0 = X_\ell\}$. This has $4$ geometrically connected components, $\{X_k = \sqrt{-1}^t b X_\ell \}$ for $t = 0,1,2,3$. It thus has $4$ or $3$ connected components for $p \equiv 1 \pmod{4}$, $p \equiv 3 \pmod{4}$ respectively. 

The plane defined by $f$ is tangent to a $\Q$-point of $Y_{\mathbf{a}}$, so it is tangent to a $\Z_p$-point of $\mathcal{Y}^{\text{sm}}$. The reduction of this plane modulo $p$ is defined by $\Tilde{f}$ because $\mathcal{A}$ is normalized. But $\mathcal{Y}^{\text{sm}}_{\F_p} = \{Y_k^2 = b^4Y_\ell^2\} \setminus \{Y_k = 0 = Y_\ell\}$ is a disjoint union of $2$ planes, so $\{ \Tilde{f} = 0\}$ must be one of these two planes. Over each of these $2$ planes lie $2$ geometric components of $\mathcal{X}^{\text{sm}}_{\F_p}$.

Let $D$ be one of the connected components of $\mathcal{X}^{\text{sm}}_{\F_p}$ and $D^c$ the complement in $\mathcal{X}^{\text{sm}}_{\F_p}$. Let $\mathcal{X}^D := \mathcal{X}^{\text{sm}} \setminus D^c$. The special fiber $\mathcal{X}^D_{\F_p} = D$ is connected. By Lemma \ref{Relation invariant and partial} it suffices to compute the residue $\partial_D \mathcal{A}$ for $\mathcal{X}^D$. We can extend $v_p$ to a valuation $v_D$ of $\Q_p(X)$. Since $v_p(\theta_{\mathbf{a}})$ is even and $\theta_{\mathbf{a}}\not \in \Q_p^{\times 2}$ we can apply Lemma~\ref{Lemma~residue of cyclic algebras} to find that
\begin{equation*}
    \partial_D \mathcal{A} = \frac{1}{2} v_D\Big(\frac{\alpha^* f}{X_3^2}\Big) \in \HH^1(\F_p, \Z/2\Z) \subset \HH^1(D, \Q/\Z).
\end{equation*}

If $D \not \subset \{\alpha^*\Tilde{f} = 0\}$ then $v_D(\frac{\alpha^*f}{X_3^2}) = 0$ so $\partial_p \mathcal{A} = 0$. 

If we change $\mathcal{A}$ by a constant algebra then $\partial_D \mathcal{A}$ changes by a constant independent of $D$. It thus suffices to show that for a single choice of $f$ we have $\partial_D \mathcal{A} = \frac{1}{2}$ if $D \subset \{\alpha^*\Tilde{f} = 0\}$. By Remark \ref{A, B_klm and change of fields} we may take $f = Y_k - b^2Y_\ell$ and thus $\alpha^* f = X_k^2 - b^2X_\ell^2$. Then
\begin{equation*}
    \frac{X_k^2 - b^2X_\ell^2}{X_3^2} = -\frac{a_m X_n^4 + a_n X_n^4}{a_k(X_k^2 + b^2 X_\ell^2)X_3^2} = -p^i \frac{\frac{a_m}{p^i} X_m^4 + \frac{a_n}{p^i} X_n^4}{a_k(X_k^2 +b^2 X_\ell^2)X_3^2}
\end{equation*}
Where $i = \min(v_p(a_m), v_p(a_n))$, so $i$ is odd, $\frac{a_m}{p^i}, \frac{a_n}{p^i} \in \Z_p$ and at least one of them is a unit. Because $D \not \subset \{X_k^2 = - b^2X_\ell^2\}, \{\frac{a_m}{p^i} X_m^4 + \frac{a_n}{p^i}X_n^4 = 0\}$ we get $\partial_D( \mathcal{A}) = \frac{i}{2} = \frac{1}{2}$.

Note that $\partial_D( \mathcal{A})$ changes if we change $D$, so $\mathcal{A} \not \in \Br_0(X)$. The last statement follows by Hensel's lemma since $\alpha^*\{ Y_k = -b^2 Y_\ell\} = \{X_k^2 = - b^2 X_\ell^2\}$ has $\F_p$-points if and only if $p \equiv 3 \pmod{4}$.
\end{proof}

We are now interested in how the invariant changes as the surface changes. Since there is no uniform formula for $\mathcal{A}$ as shown in Corollary \ref{No uniform formula} we will only be able to understand this for certain limited changes. Given $\mathbf{u} = (u_0, u_1, u_2, u_3) \in (\Q^{\times})^4$ we let $\mathbf{a} \mathbf{u}^2 =(a_k u_k^2)$. Choose a representation $\mathcal{A} = (\frac{\alpha^*f}{X_3^2}, \theta)$ with $f = \sum_{i = 0}^3 y_i a_i Y_i $ tangent to the point $[y_0: y_1: y_2 : y_3] \in Y_{\mathbf{a}}(\Q).$ We then get a point $[\frac{y_0}{u_0}:\frac{y_1}{u_1}:\frac{y_2}{u_2}:\frac{y_3}{u_3}] \in Y_{\mathbf{a} \mathbf{u}^2}(\Q)$ with a tangent plane defined by $f_{\mathbf{u}} = \sum_{i = 0}^3 y_i u_i a_i Y_i^2$. This defines an element $\mathcal{A}_{\mathbf{u}} = (\frac{\alpha^* f_{\mathbf{u}}}{X_3^2}) \in \Br(X_{\mathbf{a}\mathbf{u}^2})$. 

Note that if $\mathcal{A}$ is $p$-normalized and $u_i \in \Z_p^{\times}$ then $\mathcal{A}_{\mathbf{u}}$ will also be $p$-normalized. From this we can prove the following proposition.
\begin{proposition}
Let $\{k, \ell, m, n\} = \{0,1,2,3\}$. Let $p$ be an odd prime such that $v_p(a_k),v_p(a_\ell)$ have the same parity and $v_p(a_m), v_p(a_n)$ have the other parity. 

If $\theta_{\mathbf{a}}\in \Q_p^{\times 2} $ then $\emph{inv}(\mathcal{A}(\cdot)) = 0$. Otherwise $\mathcal{A} \not \in \Br_0(X)$.

If $\theta_{\mathbf{a}}\not \in \Q_p^{\times 2}$ and $p \equiv 1 \pmod{4}$ then $\mathcal{A}$ is prolific so induces no obstruction to the Brauer-Manin obstruction.

If $\theta_{\mathbf{a}}\not \in \Q_p^{\times 2}$, $p \equiv 3 \pmod{4}$ and $-\frac{a_k}{a_\ell} \in \Q_p^{\times 4}$ then $\emph{inv}(\mathcal{A}(\cdot))$ is constant and for all $\mathbf{u} \in (\Z_p)^{\times 4}$ we have 
\begin{equation*}
    \emph{inv}(\mathcal{A_{\mathbf{u}}}(\cdot)) = \emph{inv}(\mathcal{A}(\cdot)) + \frac{\left(\frac{u_k u_\ell}{p} \right) - 1}{4}.
\end{equation*}
\label{Prop if p divides 2 coefficients}
\end{proposition}
Recall from Lemma~\ref{Local solubility} that since $X_{\mathbf{a}}(\Q_p) \neq \emptyset$ either $-\frac{a_k}{a_\ell} \in \Q_p^{\times 4}$ or $-\frac{a_m}{a_n} \in \Q_p^{\times 4}$, the second case can also be handled via this proposition via permuting the indices.
\begin{proof}
If $\theta_{\mathbf{a}}\in \Q_p^2$ then $\mathcal{A} = 1$ as an element of $\Br(X_{\Q_p})$ so $\text{inv}_p(\mathcal{A}(\cdot))$ will be constant.

By taking an equivalent surface we may assume that $p \nmid a_k a_\ell$, $v_p(a_m), v_p(a_n) \in \{1, 3\}$ and $- \frac{a_k}{ a_\ell} \in \Q_p^{\times 4}$. We may then also assume that $\mathcal{A}$ is $p$-normalized. The only thing to show after applying Lemma~\ref{Computation invariant if p divides 2 coefficients} is the formula for $\text{inv}(\mathcal{A_{\mathbf{u}}}(\cdot))$. As in the lemma there exists a $t$ such that $\alpha^* \Tilde{f} = X_k^2 - (-1)^{t} X_\ell^2 $. Let $\mathcal{X}_{\mathbf{a} \mathbf{u}^2}$ be the obvious integral model. By definition the reduction of $\alpha^* \Tilde{f}$ is equal to $\alpha^* \Tilde{f}_\mathbf{u} = u_k X_k^2 - (-1)^t u_\ell X_\ell^2.$ The desired statement follows after applying Lemma~\ref{Computation invariant if p divides 2 coefficients} to $X_{\mathbf{a} \mathbf{u}^2}.$
\end{proof}

If there exists $v_k \in \Q_p$ such that $u_k = v_k^2$ for all $k$ then there exists a bijection $X_{\mathbf{a}}(\Q_p) \to X_{\mathbf{a} \mathbf{u}^2}(\Q_p): \mathbf{x} \to \mathbf{x} \mathbf{v}^{-1} = [x_0 v_0^{-1}: x_1 v_1^{-1}: x_2 v_2^{-1}: x_3 v_3^{-1}]$. It follows directly from the definitions that \begin{equation}
    \text{inv}(\mathcal{A_{\mathbf{u}}}(\mathbf{x} \mathbf{v}^{-1})) = \text{inv}(\mathcal{A}(\mathbf{x})).
    \label{Invariant stays the same for bad primes}
\end{equation}

A fact we will need is 
\begin{lemma}
If $a_k \in \Z_p^{\times}$ and $u_k \in \Z_p$ for all $k$ and if $u_k$ is a unit for all but one $k$ then $\mathcal{A}$ is $p$-normalized if and only if $\mathcal{A}_{\mathbf{u}}$ is $p$-normalized.
\label{A stays p-normalized}
\end{lemma}
\begin{proof}
Assume that $u_0$ is the non-unit. Let $[y_0: y_1: y_2 : y_3] \in Y_{\mathbf{a}}(\Q_p)$ be such that $f = \sum_{i=0}^3 y_i a_i Y_i$ and thus $f_{\mathbf{u}} = \sum_{i=0}^3 y_i u_i a_i Y_i$. If $\mathcal{A}$ is $p$-normalized and $\mathcal{A}_{\mathbf{u}}$ is not then $v_p(y_0) = 0$ but $v_p(y_i) \geq 1$ for $i \neq 1$. On the other hand, if $\mathcal{A}_{\mathbf{u}}$ is $p$-normalized but $\mathcal{A}$ is not then $v_p(y_0) = - v_p(u_0) \leq -1$ but $v_p(y_i) \geq 1$ for $i  \neq 1$. Both are impossible because $\sum_{i = 0}^3 a_i y_i^2 = 0$.
\end{proof}

We will require the following example
\begin{lemma}
The surface $X:X_0^4 - 16X_1^4 + 7X_2^4 + 7X_3^4 = 0$ has local points everywhere and $\mathcal{A}$ is locally nowhere prolific.
\label{Example which is locally not prolific with theta = -1}
\end{lemma}
\begin{proof}
It has local solutions because it has a rational point $[2:1:0:0] \in X(\Q)$.

The quadric $Y$ has a rational point $[4:1:0:0]$ so we have a normalized $\mathcal{A} = (1 + 4\frac{X_1^2}{X_0^2}, -1)$. The left-hand side is always positive so $\text{inv}_{\infty}(\mathcal{A}(\cdot)) = 0$. By Lemma~\ref{Lemma if p divides at most one coefficient} $\text{inv}_{p}(\mathcal{A}(\cdot)) = 0$ for $p \neq 2,7$. Because $7 \equiv 3 \pmod{4}$ Lemma~\ref{Computation invariant if p divides 2 coefficients} implies that $\text{inv}_{7}(\mathcal{A}(\cdot))$ is constant. Lastly, by looking at primitive solutions $\mathbf{x} = [x_0: x_1: x_2: x_3] \in X(\Q_2)$ modulo $4$ we see that $x_0 \in \Z_2^{\times}$ so we have $1 + 4\frac{x_1^2}{x_0^2} \equiv 1 \pmod{4}$. The invariant $\text{inv}_{2}(\mathcal{A}(\mathbf{x}))$ is thus always equal to $(1 + 4\frac{x_1^2}{x_0^2},-1)_2 = (1, -1)_2 = 0$.
\end{proof}
\label{Section local computations}
\section{Analytic preliminaries}
We will want to apply a generalization of the method of \cite{heathbrown1993selmer} to multiple different sums, to do this we will prove a general theorem describing this generalization.
\subsection{Frobenian multiplicative functions}
We start by slightly generalizing the notion of frobenian multiplicative functions \cite[Definition~2.1]{loughran2019frobenian} to number fields. These are closesly related to the \emph{frobenian functions} of Serre \cite[\S 3.3]{serre2012lectures}.
\begin{definition}
Let $K$ be a number field and $S$ a finite set of primes of $K$. A multiplicative function $\rho: I_K \to \C$ is a $S$-\emph{frobenian multiplicative function} if it satisfies the following properties.
\begin{enumerate}
    \item There exists a $H \in \N$ such that $|\rho(\mathfrak{p}^k)| \leq H^k$ for all primes $\mathfrak{p}$ and $k \in \N$.
    \item For all $\varepsilon > 0$ there exist a constant $C_\varepsilon$ such that $|\rho(\mathfrak{n})| \leq C_{\varepsilon}\Norm(\mathfrak{n})^{\varepsilon}$
    \item The reduction of $\rho$ to the totally split primes is $S$-\text{frobenian}, i.e. there exists a Galois extension $L/K$ with Galois group $\Gamma$ such that $S$ contains all the primes which ramify in this extension and a class function $\varphi: \Gamma \to \C$ such that for all totally split primes $\mathfrak{p} \not \in S$
    \begin{equation*}
        \rho(\mathfrak{p}) = \varphi(\text{Frob}_{\mathfrak{p}}).
    \end{equation*}
\end{enumerate}
The \emph{mean} of a frobenian multiplicative function is the mean of $\varphi$, i.e.
\begin{equation*}
    m(\rho) = \frac{1}{|\Gamma|} \sum_{\gamma \in \Gamma} \varphi(\gamma).
\end{equation*}
\end{definition}
\begin{remark}
We will only require that the restriction to the set of totally split primes is frobenian since the non-totally split primes end up only affecting the constant in the asymptotics.

The mean is independent of $L$ since this is true for class functions.

Hecke characters with finite modulus are frobenian multiplicative by class field theory.
\end{remark}
For any two functions $f,g: I_K \to \C$ let $f * g$ be their Dirichlet convolution.
\begin{lemma}
If $\rho_1, \rho_2$ are $S$-frobenian multiplicative functions defined by Galois extensions $L_1, L_2$ then their product $\rho_1 \rho_2$ and their convolution $\rho_1 * \rho_2$ are $S$-frobenian multiplicative function defined by the Galois extension $L_1 L_2$. The implicit constants of $\rho_1 \rho_2$ and $\rho_1 * \rho_2$ are bounded by those of $\rho_1, \rho_2$. The mean of $\rho_1 * \rho_2$ is $m(\rho_1 * \rho_2) = m(\rho_1) + m(\rho_2)$.
\label{Products and convolutions of frobenian multiplicative functions are frobenian multiplicative}
\end{lemma}
\begin{proof}
We first show the required bounds, let $H_1, C_{1, \varepsilon}$, respectively $H_2, C_{2, \varepsilon}$ be the required constants for $\rho_1$, respectively for $\rho_2$. Then by the divisor bound
\begin{align*}
    &|\rho_1 \rho_2 (\mathfrak{p}^k)| = |\rho_1(\mathfrak{p}^{k})| |\rho_2(\mathfrak{p}^{k})| \leq (H_1 H_2)^k \\
    &|\rho_1 \rho_2 (\mathfrak{n})| = |\rho_1(\mathfrak{n})| |\rho_2(\mathfrak{n})| \leq C_{1, \varepsilon} C_{2, \varepsilon} \Norm(\mathfrak{n})^{2 \varepsilon} \\
    &|\rho_1 * \rho_2(\mathfrak{p}^k)|  \leq \sum_{i = 0}^k |\rho_1(\mathfrak{p}^{i})||\rho_2(\mathfrak{p}^{k-i})|  \leq \sum_{i = 0}^k H_1^{i} H_2^{k - i} \leq (H_1 + H_2)^k \\
    &|\rho_1 * \rho_2(\mathfrak{n})| \leq \sum_{\mathfrak{d}| \mathfrak{n}} |\rho_1(\mathfrak{d})||\rho_2(\frac{\mathfrak{n}}{\mathfrak{d}})| \leq \tau(\mathfrak{n}) C_{1, \varepsilon} C_{2, \varepsilon} \Norm(\mathfrak{d})^{\varepsilon} \Norm(\frac{\mathfrak{n}}{\mathfrak{d}})^{\varepsilon} \ll_{\varepsilon} C_{1, \varepsilon} C_{2, \varepsilon} \Norm(\mathfrak{n})^{2 \varepsilon}.
\end{align*}

If $\mathfrak{p}$ is totally split then $\rho_1 \rho_2(\mathfrak{p}) = \rho_1(\mathfrak{p}) \rho_2(\mathfrak{p}), \rho_1 * \rho_2(\mathfrak{p}) = \rho_1(\mathfrak{p}) + \rho_2(\mathfrak{p})$. The class function of $\rho_1 \rho_2,\rho_1 * \rho_2$ is thus the product, respectively the sum of the class functions of $\rho_1$ and $\rho_2$.
\end{proof}

We will require the following lemma.
\begin{lemma}
Let $\rho$ be a frobenian multiplicative function defined by a Galois extension $L/K$ and $\chi$ a Hecke character with finite modulus of $K$, then $m(\chi \rho) \neq 0$ only if $\chi$ corresponds via class field theory to a character of $\emph{Gal}(L/K)$.
\label{Mean after twisting by character}
\end{lemma}
\begin{proof}
For $K = \Q$ this is the first part of the proof of \cite[Lemma~2.4]{loughran2019frobenian}. The argument generalizes to general number fields. To spell out the details, let $L'/L$ be an extension such that $L'/K$ is Galois and $\chi$ corresponds via class field theory to a character of $\Gamma := \text{Gal}(L'/K)$, which we will also denote $\chi$. Let $\varphi: \Gamma \to \C$ be the class function defining $\rho$. The subgroup $N := \text{Gal}(L'/ L) \subset \Gamma$ is normal and $\varphi$ is invariant under translation by $N$. Then
\begin{equation*}
    |\Gamma|m(\chi \rho) = \sum_{g \in \Gamma} \chi(g) \varphi(g) =  \sum_{gN \in \Gamma/N} \varphi(gN) \sum_{h \in gN} \chi(h) = \sum_{gN \in \Gamma/N} \varphi(gN) \sum_{h' \in N} \chi(gh').
\end{equation*}
The sum $\sum_{h' \in N} \chi(gh')$ is $0$ unless $\chi$ is trivial on $N$. In that case $\chi$ is a character of $\Gamma/N = \text{Gal}(L/K)$ as desired.
\end{proof}
It will be easier to deal with frobenian multiplicative functions over $\Q$.
\begin{definition}
For any finite field extension $K'/K$ and any frobenian multiplicative function $\rho$ over $K'$ we define the \emph{induced frobenian multiplicative function} $\text{Ind}^{K}_{K'}(\rho)$ for all $\mathfrak{n} \in I_K$ by 
\begin{equation*}
    \text{Ind}^{K}_{K'}(\rho)(\mathfrak{n}) := \sum_{\substack{\mathfrak{N} \in I_{K'} \\ \Norm_{K'/K}(\mathfrak{N}) = \mathfrak{n}}} \rho(\mathfrak{N}).
\end{equation*}
\end{definition}

It follows directly from the definition that
\begin{equation*}
    \sum_{\Norm_{K'/\Q}(\mathfrak{N}) \leq x} \rho(\mathfrak{N}) =  \sum_{\Norm_{K/\Q}(\mathfrak{n}) \leq x} \text{Ind}^{K}_{K'}(\rho)(\mathfrak{n}).
\end{equation*}
We can thus reduce the study of such sums for general frobenian multiplicative functions to those over $\Q$. Let us show that these are indeed frobenian multiplicative.

\begin{lemma}
If $\rho$ is a $S$-frobenian multiplicative function over $K'$ then $\emph{Ind}^{K}_{K'}(\rho)$ is a $S'$-frobenian multiplicative function of the same mean with 
\begin{equation*}
    S' := \{\mathfrak{p} : \mathfrak{p} \emph{ ramifies in } K' \emph{ or lies under an element of } S \}.
\end{equation*} The required constants of $\emph{Ind}^{K}_{K'}(\rho)$ are bounded in terms of those of $\rho$ and $[K' : K]$.
\begin{proof}
The multiplicativity is clear. To show that it satisfies the desired inequalities we note that $\sum_{\Norm_{K'/K}(\mathfrak{N}) = \mathfrak{n}} 1$ is multiplicative, and is smaller than $\tau_{[K':K]}(\mathfrak{n})$ as can be seen by looking at prime powers. By the divisor bound we get
\begin{align}
    |\text{Ind}^{K}_{K'}(\rho)(\mathfrak{p}^k)| \leq \sum_{\substack{\mathfrak{N} \in I_{K'} \\ \Norm_{K'/K}(\mathfrak{N}) = \mathfrak{p}^k}} |\rho(\mathfrak{N})| \leq \tau_{[K':K]}(\mathfrak{p}^k) H^k \leq [K':K]^k H^k \label{H for induced frobenian function}\\
    |\text{Ind}^{K}_{K'}(\rho)(\mathfrak{n})| = \sum_{\substack{\mathfrak{N} \in I_{K'} \\ \Norm_{K'/K}(\mathfrak{N}) = \mathfrak{n}}} |\rho(\mathfrak{N})| \leq \tau_{[K':K]}(\mathfrak{n}) C_{\varepsilon} \Norm(\mathfrak{n})^{\varepsilon} \ll_{\varepsilon, [K': K]} C_{\varepsilon} \Norm(\mathfrak{n})^{2 \varepsilon}.
    \label{C epsilon for induced frobenian function}
\end{align}

It remains to show that the value of $\text{Ind}^{K}_{K'}(\rho)$ at the totally split primes is determined by a frobenian function of the same mean. By enlarging $L/K'$ we may assume that $L/K$ is Galois. Let $G := \text{Gal}(L/K)$, $H := \text{Gal}(L/K')$. Assume that $\rho$ has an associated class function $\varphi$. For any totally split prime ideal $\mathfrak{p} \not \in S'$ of $K$ we have
\begin{equation*}
    \text{Ind}^{K}_{K'}(\rho)(\mathfrak{p}) = \sum_{\Norm_{K'/K}(\mathfrak{q}) = \mathfrak{p}} \rho(\mathfrak{q}) = \sum_{\Norm(\mathfrak{q}) =  \Norm(\mathfrak{p})} \varphi(\text{Frob}_{\mathfrak{q}}).
\end{equation*}

Consider the set of prime ideals $\mathfrak{P} \in I_L$ lying over $\mathfrak{p}$. This set has a transitive $G$-action. Let us show that the prime ideal $\mathfrak{q} = \mathfrak{P} \cap \mathcal{O}_{K'}$ has the property $\Norm(\mathfrak{q}) =  \Norm(\mathfrak{p})$ if and only if $\text{Frob}_{\mathfrak{P}/\mathfrak{p}} \in H$. If $\Norm(\mathfrak{q}) =  \Norm(\mathfrak{p})$ then $\text{Frob}_{\mathfrak{P}/\mathfrak{p}} = \text{Frob}_{\mathfrak{P}/\mathfrak{q}} \in H$. On the other hand if $\text{Frob}_{\mathfrak{P}/\mathfrak{p}} \in H$ then $\text{Frob}_{\mathfrak{P}/\mathfrak{p}}$ keeps $K'$ invariant so fixes $\F_{\mathfrak{q}}$, hence $\Norm_{K'/K}(\mathfrak{q}) = \mathfrak{p}$.

Let $f := [\F_{\mathfrak{P}}: \F_{\mathfrak{p}}] = [\F_{\mathfrak{P}}: \F_{\mathfrak{q}}]$. This is independent of $\mathfrak{P}$ because $L/K$ is Galois. The set of prime ideals $\mathfrak{P} \in I_L$ lying over $\mathfrak{p}$, respectively $\mathfrak{q}$ has size $|G|/f$, respectively $|H|/f$. Thus
\begin{equation*}
    \sum_{\Norm(\mathfrak{q}) =  \Norm(\mathfrak{p})} \varphi(\text{Frob}_{\mathfrak{q}}) = \frac{f}{|H|} \sum_{\mathfrak{P}, \text{Frob}_{\mathfrak{P}/\mathfrak{p}} \in H} \varphi(\text{Frob}_{\mathfrak{P}/\mathfrak{q}}).
\end{equation*}
The $ \text{Frob}_{\mathfrak{P}/\mathfrak{p}}$ are equidistributed over the conjugacy class $C_{\mathfrak{p}}$ of $\text{Frob}_{\mathfrak{p}}$, i.e. there are $|G|/f |C_{\mathfrak{p}}|$ primes $\mathfrak{P}$ with the same $\text{Frob}_{\mathfrak{P}/\mathfrak{p}}$. This implies that
\begin{equation*}
    \text{Ind}^{K}_{K'}(\rho)(\mathfrak{p}) = \frac{|G|}{|H| |C_{\mathfrak{p}}|}\sum_{h \in H \cap C_{\mathfrak{p}}} \varphi(h) = \text{Ind}^G_H(\varphi)( \text{Frob}_{\mathfrak{p}}).
\end{equation*}
Where $\text{Ind}^G_H(\varphi)$ is the induced class function and the last equality is by definition. This class function has the same mean as $\varphi$.
\end{proof}
\end{lemma}
    That the class function of the induced frobenian multiplicative function is the induction of the class function of the original frobenian multiplicative function justifies the term induced.
\begin{definition}
The $L$-function of a frobenian multiplicative function $\rho$ for $\text{Re}(s) > 1$ is given by the Dirichlet series
\begin{equation*}
    L(\rho, s) := \sum_{\mathfrak{n} \in I_K} \frac{\rho(\mathfrak{n})}{\Norm(\mathfrak{n})^s} = \prod_\mathfrak{p} \left(1 + \frac{\rho(\mathfrak{p})}{\Norm(\mathfrak{p})^s} + \frac{\rho(\mathfrak{p}^2)}{\Norm(\mathfrak{p})^{2s}}  +  \cdots\right).
\end{equation*}
This converges absolutely in this region because $|\rho(\mathfrak{n})| \leq C_{\varepsilon}\Norm(\mathfrak{n})^{\varepsilon}$ for all $\varepsilon > 0$.

The normalized $L$-function is defined as $L_n(\rho, s) = L(\rho, s)\zeta(s)^{-m(\rho)}$.
\end{definition}
We remark that the $L$-function of the induction of a frobenian multiplicative function is equal to the $L$-function of the original frobenian multiplicative function. 

The normalized $L$-function can be analytically continued slightly to the left of the line $\text{Re}(s) > 1$ as in \cite[Proposition 2.7]{loughran2019frobenian}. By controlling how far to the left it can be extended we will be able to derive a Selberg-Delange type asymptotic formula for $\sum_{\Norm(\mathfrak{n}) \leq x} \rho(\mathfrak{n})$.
To control this we will need the notion of the conductor of a frobenian multiplicative function. 
\begin{definition}
The class function of a frobenian multiplicative function can be written as a sum of characters of $\Gamma$. We define the \emph{conductor} $q(\rho)$ of a frobenian multiplicative function as the maximum taken over all characters $\chi$ whose coefficient is non-zero of the conductor of the Artin $L$-series $L(\chi, s)$. This is $\Norm_{K}(\mathfrak{f}_\chi) d_K^{\chi(1)}$, where $\mathfrak{f}_\chi$ is the Artin conductor of $\chi$ and $d_K$ is the discriminant of $K$.
\end{definition}
The conductor of a frobenian multiplicative function is bounded by the discriminant of $L$. We have the inequality $q(\text{Ind}_{K'}^K(\rho)) \leq q(\rho)$ because the Artin $L$-series of an induced character is equal to the $L$-series of the original character.

We start with the mean $0$ case.
\begin{lemma}
Let $A > 0$ and $\rho$ a frobenian multiplicative function of mean $0$ and conductor $q := q(\rho) \leq (\log x)^{A}$. For all $x \geq 2$ there exists a $c > 0$ such that,
\begin{equation*} 
\sum_{\Norm(\mathfrak{n}) \leq x} \rho(\mathfrak{n}) \ll e^{O(|S|)} x e^{-c\sqrt{\log x}}
\label{Formula for frobenian function of mean zero with small conductor}
\end{equation*}
The implicit constants depend on $n := [K: \Q], \Gamma, C_{\epsilon},H, A$, but for the rest they are are independent of $\rho$.
\label{Lemma frobenian function of mean zero with small conductor}
\end{lemma}
\begin{proof}
Since $\varphi$ is a class function we can write it as a sum over the irreducible characters of $\Gamma.$
\begin{equation*}
    \varphi = \sum_{\chi} \lambda_\chi \chi.
\end{equation*}
We have the trivial inequality $|\lambda_\chi| \leq \chi(1)H = O_{\Gamma, H}(1)$.

This leads to the equality 
\begin{equation*}
    L(\rho, s) = G(s)\prod_{\chi} L(\chi, s)^{\lambda_\chi}
\end{equation*}
where $L(\chi, s)$ is the Artin $L$-function of $\chi$ and
\begin{equation*}
    G(s) = \prod_{\mathfrak{p}} \left(1 + \frac{\rho(\mathfrak{p})}{\Norm(\mathfrak{p})^s} + \frac{\rho(\mathfrak{p}^2)}{\Norm(\mathfrak{p})^{2s}} + \cdots\right) \prod_{\chi}\text{det}\left(1- \frac{\pi_{\chi}(\text{Frob}_{\mathfrak{p}})}{ \Norm(\mathfrak{p})^s}\right)^{\lambda_\chi}.
\end{equation*}
Here $\pi_\chi$ is the representation corresponding to $\chi$.

Let us consider the local factors of $G(s)$, there are three cases. We use the standard notation $s = \sigma + it$. 
\begin{enumerate}
    \item If $\mathfrak{p} \in S$ then the contribution of the factors $\text{det}(1- \frac{\pi_{\chi}(\text{Frob}_{\mathfrak{p}})}{ \Norm(\mathfrak{p})^s})$ is $e^{O_{n, \Gamma, H}(|S|)}$. For the other factors we have a bound for any $\varepsilon > 0$ of
    \begin{equation*}
        1 + \frac{|\rho(\mathfrak{p})|}{\Norm(\mathfrak{p})^{\sigma}} + \cdots \leq 1 + C_{\varepsilon}  \frac{\Norm(\mathfrak{p})^{\varepsilon}}{\Norm(\mathfrak{p})^{\sigma}} + \cdots = 1 + C_{\varepsilon} \Norm(\mathfrak{p})^{\varepsilon - \sigma}( 1 - \Norm(\mathfrak{p})^{\varepsilon - \sigma})^{-1}.
    \end{equation*}
    
    The total contribution of these parts in the vertical strip $\sigma > \frac{3}{4}$ is thus $e^{ O_{n,\Gamma, H, C_{\varepsilon}}(|S|)}$.
    \item The contribution of the primes $\mathfrak{p}$ which are not totally split is absolutely convergent for $\sigma > \frac{1}{2}$ and in the vertical strip $\sigma > \frac{3}{4}$ is $O_{n, \Gamma, H, C_{\varepsilon}}(1)$.
    \item For the other primes $\mathfrak{p}$ we have by definition $\rho(\mathfrak{p}) = \sum_\chi \lambda_\chi \text{Tr}(\pi_{\chi}(\text{Frob}_{\mathfrak{p}}))$ so the associated factor is 
    \begin{equation*}
        1 + O_{\Gamma, H}\left(\frac{|\rho(\mathfrak{p})| + |\rho(\mathfrak{p}^2)|}{\Norm(\mathfrak{p})^{2 \sigma}} + \frac{ |\rho(\mathfrak{p})| + |\rho(\mathfrak{p}^2)| + |\rho(\mathfrak{p}^3)|}{\Norm(\mathfrak{p}^{3\sigma})} + \cdots \right).
    \end{equation*}
    For all $\sigma > \varepsilon > 0$ this is bounded by 
    \begin{equation*}
        1 + C_{\varepsilon} O_{\Gamma, H}\left(\frac{\Norm(\mathfrak{p})^{2 \varepsilon}}{\Norm(\mathfrak{p})^{2 \sigma}} + \cdots \right) = 1 + C_{\varepsilon}\Norm(\mathfrak{p})^{2\varepsilon - 2\sigma} O_{\Gamma}\left((1 - \Norm(\mathfrak{p})^{ \epsilon - \sigma})^{-1}\right).
    \end{equation*}
    The total contribution of all of these primes is thus absolutely convergent for $\sigma > \frac{1}{2}$ and is $O_{n, \Gamma, H,C_{\varepsilon}}( 1)$ in the vertical strip $\sigma > \frac{3}{4}$.
\end{enumerate}

Combining everything we see that $G(s) = e^{O_{n, \Gamma, H, C_{\varepsilon}}(|S|)}$ in the vertical strip $\sigma > \frac{3}{4}$.

By the Brauer-induction theorem \cite[Theorem 18]{Serre1977Linear} we may assume that the $\chi$ are Hecke characters such that the $L$-series $L(\chi, s)$ has conductor at most $q$ at the cost of changing $\lambda_\chi$. But the $\lambda_\chi$ will still be $O_{\Gamma, H}(1)$.

Let $T = e^{\sqrt{\log x}}$. Choose $1 > 1 - \delta \geq \frac{3}{4}$ such that $s$ has at least a distance $\delta$ from every zero of $L(\chi, s)$ for all $s$ in the region
\begin{equation}
    \sigma \geq 1 -\delta, |t| \leq T.
    \label{Zero-free region}
\end{equation} 
We can choose $\delta$ such that $\delta \gg_A q^{-\frac{1}{2A}} \geq 1/\log T = 1/\sqrt{\log x}$ by the standard zero-free region of Hadamard and de la Vall\'ee Poussin and Siegel's theorem for Hecke characters \cite{fogel1963siegel}. We have the following inequality in the region \eqref{Zero-free region}
\begin{equation}
    \left| \log L(\chi, s) \right| \ll_A \log \log  q t \ll_A \log \log T.
    \label{Bound for logarithmic derivative of L}
\end{equation}
The case of Dirichlet characters is \cite[Theorem 11.4]{montgomery2007multiplicative}. The proof for general Hecke characters is the same.

Exponentiating we get
\begin{equation*}
    |L(\chi, s)^{\lambda_\chi}| = (\log T )^{O_{A, \Gamma, H}(1)} = (\log x )^{{O_{A, \Gamma, H}(1)}}.
\end{equation*}
And thus
\begin{equation}
    |F(s)| = e^{O_{n, \Gamma, H, C_{\epsilon}}(|S|)} (\log x)^{O_{A, \Gamma, H}(1)}.
\label{Bounds for F}
\end{equation}
We now apply Perron's formula \cite[Theorem II.3.3]{tenenbaum1995introduction} for $\kappa = 1 + \frac{1}{\log x}$
\begin{equation*}
    \sum_{\Norm(\mathfrak{n}) \leq x} \rho(\mathfrak{n}) \log \left(\frac{x}{\Norm(\mathfrak{n})}\right) = \frac{1}{2 \pi i}\int_{\kappa - i \infty}^{\kappa + i\infty} F(s) x^s \frac{ds}{s^2}.
\end{equation*}

The function $F(s)$ has an analytic continuation to the region \eqref{Zero-free region}. We can thus change the contour to be the straight lines between $\kappa - i \infty, \kappa - iT, 1 - \frac{\delta}{\log(1 + T)} - iT, 1 - \frac{\delta}{\log(1 + T)} + iT, \kappa + iT, \kappa + i\infty.$ The integrals over everything except the line between $1 - \frac{\delta}{\log(1 + T)} - iT, 1 - \frac{\delta}{\log(1 + T)} + iT$ are \begin{equation*}
    \ll e^{O_{n, \Gamma, C_{\epsilon}}(|S|)} (\log x)^{O_{A, \Gamma, H}(1)} \frac{x^{\kappa}}{T} \ll_{A, \Gamma, H} e^{O_{n,\Gamma, C_{\epsilon}}(|S|)} x e^{- \frac{1}{2} \sqrt{ \log x}}.
\end{equation*}
The contribution of this last line segment is
\begin{equation*}
    \ll e^{O_{n,\Gamma, H, C_{\epsilon}}(|S|)} (\log x)^{O_{A, H}(1)} \frac{x^{1 - \delta}}{\delta} \ll_{A, \Gamma, H} e^{O_{n,\Gamma, H, C_{\epsilon}}(|S|)} x e^{- \frac{1}{2} \sqrt{ \log x}}.
\end{equation*}

The equation \eqref{Formula for frobenian function of mean zero with small conductor} follows from partial summation.
\end{proof}

To deal with the general case we will require a uniform upper bound for absolute values of frobenian multiplicative functions.
\begin{lemma}
    Let $\rho$ be a frobenian multiplicative function with implicit constants $C_{\varepsilon}$, then
    \begin{equation*}
        \sum_{\Norm(\mathfrak{n}) \leq x} |\rho(\mathfrak{n})| \ll_{[K : \Q]} C_{\frac{1}{6}} x (\log x)^{[K : \Q] H}.
    \end{equation*}
\label{Lemma for uniform absolute upper bound for frobenian functions}
\end{lemma}
\begin{proof}
Consider the induced frobenian multiplicative function $\text{Ind}^{\Q}_{K}(\rho)$. We may take its $H$ constant to be $H [K :\Q]$ by \eqref{H for induced frobenian function} and its $C_{\frac{1}{3}}$ constant may be taken to be $\ll C_{\frac{1}{6}}$. The proof thus reduces to proving that if $\rho$ is a frobenian multiplicative function over $\Q$ then 
\begin{equation*}
    \sum_{n \leq x} |\rho(n)| \ll C_{1/3} x (\log x)^{H}.
\end{equation*}

Write $n = m k^2$ with $m$ square-free. We claim that $|\rho(n)| \leq C_{1/3} \tau_H(m) k$. To prove this we reduce to the case when $n$ is a prime power. If $n = p$ then $|\rho(p)| \leq H = \tau_H(p)$. If $n = p^i$ for $i > 1$ then $|\rho(p^i)| \leq C_{1/3} p^{i/3} \leq p^{\lfloor i \rfloor/2} = \sqrt{k}$.

We thus get an upper bound
\begin{equation*}
    \leq \sum_{m k^2 \leq x} |\rho(m k^2)| \leq C_{1/3} \sum_{m k^2 \leq x} \tau_H(m) k \leq C_{1/3 } x \sum_{m \leq x} \frac{\tau_H(m)}{m} \leq C_{1/3} x (\log x)^{H}.
\end{equation*}
\end{proof}
\begin{corollary}
If $\rho$ is a frobenian multiplicative function of mean $0$ then for all $C, \varepsilon > 0$ and $x \geq 2$
\begin{equation*} 
\sum_{\Norm(\mathfrak{n}) \leq x} \rho(\mathfrak{n}) \ll_{\varepsilon, C, [K: \Q], H} q(\rho)^{\varepsilon} e^{|S|}x (\log x)^{-C}.
\end{equation*}
\label{Corollary of bound for frobenian function of mean 0 invariant in the conductor}
\end{corollary}
\begin{proof}
If $q(\rho) \leq (\log x)^{([K :\Q] H + C)\varepsilon^{-1}}$ then this follows from Lemma~\ref{Lemma frobenian function of mean zero with small conductor}. Otherwise it follows from Lemma~\ref{Lemma for uniform absolute upper bound for frobenian functions}.
\end{proof}

We can now prove an asymptotic formula for frobenian multiplicative functions of arbitrary means. To state the result we will need the complex-valued functions $\gamma_j(z)$ from \cite[II.5]{tenenbaum1995introduction}. We will then define 
\begin{equation}
    \lambda_k(\rho) := \frac{1}{\Gamma(m(\rho)-k)} \sum_{h + j = k} \frac{1}{h! j!} L_n^{(h)}(\rho,1) \gamma_j(m(\rho)).
\end{equation}
In particular $\lambda_0(\rho) = L_n(\rho,1)$.
\begin{theorem}
Let $\rho$ a frobenian multiplicative function of conductor $q := q(\rho)$ and $N \in \Z_{\geq -1}$. For all $x \geq 2$
\begin{equation} 
\sum_{\Norm(\mathfrak{n}) \leq x} \rho(\mathfrak{n}) = x (\log x)^{m(\rho) - 1}\left( \sum_{k = 0}^N \frac{\lambda_k(\rho)}{(\log x)^k} + e^{O(|S|)}O_{\varepsilon}\left(q^{\varepsilon}(\log x)^{-N- 1}\right)\right). \label{Formula for frobenian function with small conductor}
\end{equation}
The implicit constants depend on $N, [K: \Q], \Gamma, C_{\epsilon},H, m(\rho)$, but for the rest are independent of $\rho$.
\label{Propostion for frobenian function with small conductor}
\end{theorem}
\begin{proof}
By considering the induced frobenian multiplicative function we may assume that $K = \Q$, note that $|S|$ may increase at most by the amount of prime divisors $\omega(q)$ of the conductor $q$, but by the divisor bound $e^{O(\omega(q))} \leq \tau_{O(1)}(q) \ll_{\varepsilon} q^{\varepsilon}$. 

The decomposition of $L$-series $L(\rho,s) = L_n(\rho, s) \zeta(s)^{m(\rho)}$ corresponds to writing $\rho$ as a convolution $\rho_n * \tau_{m(\rho)}$. Here $\rho_n = \rho * \tau_{-m(\rho)}$.
The function $\tau_{-m(\rho)}$ is frobenian multipliciative of mean $-m(\rho)$ by the divisor bound, so $\rho_n$ is frobenian multiplicative of mean $0$ by Lemma \ref{Products and convolutions of frobenian multiplicative functions are frobenian multiplicative}.

By the hyperbola method we have 
\begin{equation}
    \begin{split}
        \sum_{n \leq x} \rho(n) = &\sum_{n \leq \sqrt{x}} \rho_n(n) \sum_{m \leq \frac{x}{n}} \tau_{m(\rho)}(m)  + \sum_{m \leq \sqrt{x}} \tau_{m(\rho)}(m)   \sum_{n \leq \frac{x}{m}} \rho_n(n) \\ & -\sum_{n \leq \sqrt{x}} \rho_n(n) \sum_{m \leq \sqrt{x}} \tau_{m(\rho)}(m).
    \end{split}
    \label{Application of hyperbola method}
\end{equation}
It follows from the Selberg-Delange method \cite{tenenbaum1995introduction} that for any $M \in \N, y \geq 2$
\begin{equation}
    \begin{split}
    \sum_{m \leq y} \tau_{m(\rho)}(m) =& y (\log y)^{m(\rho) - 1}\Big( \sum_{j = 0}^M \frac{\gamma_j(m(\rho))}{\Gamma(m(\rho) - j) j! (\log y)^j}\Big)  \\ &+ O_M((\log y)^{m(\rho) -M - 2}).
    \end{split}2
\label{Asymptotic formula for tau of mean rho}
\end{equation}

From this, Corollary~\ref{Corollary of bound for frobenian function of mean 0 invariant in the conductor}, Lemma~\ref{Lemma for uniform absolute upper bound for frobenian functions} and partial summation we deduce that the contribution of the latter two sums is
\begin{equation*}
    \ll_{ B, \varepsilon, H, n} q^{\varepsilon} x (\log x)^{m(\rho) - 1 - B }.
\end{equation*}
This is part of the error term of \eqref{Formula for frobenian function with small conductor} if we choose $B= N  + 1$.

Filling \eqref{Asymptotic formula for tau of mean rho} into the first sum of \eqref{Application of hyperbola method} gives an error term
\begin{equation*}
    \ll x (\log x)^{m(\rho) - M - 2} \sum_{n \leq \sqrt{x}} \frac{|\rho_n(n)|}{n} \ll  x (\log x)^{m(\rho) - M - 2} (\log x)^{H + |m(\rho)|}.
\end{equation*}
With the choice $M  =  H + 2\lceil |m(\rho)| \rceil + N$ this is part of the error term of \eqref{Formula for frobenian function with small conductor}.

The remaining part to deal with is the sum over the following terms
\begin{equation*}
    \frac{\gamma_j(m(\rho)) x}{\Gamma(m(\rho) - j) j!} \sum_{n \leq \sqrt{x}} \frac{\rho_n(n)}{n} (\log \frac{x}{n})^{m(\rho) - 1 - j}
\end{equation*}
We expand $ (\log \frac{x}{n})^{m(\rho) - 1 - j} =  (\log x- \log n)^{m(\rho) - 1 - j}$ to get
\begin{equation*}
    (\log \frac{x}{n})^{m(\rho) - 1 - j} = \sum_{h = 0}^{\infty} (-1)^h \binom{m(\rho) - 1 - j}{h}(\log x)^{m(\rho) - 1 - j - h} (\log n)^h.
\end{equation*}
Using the identity $\binom{m(\rho) - 1 - j}{h} = \frac{\Gamma(m(\rho) - j)}{\Gamma(m(\rho) - j - h) h!}$ we have the sum over $j,h$ of 
\begin{equation*}
     \frac{\gamma_j(m(\rho)) x (\log x)^{m(\rho) - 1 - j - h}}{\Gamma(m(\rho) - j - h) j! h!} \sum_{n \leq \sqrt{x}} \frac{\rho_n(n)}{n} (\log n)^{h}.
\end{equation*}
The sum over $n$ converges by Corollary \ref{Corollary of bound for frobenian function of mean 0 invariant in the conductor} and partial summation. The total series is thus equal to $L_n^{(h)}(\rho,1)$ by Abel's theorem. The resulting error bounds are
\begin{align*}
     &\sum_{n > \sqrt{x}} \frac{\rho_n(n)}{n} (\log n)^{h} \ll_{C, \varepsilon} (h + 1) e^{|S|} q^{\varepsilon} (\log x)^{-C + h} \\
     &\sum_{n = 1}^{\infty} \frac{\rho_n(n)}{n} (\log n)^{h} \ll_{C, \varepsilon} (h + 1)e^{|S|} q^{\varepsilon}.
\end{align*}
We use the first bound if $j + h \leq N$ and the second one if $j + h > N$. Taking $C \geq 2 N + 1$, using the bounds $\Gamma(m(\rho) - j - h) \gg 1$ and \cite[II.5.(7)]{tenenbaum1995introduction} proves the theorem.
\end{proof}
\begin{remark}
In Corollary \ref{Corollary of bound for frobenian function of mean 0 invariant in the conductor} and Theorem~\ref{Propostion for frobenian function with small conductor} we may count in $S$ only those $p \in S$ which do not divide the conductor. This is because for all $\varepsilon > 0$ we have $e^{O(\omega(q(\rho)))} \ll_{\varepsilon} q(\rho)^{\varepsilon}$ by the divisor bound.
\end{remark}
\subsection{Linked sums}
We first introduce relevant definitions.
\begin{definition}
Let $I$ be an index set. A subset $U_{\R} \subset \R_{\geq 0}^{I}$ is \emph{downward-closed } if $(x_i)_{i \in I} \in U_{\R}$ implies that $\{(y_i)_{i \in I} \in \R_{\geq 0}^{I}: y_i \leq x_i \text{ for all } i \in I \} \subset U$.

Similarly, let $K_i$ be a number field for all $i \in I$. A subset $U \subset \prod_{i \in I} I_{K_i}$ is \emph{downward-closed} if $(\mathfrak{a}_i)_{i \in I} \in U$ implies that $ \{(\mathfrak{b}_i)_{i \in I} \in \prod_{i} I_{K_i}: \Norm_{K_i}(\mathfrak{b}_i) \leq \Norm_{K_i}(\mathfrak{a}_i) \text{ for all } i \in I \} \subset U$.

The \emph{length} of a downward-closed set $U_{\R} \subset \R_{\geq 0}^{I}$, respectively $U \subset \prod_i I_{K_i} $ is
\begin{align*}
    &L_{U_{\R}} := \sup_{(x_i)_{i \in I} \in U_{\R}} \prod x_i
    &L_U := \sup_{(\mathfrak{a}_i)_{i \in I} \in U} \prod_{i \in I} \Norm_{K_i}(\mathfrak{a}_i).
\end{align*}
\end{definition}

We will use the following notation.
\begin{definition}
    Let $I$ be an index set, and $X_i$ a set for all $i \in I$. Let $U \subset  \prod_{i \in I} X_i$, $J \subset I$ and $x_j \in X_j$ for all $j \in J$. We define the $\emph{slice}$ of $U$ with respect to the $x_j$ as
    \begin{equation*}
        U[(x_j)_{j \in J}] := U[(x_j \in X_j)_{j \in J}] := U \cap \left(\prod_{j \in J}\{ x_j \} \times \prod_{i \in I \setminus J} X_i \right).
    \end{equation*}
\end{definition}
Some basic properties of downward-closed sets are 
\begin{proposition}
Let $U \subset \prod_{i \in I} I_{K_i}$ be a downward-closed set, then
\begin{enumerate}
    \item If $I$ is finite then $L_U \ll_{K_i} \# U \ll_{K_i} L_U (\log L_U)^{|I| -1}$.
    \item Let $J \subset I$ and $\mathfrak{a}_j \in I_{K_j}$ for $j \in J$. The slice $U[ (\mathfrak{a}_j)_{j \in J}]$ is then also downward-closed. Similarly if $U_{\R} \subset \R_{\geq 0}^I$ is downward-closed and $x_j \in \R_{\geq 0}$ then the slice $U[(x_j)_{j \in J}]$ is downward-closed.
\end{enumerate}
\end{proposition}
\begin{proof}
The first statement is clear if $U$ is infinite. 

If it is finite choose $\mathfrak{a_i}$ such that $L_U = \prod_{i \in I} \Norm(\mathfrak{a_i})$. Recall that for any number field $K$ there exists a constant $C_K$ such that $\#\{ \mathfrak{n} \in I_{K}: \Norm(\mathfrak{n}) \leq T\} \sim C_K T$ as $T \to \infty$. Note now that $U$ contains the set 
$\{(\mathfrak{u}_i)_{i \in I} : \Norm(\mathfrak{u_i}) \leq \Norm(\mathfrak{a_i}) \}$ which has size $\gg \prod_{i \in I} \Norm(\mathfrak{a_i}) = L_U$. It is also contained in the set $\{(\mathfrak{u}_i)_{i \in I} : \prod_{i \in I}\Norm(\mathfrak{u_i}) \leq L_U \}$ which is of size $\ll L_U (\log L_U)^{|I| -1}$.

The latter statements follow directly from the definitions.
\end{proof}
We will require a lemma that positive sums over downward-closed sets can be approximated by integrals.
Note that any downward-closed set $U_\R \subset \R_{\geq 0}^n$ defines a downward-closed set $U := \{ (\mathfrak{u}_i)_{i \in I} \in \prod_{i = 1}^n I_{K_i}: (\Norm(\mathfrak{u}_i))_{i \in I} \in U_\R \} $. 
\begin{lemma}
For $i = 1, \dots, n$ let $f_i: I_{K_i} \to \R_{\geq 0}$ be a non-negative arithmetic function and $g_i,h_i: \R_{\geq 0} \to \R_{\geq 0}$ non-negative real functions such that for all $x > 0$ we have $|\int_0^x g_i(x_i)d x_i  - \sum_{\Norm(\mathfrak{u}_i) \leq x} f_i(\mathfrak{u}_i)| \leq \int_0^x h_i(x_i) d x_i$. Then 
\begin{equation*}
    \Big |\int_{\mathbf{x} \in U_{\R}} \prod_{i = 1}^n g_i(x_i) dx_i - \sum_{\mathbf{u} \in U} \prod_{i = 1}^n f_i(\mathfrak{u}_i)\Big| \leq \sum_{J \subsetneq \{1, \dots, n\}} \int_{U_{\R}} \prod_{j \in J} g_j(x_j) dx_j \prod_{i \in \{1, \dots, n\} \setminus J} h_i(x_i) d x_i.
\end{equation*}
\label{Lemma that sums can be approximated by integrals for downward-closed sets}
\end{lemma}
\begin{proof}
We will prove this via induction on $n$, if $n = 1$ then the result is equal to the assumption. Otherwise we can write 
\begin{equation*}
    \sum_{\mathbf{u} \in U} \prod_{i \in I} f_i(\mathfrak{u}_i) = \sum_{\mathfrak{u}_1} f_1(\mathfrak{u}_1) \sum_{\mathbf{u} \in U[\mathfrak{u}_1]} \prod_{i \in I} f_i(\mathfrak{u}_i).
\end{equation*}

By the induction hypothesis we can approximate the sum over the slice $U[\mathfrak{u}_1]$, because it is also downward-closed. The error term is
\begin{equation*}
    \begin{split}
        &\sum_{\mathfrak{u}_1} f_1(\mathfrak{u}_1) \sum_{J \subsetneq \{2, \dots, n\}} \int_{U_{\R}[\Norm(\mathfrak{u}_1)]} \prod_{j \in J} g_j(x_j) dx_j \prod_{i \in \{2, \dots, n\} \setminus J} h_i(x_i) d x_i = \\ 
        &\sum_{J \subsetneq \{2, \dots, n\}} \int \sum_{\Norm(\mathfrak{u}_1) \in U_{\R}[x_i : i \in \{2, \dots, n\}]} f_1(\mathfrak{u}_1) \prod_{j \in J} g_j(x_j) dx_j \prod_{i \in \{2, \dots, n\} \setminus J} h_i(x_i) d x_i.
    \end{split}
\end{equation*}
Since $U_{\R}[x_i : i \in \{2, \dots, n\}]$ is downward-closed there exists an $x > 0$ such that it is equal to $[0, x]$ or $[0,x)$. Applying the inequality $\sum_{\Norm(\mathfrak{u}_1) \leq x} f_1(\mathfrak{u}_1) \leq \int_0^x g_1(x_1)d x_1  +\int_0^x h_1(x_1)dx_1$ we get
\begin{equation}
    \sum_{ \{2, \dots, n\} \neq J \subsetneq \{1, \dots, n\}} \int_{U_{\R}} \prod_{j \in J} g_j(x_j) dx_j \prod_{i \in \{1, \dots, n\} \setminus J} h_i(x_i) d x_i.
    \label{Error term coming from induction step}
\end{equation}

The main term is given by 
\begin{equation*}
    \sum_{\mathfrak{u}_1} f_1(\mathfrak{u}_1) \int_{U_{\R}[\Norm(\mathfrak{u}_1)]} \prod_{i = 2}^n g_i(x_i) dx_j   = \int  \Big(\sum_{\Norm(\mathfrak{u}_1) \in U_{\R}[x_i : i \in \{2, \dots, n\}]} f_1(\mathfrak{u}_1) \Big) \prod_{i = 2}^s g_i(x_i) dx_i.
\end{equation*}
Using again that $U_{\R}[x_i : i \in \{2, \dots, n\}]$ is downward-closed we can approximate the sum over $\mathfrak{u}_1$ using the assumption. The resulting error term is 
\begin{equation*}
    \int_{U_{\R}}  h_1(x_1)d x_1 \prod_{i = 2}^s g_i(x_i) dx_i.
\end{equation*}
Which together with \eqref{Error term coming from induction step} gives the required error term. The main term is
\begin{equation*}
 \int_{\mathbf{x} \in U_{\R}} \prod_{i = 1}^n g_i(x_i) dx_i
\end{equation*}
as desired.
\end{proof}
\begin{remark}
An interesting special case is when $g_i = 0$. In this case the lemma says that if $ |\sum_{\Norm(\mathfrak{u}_i) \leq x} f_i(\mathfrak{u}_i)| \leq \int_0^x h_i(x_i)$ then
\begin{equation*}
    |\sum_{\mathbf{u} \in U} \prod_{i = 1}^n f_i(\mathfrak{u}_i)| \leq \int_{\mathbf{x} \in U_{\R}} \prod_{i = 1}^n h_i(x_i) dx_i.
\end{equation*}

We will apply Lemma~\ref{Lemma that sums can be approximated by integrals for downward-closed sets} to the case when $U$ is defined by inequalities of the type $\prod_i x_i^{\alpha_i} \leq T$ for $\alpha_i \geq 0$. It that case it can be reinterpreted as a form of the hyperbola method on toric varieties \cite{Pieropan2020Hyperbola} which only works if the function $f(x_1, \dots, x_n)$ can be written as a product of single variable function $f_i(x_i)$. However, unlike \cite[Theorem~1.1]{Pieropan2020Hyperbola}, it can also be applied when $\sum_{x_i \leq T} f_i(x_i) \sim C(\log T)^M$.
\end{remark}

The functions which will \emph{link} variables are the following.
\begin{definition}
Let $K, L$ be number fields, $A > 0$ and $q \in \N$, an $(A,q)$-\emph{oscillating bilinear character} is a bilinear morphism $\alpha:I_K \times I_L \to \C$ such that for all $\mathfrak{a} \in I_K, \mathfrak{b} \in I_{L}$ the character $\alpha(\mathfrak{a}, \cdot)$, respectively $\alpha(\cdot, \mathfrak{b})$
\begin{enumerate}
    \item Is either $0$ or a Hecke character of finite modulus $\mathfrak{m}$ such that $\Norm_L(\mathfrak{m}) \leq q  \Norm_K(\mathfrak{a})^{A}$, respectively $\Norm_K(\mathfrak{m}) \leq q \Norm_L(\mathfrak{b})^{A}$.
    \item Is non-principal if there exists a $p \nmid q$ such that $v_p(N_K(\mathfrak{a})) = 1$, respectively $v_p(N_L(\mathfrak{b})) = 1$.
\end{enumerate}
\end{definition}

The oscillating bilinear characters used in this paper are described in the following lemma. We will use the notation $\Jacobi{\cdot}{\cdot}$ for the Jacobi symbol of $\Q$ and $\PowerRes{\cdot}{\cdot}{4}$ for the $4$-th power residue symobol of $\Q(i)$. Note that for all $n \in \N, \mathfrak{m} \in I_{\Q(i)}$ we have $\Jacobi{n}{\Norm(\mathfrak{m})} = \PowerRes{n^2}{\mathfrak{m}}{4}$. This can be seen since they are both bilinear and take the same value if $\mathfrak{m}$ is prime.
\begin{lemma}
The following functions $\alpha$ and their inverses are $(6,2^{24})$-oscillating bilinear characters. The functions take the value $0$ if the below description is not defined.
\begin{enumerate}
    \item $(m,n) \in \N^2 \to \Jacobi{n}{m}$.
    \item $(\mathfrak{m},n) \in I_{\Q(i)} \times \N \to \PowerRes{n}{\mathfrak{m}}{4}$.
    \item $(\mathfrak{m},n) \in I_{\Q(i)} \times \N \to \Jacobi{\Norm(\mathfrak{m})}{n}$.
    \item $(\mathfrak{m},n) \in I_{\Q(i)} \times \N \to \Jacobi{n}{\Norm(\mathfrak{m})} = \PowerRes{n^2}{\mathfrak{m}}{4}$.
    \item $(\mathfrak{m},n) \in I_{\Q(i)} \times \N \to \PowerRes{n}{\mathfrak{m}}{4} \Jacobi{\Norm(\mathfrak{m})}{n}$.
    \item $(\mathfrak{m},\mathfrak{n}) \in I_{\Q(i)}^2 \to \PowerRes{\Norm(\mathfrak{n})}{\mathfrak{m}}{4}$.
    \item $(\mathfrak{m},\mathfrak{n}) \in I_{\Q(i)}^2 \to \PowerRes{\Norm(\mathfrak{n})}{\mathfrak{m}}{4} \PowerRes{\Norm(\mathfrak{m})}{\mathfrak{n}}{4}$. 
    \item $(\mathfrak{m},\mathfrak{n}) \in I_{\Q(i)}^2 \to \PowerRes{\Norm(\mathfrak{n})}{\mathfrak{m}}{4} \PowerRes{\Norm(\mathfrak{m})}{\mathfrak{n}}{4}^{-1}$. 
\end{enumerate}

Even more is true. If $p$ is a prime such that $v_p(\Norm(\mathfrak{u})) = 1$ for $\mathfrak{u} \in \N, I_{\Q(i)}$ then $p$ divides the norm of the conductor of the Hecke characters $\alpha(\mathfrak{u}, \cdot), \alpha(\cdot, \mathfrak{u})$.
\label{Examples of oscillating bilinear characters}
\end{lemma}
\begin{proof}
That $\alpha$ is bilinear is clear. In what follows we will write $\mathfrak{u}$ for one of $n, m, \mathfrak{n}, \mathfrak{m}$ and $\mathfrak{v}$ for the other variable of $\alpha$ to keep the arguments uniform. If we fix the variable $\mathfrak{v}$ then $\alpha(\mathfrak{v}, \cdot)$, respectively $\alpha(\cdot, \mathfrak{v})$ is a product of a Hecke characters modulo $\mathfrak{v}$ or a Hecke character of the field extension $\Q(\sqrt{-1}, \sqrt[4]{\Norm(\mathfrak{v})})$ which has discriminant dividing $ 2^{24}\Norm(\mathfrak{v})^6$. The norm of the modulus of this Hecke character must divide this discriminant. It only remains to show the non-principality.

We will first do the case that $\mathfrak{u}$ is an odd totally split prime $\mathfrak{p}$. It then suffices to show that there exists a single $\mathfrak{v} \equiv 1 \pmod{2^{24}}$ such that the value of $\alpha(\mathfrak{p}, \mathfrak{v})$, respectively $\alpha(\mathfrak{v},\mathfrak{p})$ is not $1$. Except in the cases $5,6,8$ this can be done by choosing $\mathfrak{v}$ as a totally split prime and applying Chebotarev density. 

Consider now the cases of $5,6,8$ where we have to choose an ideal $\mathfrak{v} \in I_{\Z[i]}$. Start by finding an prime $\mathfrak{q} \equiv 1 \pmod{2^{24}}$ such that $\PowerRes{\Norm(\mathfrak{p})}{\mathfrak{q}}{4}$ is equal to $i$ or $-i$, i.e. $\Norm(\mathfrak{p})$ is not a square modulo $\Norm(\mathfrak{q})$. This is possible because of Dirichlet's theorem. Then $\alpha(\mathfrak{p}, \mathfrak{v})$, respectively $\alpha(\mathfrak{v},\mathfrak{p})$ is not equal to $1$ for one of two choices $\mathfrak{v} = \mathfrak{q}, \overline{\mathfrak{q}}$.

The only remaining case is to find some $n \in \N$ such that $\PowerRes{n}{\mathfrak{p}}{4} \Jacobi{\Norm(\mathfrak{p})}{n} \neq 1$. For this choose $n \equiv 1 \pmod{2^{24}}$ which is a not a square modulo $\Norm(\mathfrak{p})$. By quadratic reciprocity 
\begin{equation*}
    \PowerRes{n}{\mathfrak{p}}{4} \Jacobi{\Norm(\mathfrak{p})}{n} = \PowerRes{n}{\mathfrak{p}}{4} \Jacobi{n}{\Norm(\mathfrak{p})} = \PowerRes{n^3}{\mathfrak{p}}{4} \neq 1.
\end{equation*}

To deduce the general case assume that $p$ is an odd prime such that $v_p(\Norm(\mathfrak{m})) = 1$. In this case we can write $\mathfrak{m} = \mathfrak{p} \mathfrak{m}'$ such that $\Norm(\mathfrak{p}) = p$ and $p, \Norm(\mathfrak{m}')$ are coprime. In this case $\alpha(\mathfrak{m}, \cdot) = \alpha(\mathfrak{p}, \cdot) \alpha(\mathfrak{m}', \cdot)$ is a product of a non-principal Hecke character such that $p$ divides the conductor and a Hecke character whose conductor is coprime to $p$. The conductor of such a product has to be be divisible by $p$. The case $v_p(\Norm(\mathfrak{n})) = 1$ is completely analogous.
\end{proof}
For these oscillating bilinear character we prove an analogue of \cite[Lemma~4]{heathbrown1993selmer}.
\begin{proposition}
Let $K, L$ be number fields. Let $a_{\mathfrak{m}}, b_{\mathfrak{n}}$ be complex numbers of absolute value at most $1$ where $\mathfrak{m}, \mathfrak{n}$ ranges over $I_K, I_L$ respectively. Let $U \subset I_K \times I_L$ be a downward-closed set and $\alpha$ an $(A, q)$-oscillating bilinear character. There exists a $\delta > 0$ depending only on $K, L, A$ such that for all $ M,N \geq 2$
\begin{equation*}
    \mathop{\sum\sum}_{\substack{(\mathfrak{m}, \mathfrak{n}) \in U \\ \Norm(\mathfrak{m}) \leq M \\ \Norm( \mathfrak{n}) \leq N}} a_{\mathfrak{m}} b_{\mathfrak{n}} \alpha(\mathfrak{m}, \mathfrak{n}) \ll_{K, L, A} qMN(\log M N)^{\delta^{-1}}( N^{-\delta}+ M^{-\delta}).
\end{equation*}
The implicit constant are in particular independent of $U$.
\label{Proposition on double oscillation}
\end{proposition}
\begin{proof}
Write $B(M, N)$ for the left-hand side above. We may assume by symmetry that $M  \leq N$. Let $k$ be a natural number to be chosen later. By applying H\"older's inequality to the sum over $\mathfrak{m}$ and switching the order of summation we get
\begin{equation*}
    \begin{split}
        |B(M,N)|^{k} &\ll M^{k-1}  \sum_{\Norm(\mathfrak{m}) \leq M} \Big\lvert\sum_{\substack{(\mathfrak{m}, \mathfrak{n}) \in U \\ \Norm(\mathfrak{n}) \leq N}}  b_{\mathfrak{n}} \alpha(\mathfrak{m}, \mathfrak{n})  \Big\rvert^{k}  \\
        &\leq M^{k-1} \sum_{\Norm(\mathfrak{n}_1), \dots,  \Norm(\mathfrak{n}_k) \leq N} \Big\lvert \sum_{\substack{(\mathfrak{m}, \mathfrak{n}_i) \in U \\ \Norm(\mathfrak{m}) \leq M}} c_{\mathfrak{m}} \alpha(\mathfrak{m}, \mathfrak{n}_1 \cdots \mathfrak{n}_k) \Big\rvert.
    \end{split}
\end{equation*}
Where $|c_\mathfrak{m}| = 1$ is the complex number such that $c_\mathfrak{m} \Big(\sum b_{\mathfrak{n}} \alpha(\mathfrak{m}, \mathfrak{n})\Big)^{k}$ is a positive real.

At the cost of a factor of $k$ we may sum only over those tuples such that $\Norm(\mathfrak{n}_1), \cdots,$ $\Norm(\mathfrak{n}_{k -1}) \leq  \Norm(\mathfrak{n}_k) \leq N$. Note that the condition $(\mathfrak{m}, \mathfrak{n}_i) \in U$ is then implied by $(\mathfrak{m}, \mathfrak{n}_k) \in U$ since $U$ is downward-closed. Write $\mathfrak{l} = \mathfrak{n}_1 \cdots \mathfrak{n}_{k-1}$. The amount of tuples giving the same $\mathfrak{l}$ is 
\begin{equation*}
    \sum_{\substack{\Norm(\mathfrak{n}_1), \dots, \Norm(\mathfrak{n}_{k-1}) \leq \Norm(\mathfrak{n}_k) \\ \mathfrak{n}_1 \cdots \mathfrak{n}_{k-1} = \mathfrak{l}}} 1 \leq \tau_{k-1}(\mathfrak{l}).
\end{equation*}

Applying Cauchy-Schwarz gives that $|B(M,N)|^{2k}$ is bounded by
\begin{equation*}
    \begin{split}
        & k^2 M^{2k-2} \Big(\sum_{\Norm(\mathfrak{n}_k) \leq N} \sum_{\Norm(\mathfrak{l}) \leq N^{k-1}}  \tau_{k-1}(\mathfrak{l})^2 \Big) \Big( \sum_{\Norm(\mathfrak{n}_k) \leq N}  \sum_{\Norm(\mathfrak{m}_1), \Norm(\mathfrak{m}_2) \leq M} \Big\lvert \sum_{\mathfrak{l}} \alpha(\mathfrak{m}_1 , \mathfrak{l}) \overline{\alpha(\mathfrak{m}_2 , \mathfrak{l})} \Big\rvert \Big) \\
        &\ll k^2 M^{2k - 2} N^{k + 1} (\log N)^{k^2 -2 k} \sum_{\Norm(\mathfrak{m}_1), \Norm(\mathfrak{m}_2) \leq M} \Big\lvert \sum_{\Norm(\mathfrak{l}) \leq N^{k-1}} \alpha(\mathfrak{m}_1 , \mathfrak{l}) \overline{\alpha(\mathfrak{m}_2 , \mathfrak{l})} \Big\rvert.
    \end{split}
\end{equation*}

Every pair $\mathfrak{m}_1, \mathfrak{m}_2$ such that $v_p(\Norm_K(\mathfrak{m}_1 \mathfrak{m}_2)) \neq 1$ for all $p \nmid q$ has the property that their product can be written (non-uniquely) as $\mathfrak{m}_1 \mathfrak{m}_2 = \mathfrak{u} \mathfrak{v}^2 \mathfrak{w}^3$ with $\Norm(\mathfrak{u})| q$. The amount of such pairs is by the divisor bound 
\begin{equation*}
    \ll \sum_{\substack{\Norm(\mathfrak{u} \mathfrak{v}^2 \mathfrak{w}^3) \leq M^2\\ \Norm(\mathfrak{u}) \mid q}} \tau(\mathfrak{u} \mathfrak{v}^2 \mathfrak{w}^3) \ll \sum_{\Norm(\mathfrak{u}) \mid q} \Norm(\mathfrak{u})^{\frac{1}{2}} \sum_{\Norm(\mathfrak{v}) \Norm(\mathfrak{w})^{\frac{3}{2}} \leq M}  \Norm(\mathfrak{v} \mathfrak{w})^{\frac{1}{2}}\ll q M^{\frac{3}{2}}.
\end{equation*}
The total contribution of this case is thus $ \ll k^2 q M^{2k-\frac{1}{2}} N^{2k} (\log N)^{k^2-2k} $. 

We claim that in the other cases $\alpha(\mathfrak{m}_1 , \cdot) \overline{\alpha(\mathfrak{m}_2 , \cdot})$ is a non-principal character. Since $\alpha$ is an oscillating bilinear character the norm of its modulus is at most $q \Norm_K(\mathfrak{m}_1 \mathfrak{m_2})^{A}$. The sum over $\mathfrak{l}$ is in this case \begin{equation*}
    \ll q^{\frac{1}{d_L + 1}} \Norm_K(\mathfrak{m}_1 \mathfrak{m_2})^{\frac{A}{d_L +1}} (\log M)^{d_L}N^{(k-1)\frac{d_L-1}{d_L + 1}}
\end{equation*} 
by Polya-Vinogradov for number fields \cite{landau1918verallgemeinerung}.

The total contribution of this part is thus \begin{equation*}
    \ll k^2 q M^{2k+ \frac{2 A}{d_L + 1}} (\log M)^{d_L} N^{\frac{ 2k d_L  + 2}{d_L + 1}} (\log N)^{k^2 -2k}.
\end{equation*} 
Choosing $k$ sufficiently large such that $\frac{2k d_L + 2 + 2A}{d_L + 1} \leq 2k - \frac{1}{2}$ and $d_L \leq 2k$, taking $\delta = \frac{1}{4k}$ and using that $M \leq N$ gives the desired bound.

It remains to prove the claim. We may assume that there exists a prime $p \nmid q$ such that $v_{p}(\Norm(\mathfrak{m}_1)) = 1$ and $p \nmid \Norm(\mathfrak{m}_2)$. There exists a natural number $d$ such that $\overline{\alpha(\mathfrak{m}_2 , \cdot)} = \alpha(\mathfrak{m}_2 , \cdot)^d$ because it is a Hecke character. We thus have to show that $\alpha(\mathfrak{m}_1 \mathfrak{m}_2^d, \cdot)$ is non-principal. But $v_p(\Norm(\mathfrak{m}_1 \mathfrak{m}_2^d)) = 1$ so this is true by definition.
\end{proof} 

We now consider the following situation. Let $f: \prod_{i \in I} I_{K_i} \to \C$ be a function and $U \subset \prod_i I_{K_i}$ a downward-closed set. Our goal will be to evaluate $\sum_{\mathbf{u} \in U} f(\mathbf{u})$ uniformly in $U$. We will require the following additional data. 
\begin{enumerate}
    \item Two integers $q_{\text{osc}}, q_{\text{frob}}$ and a constant $A$.
    \item A relation $R \subset I \times I$ which is symmetric, i.e. $(i, j) \in R$ implies $(j, i) \in R$, and such that $(i, i) \not \in R$ for all $i \in I$. We will say that $i$ is \emph{linked} to $j$ or that the variable $\mathfrak{u}_i$ is \emph{linked} to $\mathfrak{u}_j$ if $(i,j) \in R$.
    \item A set of indices $i$ which we will call \emph{frobenian}. We will also say that the variable $\mathfrak{u}_i$ is \emph{frobenian}.
\end{enumerate}

This data has to satisfy the following properties.
\begin{enumerate}
    \item If $i$ is linked to $j$ then for all tuples $\mathbf{u}_{\hat{i j}} := (\mathfrak{u}_k)_{k \in I, k \neq i,j}$ then three functions
    \begin{align*}
        &a(\cdot ;{\mathbf{u}_{\hat{i j}}}): I_{K_i} \to \C: \mathfrak{u}_i \to a(\mathfrak{u}_i ; \mathbf{u}_{\hat{i j}}) \\
        &b(\cdot ;{\mathbf{u}_{\hat{i j}}}): I_{K_j} \to \C: \mathfrak{u}_j \to b(\mathfrak{u}_j ; \mathbf{u}_{\hat{i j}}) \\ 
        &\alpha(\cdot, \cdot ; \mathbf{u}_{\hat{i j}}): I_{K_i} \times I_{K_j} \to \C: (\mathfrak{u}_i, \mathfrak{u}_j) \to \alpha(\mathfrak{u}_i, \mathfrak{u}_j; \mathbf{u}_{\hat{i j}})
    \end{align*}
    exist such that $f((\mathfrak{u}_{\ell})_{\ell \in I}) = a(\mathfrak{u}_i ; \mathbf{u}_{\hat{i j}}) b(\mathfrak{u}_j ; \mathbf{u}_{\hat{i j}}) \alpha(\mathfrak{u}_i, \mathfrak{u}_j; \mathbf{u}_{\hat{i j}})$.
    We also require that the function $\alpha(\cdot, \cdot; \mathbf{u}_{\hat{i j}})$ has to be an $(A, q_{\text{osc}})$-oscillating bilinear character and $|a(\mathfrak{u}_i ; \mathbf{u}_{\hat{i j}})|, |b(\mathfrak{u}_j ; \mathbf{u}_{\hat{i j}})| \leq 1$. Less formally, the only interaction between $\mathfrak{u}_i$ and $\mathfrak{u}_j$ in $f$ has to come from an oscillating bilinear function.
    \item If $i$ is frobenian then for all tuples $\mathbf{u}_{\hat{i}} := (\mathfrak{u}_k)_{k \in I, k \neq i}$ let 
    \begin{equation*}
        S(\mathbf{u}_{\hat{i}}) = \{p \mid q_{\text{frob}} \} \cup \{ p \mid \Norm(\mathfrak{u}_k) : k \neq i \}.
    \end{equation*}
    For all $\mathbf{u}_{\hat{i}}$ there has to exist a constant $|c(\mathbf{u}_{\hat{i}})| \leq 1$ and a $S(\mathbf{u}_{\hat{i}})$-frobenian multiplicative function $\rho(\cdot ;\mathbf{u}_{\hat{i}})$ such that $f((\mathfrak{u}_{\ell})_{\ell \in I}) = c(\mathbf{u}_{\hat{i}}) \rho(\mathfrak{u}_i ;\mathbf{u}_{\hat{i}})$. The frobenian multiplicative function $\rho(\cdot ;\mathbf{u}_{\hat{i}})$ has to have conductor at most $\leq q_{\text{frob}} \prod_{j \text{ linked to } i} \Norm(\mathfrak{u}_j)^A$ and its mean is $0$ if there exists an index $j$ linked to $i$ such that $\mathfrak{u}_j \neq 1$.
\end{enumerate}

Given this data we can prove that some of the variables may be taken to be small or even equal to $1$. This will allow us to reduce the amount of variables we have to consider.
\begin{theorem}
Let $f, U$ be as above. For all $\varepsilon, \delta, C_1 > 0$ there exists a $C_2 > 0$ such that
\begin{equation*}
    \sum_{\mathbf{u} \in U} f(\mathbf{u}) = \sum_{\substack{\mathbf{u} \in U \\ \mathbf{u} \emph{ satisfies \eqref{Condition for linked variables}}}} f(\mathbf{u}) + O\Big(q_{\emph{osc}}q_{\emph{frob}}^{\varepsilon}L_U (\log L_U)^{-C_1}\Big).
\end{equation*}
Where the condition in the sum is
\begin{equation}
    \begin{split}
        &\emph{For every pair of linked variables } \mathfrak{u}_i, \mathfrak{u}_j \emph{ at most one is } > (\log L_U)^{C_2}. \\
        &\emph{If } \mathfrak{u}_i \emph{ is a frobenian variable and } \Norm(\mathfrak{u}_i) > e^{(\log L_U)^{\delta}} \emph{ then every } \\& \emph{variable linked to } \mathfrak{u}_i \emph{ is equal to } 1.  
    \end{split}
    \label{Condition for linked variables}
\end{equation}
The implicit constants and $C_2$ depend only on $|I|, A, K_i, \varepsilon, \delta, C_1$.
\label{Theorem on linked variables}
\end{theorem}
\begin{proof}
Choose quantities $\frac{1}{2}  = V_0 < 1 = V_1  < V_2 < \cdots < V_N = L_U$ such that $V_{n+1} \leq 2 V_{n}$ for all $n$ and such that there exists $m_1, m_2$ such that $V_{m_1} = (\log L_U)^{C_2}, V_{m_2} = e^{(\log L_U)^{\delta}}$. We can choose such quantities with $N = O(\log L_U)$. For every tuple $\mathbf{V} = (V_{n_i})_{i \in I}$ we can then consider the subsum
\begin{equation*}
    S(\mathbf{V}) = \sum_{\substack{\mathbf{u} \in U \\ V_{n_i - 1} < \Norm(\mathfrak{u}_i) \leq V_{n_i}}} f(\mathbf{u}).
\end{equation*}

We will bound $S(\mathbf{V})$ for the tuples $\mathbf{V}$ which fail \eqref{Condition for linked variables}. Since there are at most $O\left((\log L_U)^{|I|}\right)$ of such tuples it suffices to prove that for such $\mathbf{V}$
\begin{equation*}
    S(\mathbf{V}) \ll_{\varepsilon} q_{\text{osc}}q_{\text{frob}}^{\varepsilon} L_U(\log L_U)^{-C_1 - |I|}.
\end{equation*}
Note that the sum is empty if $\prod_{i \in I} V_{n_i - 1} > L_U$ so $S(\mathbf{V}) = 0$ in that case. We may thus assume that $\prod_{i \in I} V_{n_i} \leq 2^{|I|} L_U$.

If there are two linked variables $\mathfrak{u}_j, \mathfrak{u}_k$ such that $V_{n_j - 1}, V_{n_k - 1} \geq (\log L_U)^{C_2}$ then by assumption we have
\begin{equation*}
    S(\mathbf{V}) \leq \sum_{\substack{\mathbf{u}_{\hat{j k}} \\ V_{n_i - 1} < \Norm(\mathfrak{u}_i) \leq V_{n_i}}} \Big|\mathop{\sum \sum}_{\substack{(\mathfrak{u}_j, \mathfrak{u}_k) \in U[\mathbf{u}_{\hat{j k}}]\\ V_{n_j - 1} < \Norm(\mathfrak{u}_j) \leq V_{n_j} \\ V_{n_k - 1} < \Norm(\mathfrak{u}_k) \leq V_{n_k} }} a(\mathfrak{u}_j ; \mathbf{u}_{\hat{j k}}) b(\mathfrak{u}_k ; \mathbf{u}_{\hat{j k}}) \alpha(\mathfrak{u}_j, \mathfrak{u}_k; \mathbf{u}_{\hat{j k}}) \Big|.
\end{equation*}
The slice $U[\mathbf{u}_{\hat{j k}}]$ is downward-closed, so by Proposition~\ref{Proposition on double oscillation} there exists a $\delta' > 0$ such that the inner sum is 
\begin{equation*}
    \ll_{A, K_j, K_k} V_{n_j} V_{n_k} \log(V_{n_j} V_{n_k})^{\delta'^{-1}} \left(V_{n_j}^{- \delta'} + V_{n_k}^{-\delta'}\right).
\end{equation*}
Hence,
\begin{equation*}
        S(\mathbf{V}) \ll_{A, K_i} q_{\text{osc}}\prod_{i \in I} V_{n_i} \log(V_{n_j} V_{n_k})^{\delta'^{-1}} \Big(V_{n_j}^{- \delta'} + V_{n_k}^{-\delta'} \Big) 
        \ll q_{\text{osc}} L_U \log( L_U)^{\delta'^{-1} -C_2 \delta'}.
\end{equation*}
The desired inequality is true as long as one takes $C_2$ sufficiently large.

If $\mathfrak{u}_j$ is a frobenian variable such that $V_{n_j -1 } \geq e^{(\log L_U)^{\delta}}$ and for at least one linked variable $\mathfrak{u}_{\ell}$ we have $V_{n_\ell} \neq 1$ then we can consider $2$ cases. We may assume that $e^{(\log L_U)^{\delta}} \geq (\log L_U)^{C_2}$ by increasing the implied constant in the theorem.

Either there exists a variable $\mathfrak{u}_k$ linked to $\mathfrak{u}_j$ such that $V_{n_k - 1} \geq (\log L_U)^{C_2}$. In this case we are in the previous situation. In the other case we have by assumption $V_{n_k} \leq (\log L_U)^{C_2}$ for all variables $\mathfrak{u}_k$ linked to $\mathfrak{u}_j$. We then have
\begin{equation*}
    S(\mathbf{V}) \leq \sum_{\substack{\mathbf{u}_{\hat{j}} \\ V_{n_i - 1} < \Norm(\mathfrak{u}_i) \leq V_{n_i}}} |c(\mathbf{u}_{\hat{i}})|\Big| \sum_{\substack{ \mathfrak{u}_j \in U[\mathfrak{u}_{\hat{j}}] \\ V_{n_j - 1} < \Norm(\mathfrak{u}_j) \leq V_{n_j}}}  \rho(\mathfrak{u}_i ;\mathbf{u}_{\hat{i}}) \Big|.
\end{equation*}
By assumption $|c(\mathbf{u}_{\hat{i}})| \leq 1$ and $\rho(\mathfrak{u}_i ;\mathbf{u}_{\hat{i}})$ is a frobenian multiplicative function of conductor $\ll q_{\text{frob}} \prod_{k \text{ linked to } j} V_{n_k}^{A}$. It has mean zero since $V_{n_{\ell}} \neq 1$ so there is a linked variable $\mathfrak{u}_{\ell} \neq 1$. Corollary \ref{Corollary of bound for frobenian function of mean 0 invariant in the conductor} implies that for all $D > 0 $
\begin{equation*}
    S(\mathbf{V}) \ll_{D, \varepsilon} \sum_{\substack{(\mathfrak{u}_i)_{i \neq j} \in U[\mathfrak{u}_j = 1] \\  V_{n_i - 1} < \Norm(\mathfrak{u}_i) \leq V_{n_i}}}e^{O(\prod_{i \neq j}\omega( \mathfrak{u_i}))} \big(q_{\text{frob}}\prod_{ \mathfrak{u}_k \text{ linked to } \mathfrak{u}_j} V_{n_k}^A\big)^{\varepsilon} V_{n_j}(\log V_{n_j})^{-D}.
\end{equation*}
Use upper bounds to get rid of the condition induced by $U$ and use that $e^{O(\omega(\mathfrak{u}_i))} \leq \tau_{O(1)}(\mathfrak{u}_i)$ to see that $\sum_{ V_{n_i - 1} < \Norm(\mathfrak{u}_i) \leq V_{n_i}}e^{O( \omega(\mathfrak{u}_i))} \ll  L_U  (\log L_U)^{O(1)}$ and hence
\begin{equation*}
    \begin{split}
        S(\mathbf{V}) &\ll_{D, \varepsilon} q^{\varepsilon} V_{n_j} (\log L_U)^{-D \delta}  (\log L_U)^{|I| A C_2  \varepsilon }\prod_{i \neq j} V_{n_i} (\log V_{n_i})^{O(1)} \\ &\ll  q^{\varepsilon} L_U  (\log L_U)^{|I| A C_2 \varepsilon + O(1) - D \delta}.
    \end{split}
\end{equation*}
Picking $D$ sufficiently large gives the desired statement.
\end{proof} 
\section{Counting}
Let $S$ be the set of primes $p \leq 16897$, so all the statements in \S \ref{Subsection Local computation} which are only true for significantly large enough primes will be true for primes $p \not \in S$.
In this section we will find asymptotic formulas for the size of the sets $N_{ = \square}^{\Br}(T), N_{ = -\square}^{\Br}(T), N_{ \neq \pm \square}^{\Br}(T)$ defined in the introduction.

We will split these further into certain subsets. Let $\mathcal{P}_2(4)$ be the set of size $2$ subsets of $\{0,1,2,3\}$, this has $6$ elements. Consider the set
\begin{equation*}
    \Phi := \{ \mathbf{A} = (A_0, A_1, A_2, A_3) \in \N^4 : p \mid A_0 A_1 A_2 A_3 \Rightarrow p \in S \text{ or } p^3 \mid A_0 A_1 A_2 A_3 \}.
\end{equation*}
For any $\mathbf{A} \in \Phi$ we define $m_{\mathbf{A}} := \text{rad}(A_0 A_1 A_2 A_3 \prod_{p \in S} p), \theta_{\mathbf{A}} := A_0 A_1 A_2 A_3$. Consider also a 4-tuple $\mathbf{M} = (M_0, M_1, M_2, M_3)$ of cosets of $(\Z/8 m_{\mathbf{A}} \Z)^{\times 4}$ in $(\Z/8 m_{\mathbf{A}} \Z)^{\times}$.
\begin{remark}
We may replace the set $\Phi$ by any subset and the rest of the this section will still hold without any modification. For example, if we only want to count those $\mathbf{a}$ with $a_0, a_1, a_2, a_3$ coprime then we can instead use the subset 
\begin{equation*}
    \Phi' := \{ \mathbf{A} = (A_0, A_1, A_2, A_3) \in \Phi : \gcd(A_0, A_1, A_2, A_3) = 1 \}.
\end{equation*}
\label{Case when all a are coprime}
\end{remark}
We will use the notation $\{k, \ell, m, n\} = \{0,1,2,3\}$ and $v_{k \ell} = v_{\ell k} := v_{\{k, \ell\}}$, similarly for any tuple indexed by $\mathcal{P}_2(4)$. We will also write $t_{k \ell} = v_{k \ell} w_{k \ell}$, note that $v_{k \ell}$ and $w_{k \ell}$ are uniquely determined by $t_{k \ell}$.

We define $a_k :=  A_k u_k^2 \prod_{\{k, \ell\} \in \mathcal{P}_2(4)} v_{k \ell} w_{k \ell}$ and put $\mathbf{a} := (a_0, a_1, a_2, a_3)$. Note that due to the coprimality and square-freeness conditions the tuple $\mathbf{a}$ uniquely determines $\mathbf{A}, \mathbf{M}, \mathbf{u}, \mathbf{v}, \mathbf{w}$ and that $\theta_{\mathbf{A}} = \theta_{\mathbf{a}}= a_0 a_1 a_2 a_3 \bmod{\Q^{\times 2}}$.
For each pair $\mathbf{A}, \mathbf{M}$ of such tuples we define
\begin{equation*}
    \begin{split}
        &N_{\mathbf{A}, \mathbf{M}}(T) := \left\{\begin{array}{l}
            (\mathbf{u}, \mathbf{v}, \mathbf{w}) \\
             \mathbf{u} \in \N^4  \\
             \mathbf{v}, \mathbf{w} \in \N^{\mathcal{P}_2(4)}
        \end{array} : \begin{array}{l}
            |A_k u_k^2 \prod_{\{k, \ell\}} v_{k \ell} w_{k \ell}| \leq T, \\ \mu(m_\mathbf{A} \prod_{k} u_k \prod_{\ell \neq k} v_{k \ell} w_{k \ell})^2 = 1, \\
            p \mid  v_{k \ell} \Rightarrow \theta_{\mathbf{A}} \in \Q_p^{\times 2}, p \mid w_{k \ell} \Rightarrow \theta_{\mathbf{A}} \not \in \Q_p^{\times 2}, \\
            u_k^2 \prod_{\ell \neq k} v_{k \ell} w_{k \ell} \Mod{8 m_{\mathbf{A}}} \in M_k, \\ X_{\mathbf{a}}(\Q_p) \neq \emptyset \text{ for } p \nmid m_\mathbf{A} 
            \end{array} \right\}
    \end{split}
\end{equation*}

The set $N_{\mathbf{A}, \mathbf{M}}(T)$ is only non-empty if $M_0 M_1 M_2 M_3$ is a square because $\prod_k u_k^2 \prod_{\{k, \ell\}} v_{k \ell}^2 w_{k \ell}^2 \in M_0 M_1 M_2 M_3$. Denote the set of pairs of tuples $\mathbf{A}, \mathbf{M}$ such that $M_0 M_1 M_2 M_3$ is a square by $\Psi$ and let $\Psi(T) := \{ (\mathbf{A}, \mathbf{M}) \in \Psi : |A_k| \leq T\}.$

We will study the following subsets of $N_{\mathbf{A}, \mathbf{M}}(T)$
\begin{equation*}
    \begin{split}
        &N^{\Br}_{\mathbf{A}, \mathbf{M}}(T) := \{ (\mathbf{u}, \mathbf{v}, \mathbf{w}) \in N_{\mathbf{A}, \mathbf{M}}(T): X_{\mathbf{a}} \text{ has a BM obstruction} \} \\
        &N^{\mathcal{A}}_{\mathbf{A}, \mathbf{M}}(T) := \{ (\mathbf{u}, \mathbf{v}, \mathbf{w}) \in N_{\mathbf{A}, \mathbf{M}}(T): X_{\mathbf{a}} \text{ has a BM obstruction induced by } \mathcal{A} \} \\
        &N^{\text{loc}}_{\mathbf{A}, \mathbf{M}}(T) := \{ (\mathbf{u}, \mathbf{v}, \mathbf{w}) \in N_{\mathbf{A}, \mathbf{M}}(T) : p \mid  w_{k \ell} \Rightarrow p \equiv 3 \pmod{4} \}.
    \end{split}
\end{equation*}

Let $M_{\mathbf{A}} := 11/16$ if $\theta_{\mathbf{A}} \in -\Q^{\times 2}$ and $M_{\mathbf{A}} := 15/32$ if $\theta_{\mathbf{A}} \not \in \pm \Q^{\times 2}$. It will later turn out that this is the mean of a certain frobenian multiplicative function.
\subsection{Reductions}
\label{Reductions}
In this subsection we will reduce finding an asymptotic formula for $N^{\Br}(T)$ to finding one for $N^{\text{loc}}_{\mathbf{A}, \mathbf{M}}(T)$.
\begin{lemma}
\begin{equation*}
    \begin{split}
        \# N_{= \square}^{\Br}(T) &= \sum_{\substack{(\mathbf{A}, \mathbf{M}) \in \Psi(T) \\ \theta_{\mathbf{A}} \in \Q^{\times 2}}} \# N^{\Br}_{\mathbf{A}, \mathbf{M}}(T) \\
        \# N_{= -\square}^{\Br}(T) &= \sum_{\substack{(\mathbf{A}, \mathbf{M}) \in \Psi(T) \\ \theta_{\mathbf{A}} \in -\Q^{\times 2}}} \# N^{\Br}_{\mathbf{A}, \mathbf{M}}(T) \\
        \# N_{\neq \pm \square}^{\Br}(T) &= \sum_{\substack{(\mathbf{A}, \mathbf{M}) \in \Psi(T) \\ \theta_{\mathbf{A}} \not \in \pm \Q^{\times 2}}} \ \# N^{\Br}_{\mathbf{A}, \mathbf{M}}(T).
    \end{split}
\end{equation*}
\label{Lemma for splitting N^Br into subsums}
\end{lemma}
\begin{proof}
The sets in the right-hand side are disjoint and can be interpreted as a subset of the right-hand side because $\mathbf{A}, \mathbf{M}, \mathbf{u}, \mathbf{v}, \mathbf{w}$ are uniquely determined by $\mathbf{a}$. Elements counted by the right-hand side which are not counted in the left-hand side are tuples $\mathbf{a}$ such that there exists a prime $p \not \in S$ such that $v_p(a_0 a_1 a_2 a_3) = 1$ and such that $X_{\mathbf{a}}$ has a Brauer-Manin obstruction. These tuples cannot exist by Proposition \ref{No Brauer-Manin obstruction if p divides one coefficient}.
\end{proof}

The following lemma will allow us to bound $N_{= \square}^{\Br}(T)$.
\begin{lemma}
If $\theta_{\mathbf{A}} = A_0 A_1 A_2 A_3 \in \Q^{\times 2}$ then we have for $T > 2$
\begin{equation*}
    \# N^{\Br}_{\mathbf{A}, \mathbf{M}}(T) \ll \frac{T^2 (\log T)^2}{|\theta_{\mathbf{A}}|^\frac{1}{2}}.
\end{equation*}
\label{Bound for N^Br when theta is a square}
\end{lemma}
\begin{proof}
By Lemma~\ref{No Brauer-Manin obstruction for theta square} the left-hand side is bounded by the sum of the sizes of the following subsets of $N_{\mathbf{A}, \mathbf{M}}(T)$. The first one is defined by $a_i a_j = 1,-1,2,-2 \in \Q^{\times}/ \Q^{\times 2}$ for $i \neq j$ and the second one by $u_k = 1$ for all $k$. We may assume that $i =0, j =1$ when bounding the contribution of the first one. Using that $t_{k \ell} = v_{k \ell} w_{k \ell}$ uniquely determines $v_{k \ell}, w_{k \ell}$ we see that the size of these subsets is bounded by a sum over $u_k$ and $t_{k \ell}$. If $a_0 a_1 = 1,-1,2,-2 \in \Q^{\times}/ \Q^{\times 2}$ then $t_{k \ell} = 1$ for ${k, \ell} \neq \{0, 1\}, \{2, 3\}$. We thus get an upper bound for the size of the first subset of
\begin{equation}
    \ll \mathop{\sum_{t_{01}, t_{23}} \sum_{u_k}}_{|A_k| u_k^2 \prod_{\ell \neq k} t_{k \ell} \leq T} 1 \ll \frac{T^2}{|A_0 A_1 A_2 A_3|^{\frac{1}{2}}} \sum_{t_{01}, t_{23} \leq T} \frac{1}{t_{01} t_{23}} \ll \frac{T^2(\log T)^2}{|A_0 A_1 A_2 A_3|^{\frac{1}{2}}}.
    \label{Contribution of a_k a_l is almost a square}
\end{equation}

The size of the second subset is bounded by the sum
$\sum_{{A_k \prod_{\ell \neq k} t_{k \ell} \leq T}} 1.$ By applying Lemma~\ref{Lemma that sums can be approximated by integrals for downward-closed sets} this sum can be bounded by the integral
\begin{equation*}
    2^6\int_{t_{k \ell} \geq \frac{1}{2}}^{|A_k| \prod_{\ell \neq k} t_{k \ell} \leq T} \prod_{\{k, \ell\}} d t_{k \ell}. 
\end{equation*}

Make the change of variables $x_k = t_{0 k}, y_k = \prod_{\ell \neq k} t_{k \ell}$ for $k = 1,2,3$. The inverse is given by $t_{k \ell} = \sqrt{y_k y_\ell x_m /y_m x_k x_\ell}$ for $\{k, \ell, m\} = \{1,2,3\}$. The associated Jacobian determinant is
\begin{equation*}
    \begin{vmatrix}
    \frac{1}{2}\sqrt{\frac{y_2 x_3}{y_1 y_3 x_1 x_2}} & \frac{1}{2}\sqrt{\frac{y_1 x_3}{y_2 y_3 x_1 x_2}} & -\frac{1}{2}\sqrt{\frac{y_1 y_2 x_3}{y_3^3 x_1 x_2}} \\
    \frac{1}{2}\sqrt{\frac{y_3 x_2}{y_1 y_2 x_1 x_3}} & -\frac{1}{2}\sqrt{\frac{y_1 y_3 x_2}{y_2^3 x_1 x_3}} & \frac{1}{2}\sqrt{\frac{y_1 x_2}{y_3 y_2 x_1 x_3 }} \\
    -\frac{1}{2}\sqrt{\frac{y_2 y_3 x_1}{y_1^3 y_3 x_2 x_3}} & \frac{1}{2}\sqrt{\frac{y_3 x_1}{y_1 y_2 x_2 x_3}} & \frac{1}{2}\sqrt{\frac{y_2 x_1}{y_1 y_3 x_2 x_3}}
    \end{vmatrix} = - \frac{1}{2 \sqrt{ x_1 x_2 x_3  y_1 y_2 y_3}}.
\end{equation*}

Enlarging the integral we get
\begin{equation}
    \begin{split}
        \sum_{{A_k \prod_{\ell \neq k} t_{k \ell} \leq T}} 1 &\ll \int_{x_k \geq \frac{1}{2}, y_k \geq \frac{1}{8}}^{x_1 x_2 x_3 \leq T/A_0, y_k \leq T/A_k}  \frac{1}{\sqrt{x_1 x_2 x_3 y_1 y_2 y_3}} \prod_k d x_k dy_k \\
        &\ll \frac{T^2}{|A_0 A_1 A_2 A_3|^{\frac{1}{2}}}\int_{x_1, x_2 \geq \frac{1}{2}}^{x_1, x_2 \leq 4T} \frac{1}{x_1 x_2} d x_1 d x_2 \ll \frac{T^2(  \log T)^2}{|\theta_{\mathbf{A}}|^{\frac{1}{2}}}.
    \end{split}
    \label{Integral for the t region.}
\end{equation}
\end{proof}

We will require the following bounds, they will be proven in \S \ref{Subsection for error terms}.
\begin{lemma}
If $\theta_{\mathbf{A}} \not \in \Q^{\times}$ then for all $\varepsilon, C > 0, T > 2$ and $i \neq j \in \{0,1,2,3\}$
\begin{align}
    \# \{ (\mathbf{u}, \mathbf{v}, \mathbf{w}) \in N^{\emph{loc}}_{\mathbf{A}, \mathbf{M}}(T) : w_{k \ell} \leq (\log T^2)^C \emph{ for all } \{k, \ell\} \} &\ll_{\varepsilon, C} \frac{T^2 (\log T)^{5 M_{\mathbf{A}}}}{|\theta_{\mathbf{A}}|^{\frac{1}{2} - \varepsilon}} 
    \label{Contribtution from w_kl small}\\
    \# \{ (\mathbf{u}, \mathbf{v}, \mathbf{w}) \in N^{\emph{loc}}_{\mathbf{A}, \mathbf{M}}(T) : u_{k} \leq (\log T^2)^C \emph{ for } k \neq i,j\} &\ll_{\varepsilon, C} \frac{T^2 (\log T)^{5 M_{\mathbf{A}}}}{|\theta_{\mathbf{A}}|^{\frac{1}{2} - \varepsilon}}.
    \label{Contribution from most u_k small}
\end{align}
\label{Upper bounds for when either u_k or w_kl is small}
\end{lemma}
We can now show that most of the surfaces with a Brauer-Manin obstruction have a Brauer-Manin obstruction induced by $\mathcal{A}$.
\begin{lemma} If $\theta_{\mathbf{A}} \not \in \Q^{\times 2}$, then we have for $T > 2$
\begin{equation*}
      \# N^{\Br}_{\mathbf{A}, \mathbf{M}}(T) = \# N^{\mathcal{A}}_{\mathbf{A}, \mathbf{M}}(T) + O_{\varepsilon}\Big(\frac{T^2 (\log T)^{5 M_{\mathbf{A}}}}{|\theta_{\mathbf{A}}|^{\frac{1}{2} - \varepsilon}}\Big).
\end{equation*}
\label{Lemma to show that N^Br is asymptotically equal to N^A}
\end{lemma}
\begin{proof}
The difference between these two counts is bounded by the amount of surfaces in $N^{\Br}_{\mathbf{A}, \mathbf{M}}(T)$ such that $\mathcal{A}$ does not generate $\Br(X_{\mathbf{a}})/\Br(\Q)$. This is bounded by the amount of these surfaces for which either $\mathcal{A} \in \Br(\Q)$ or for which $\Br(X_{\mathbf{a}})/\Br(\Q) \not \cong \Z/ 2 \Z$. In the second case there exist $i \neq j$ such that $a_i a_j = 1,-1,2,-2 \in \Q^\times/\Q^{\times 2}$ by \cite[Appendix A]{bright2002computations}. The contribution of this second case can thus be bounded as in the first part of the proof of Lemma \ref{Bound for N^Br when theta is a square} by \eqref{Contribution of a_k a_l is almost a square}.

A surface such that $\mathcal{A} \in \Br(\Q)$ has by Proposition \ref{Computation invariant if p divides 2 coefficients} the property that $w_{k \ell} = 1$ for all $\{k , \ell\}$. This contribution is thus bounded by the size of the subset of $N_{\mathbf{A}, \mathbf{M}}(T)$ defined by $w_{k \ell} = 1$. This is equal to the subset of $N^{\text{loc}}_{\mathbf{A}, \mathbf{M}}(T)$ defined by $w_{k \ell} = 1$. This is a subset of the left-hand side of \eqref{Contribtution from w_kl small}.
\end{proof}

Note that all for all places $v \in S \cup \{p \mid \theta_{\mathbf{A}}\} \cup \{ \infty\}$ the surfaces contained in $N_{\mathbf{A}, \mathbf{M}}(T)$ are equivalent over $\Q_v$. For $v \neq 2, \infty$ this follows from Hensel's lemma. For $v = 2$ it follows from the fact that $(\Z/2^{n+2}\Z)^\times \cong \Z/2 \Z \times \Z/2^n \Z$ so the map $\Z_2^\times/\Z_2^{\times 4} \to (\Z/ 16 \Z)^\times$ is an isomorphism. It is true for $v = \infty$ because $u_k, v_{k \ell}, w_{k \ell} > 0$. The following condition is thus either true for all these surfaces, or for none of them.
\begin{equation}
    \text{For all places } v \in S \cup \{p \mid \theta_{\mathbf{A}}\} \cup \{ \infty\},  X_{\mathbf{a}}(\Q_v) \neq \emptyset, \text{inv}_v(\mathcal{A}(\cdot)) \text{ is constant}.
\label{Condition locally not prolific}
\end{equation}

If this condition fails then there is certainly no Brauer-Manin obstruction induced by $\mathcal{A}$ so $N^{\mathcal{A}}_{\mathbf{A}, \mathbf{M}}(T) = 0$. Let $\eta(\mathbf{A}, \mathbf{M})$ be the indicator function of condition \eqref{Condition locally not prolific}. \begin{lemma}
If $\theta_{\mathbf{A}} \not \in \Q^{\times 2}$ then for $T > 2$
\begin{equation*}
    \# N^{\mathcal{A}}_{\mathbf{A}, \mathbf{M}}(T) = \frac{\eta(\mathbf{A}, \mathbf{M})}{2} \# N^{\emph{loc}}_{\mathbf{A}, \mathbf{M}}(T) + O_{\varepsilon}\Big(\frac{T^2 (\log T)^{5 M_{\mathbf{A}}}}{|\theta_{\mathbf{A}}|^{\frac{1}{2} - \varepsilon}}\Big).
\end{equation*}
\label{Lemma to show that half of locally everywhere not prolific have a Brauer-Manin obstruction}
\end{lemma}
\begin{remark}
This lemma can be interpreted as saying that $50 \%$ percent of the surfaces for which $\mathcal{A}$ has locally everywhere a constant evaluation, i.e. is not prolific, have a Brauer-Manin obstruction induced by $\mathcal{A}$.
\end{remark}
\begin{proof}
If $\eta(\mathbf{A}, \mathbf{M}) = 0$, then the left-hand side is $0$. We may thus assume that condition \eqref{Condition locally not prolific} is satisfied.

Note that $N^{\mathcal{A}}_{\mathbf{A}, \mathbf{M}}(T) \subset N^{\text{loc}}_{\mathbf{A}, \mathbf{M}}(T)$ because of Proposition \ref{Prop if p divides 2 coefficients}.

Let $\mathbf{u}, \mathbf{v}, \mathbf{w}$ be a triple of tuples in $N^{\text{loc}}_{\mathbf{A}, \mathbf{M}}(T)$. The associated surface always has $\Q_p$-points for $p \mid m_\mathbf{A}$ by condition \eqref{Condition locally not prolific}. That it is locally soluble at the other primes means by Lemma~\ref{Local solubility} that for all primes $p \mid v_{k \ell}w_{k \ell}$ one of $-\frac{a_k}{a_\ell}, -\frac{a_m}{a_n}$ is a fourth power in $\Q_p$. For a prime $p \mid w_{k \ell}$ at most one of these can be true because $\theta_{\mathbf{a}}\not \in \Q_p^2$. For every such tuple there thus exists a unique factorization $w_{k \ell} = w^L_{k \ell} w^R_{k \ell}$ ($L$ stands for left and $R$ for right) where
\begin{align*}
 &w^L_{k \ell} =  \prod_{\substack{p \mid w_{k \ell} \\ -\frac{ a_k}{ a_\ell} \in \Q_p^{\times 4}}} p,  &w^R_{k \ell} =  \prod_{\substack{p \mid w_{k \ell}\\ -\frac{ a_m}{ a_n} \in \Q_p^{\times 4}}} p.
\end{align*}

Note that for all primes $p \mid w^L_{k \ell}$ we have $p \equiv 3 \pmod{4}$, so since $-\frac{a_k}{a_\ell} \in \Z_p$ the condition $-\frac{ a_k}{ a_\ell} \in \Q_p^{\times 4}$ is equivalent to $-\frac{a_k}{a_\ell} \in \Z_p^{\times 2}$. Analogously for $w^R_{k\ell}$. In particular, this condition does not depend on the $u_i$.

For every odd prime $p$ fix an injection $\psi_p:\Z_p^{\times}/ \Z_p^{\times 4} \to \Z/4 \Z$. Denote $\bm{\zeta} = (\zeta^L, \zeta^R, \xi^L, \xi^R) \in (\Z/4 \Z)^4$. We can then uniquely factorize $v_{k \ell} = \prod_{\bm{\zeta} \in (\Z/4 \Z)^4} v_{k \ell}^{\bm{\zeta}}$, where
\begin{equation*}
    v_{k \ell}^{\bm{\zeta}} \quad = \quad \quad \quad \quad  \prod_{\mathclap{\substack{p \mid v_{k \ell} \\ \psi_p(u_{k}/u_{\ell}) = \zeta^L, \; \psi_p(u_{m}/u_{n}) = \zeta^R \\ \psi_p(t_{k m} /t_{\ell n}) = \xi^L, \; \psi_p(t_{\ell m}/t_{k n}) = \xi^R }}} \quad p
\end{equation*}

Note now that since 
\begin{align}
    &-\frac{a_{k}}{a_{\ell}} = - \frac{A_k}{A_\ell} \frac{u_k^2}{u_\ell^2} \frac{t_{k m} t_{k n}}{t_{\ell m} t_{\ell n}} & -\frac{a_m}{a_n} = - \frac{A_m}{A_n} \frac{u_m^2}{u_n^2} \frac{t_{k m} t_{\ell m}}{t_{k n} t_{\ell n}}
    \label{Writing -a_k/a_l out}
\end{align}
there exists a subset $\Omega_p \subset (\Z/ 4 \Z)^4$ for every prime $p$ with the following property. The surface $X_{\mathbf{a}}$ has a $\Q_p$-point for $p \mid v^{\bm{\zeta}}_{k \ell}$ if and only if $\bm{\zeta} \in \Omega_p$. The point of this factorization is that the conditions induced by $v_{k \ell}^{\bm{\zeta}}$ on $t_{i j}$ and on $u_i$ are independent of each other.

To summarize, we have constructed a bijection of $N^{\text{loc}}_{\mathbf{A}, \mathbf{M}}(T)$ with the set $N^{\text{loc}'}_{\mathbf{A}, \mathbf{M}}(T)$
\begin{equation}
    \left\{\begin{array}{l}
            (\mathbf{u}, \mathbf{v}^{\bm{\zeta}}, \mathbf{w}^L, \mathbf{w}^R) \\
             \mathbf{u} \in \N^4  \\
             \mathbf{v}^{\bm{\zeta}} \in (\N^{\mathcal{P}_2(4)})^{(\Z/4\Z)^4} \\
             \mathbf{w}^L, \mathbf{w}^R \in \N^{\mathcal{P}_2(4)}
        \end{array} : \begin{array}{l}
            |A_k u_k^2 \prod_{\{k, \ell\}} \prod_{\bm{\zeta} \in (\Z/4\Z)^4} v^{\bm{\zeta}}_{k \ell} w^L_{k \ell} w^R_{k \ell}| \leq T, \\ \mu(m_\mathbf{A} \prod_{k} u_k \prod_{\ell \neq k} \prod_{\bm{\zeta} \in (\Z/4\Z)^4} v^{\bm{\zeta}}_{k \ell} w^L_{k \ell} w^R_{k \ell})^2 = 1, \\
            p \mid  v^{\bm{\zeta}}_{k \ell} \Rightarrow \theta_{\mathbf{A}} \in \Q_p^{\times 2}, \bm{\zeta} \in \Omega_p  \text{ and } \psi_p(\frac{u_{k}}{u_{\ell}}) = \zeta^L, \\ \psi_p(\frac{u_{m}}{u_{n}}) = \zeta^R, \psi_p(\frac{t_{k m}}{t_{\ell n}}) = \xi^L, \psi_p(\frac{t_{\ell m}}{t_{k n}}) = \xi^R, \\
            p \mid w^L_{k \ell} \Rightarrow \theta_{\mathbf{A}} \not \in \Q_p^{\times 2}, - \frac{a_k}{a_\ell} \in \Z_p^{\times 2}, \\
            p \mid w^R_{k \ell} \Rightarrow \theta_{\mathbf{A}} \not \in \Q_p^{\times 2}, - \frac{a_m}{a_n} \in \Z_p^{\times 2}, \\
            u_k^2 \prod_{\ell \neq k} \prod_{\bm{\zeta} \in (\Z/4\Z)^4} v^{\bm{\zeta}}_{k \ell} w^L_{k \ell} w^R_{k \ell} \Mod{8 m_{\mathbf{A}}} \in M_k \\
            \end{array} \right\}
\label{An expansion of the set N^loc in which the conditions u_k and w_k l are separated}
\end{equation}

For each triple of tuples $\mathbf{v}^{\bm{\zeta}}, \mathbf{w}^L, \mathbf{w}^R$ choose a tuple $\mathbf{u}$ counted by $N^{\text{loc}}_{\mathbf{A}, \mathbf{M}}(T)$ and choose a normalized $\mathcal{\mathcal{A}}$ for the associated surface. The invariant map of $\mathcal{A}$ is constant at all places. For the places $v|m_{\mathbf{A}} \infty$ this is by condition \eqref{Condition locally not prolific}. For primes $p \mid v_{k \ell}, w_{k \ell}$ this follows from Proposition \ref{Prop if p divides 2 coefficients}. For other primes this is a consequence of Lemma~\ref{Lemma if p divides at most one coefficient}.

Consider a different tuple $\mathbf{u}'$. We have an algebra $\mathcal{A}_{\mathbf{u}'^2 \mathbf{u}^{-2}}$ on the corresponding surface. This algebra is $p$-normalized for all primes $p \nmid m_{\mathbf{A}}, v_{k \ell}, w_{k \ell}^L, w_{k \ell}^R$ by Lemma~\ref{A stays p-normalized} so for these primes $\text{inv}_p(\mathcal{A}(\cdot)) = \text{inv}_p(\mathcal{A}_{\mathbf{u}'^2 \mathbf{u}^{-2}}(\cdot)) = 0$ by Lemma~\ref{Lemma if p divides at most one coefficient}. The same is true for $p \mid v_{k \ell}$ by Proposition \ref{Prop if p divides 2 coefficients}. Also, for places $p \mid m_{\mathbf{A}} \infty$ the formula \eqref{Invariant stays the same for bad primes} implies that $\text{inv}_p(\mathcal{A}(\cdot)) = \text{inv}_p(\mathcal{A}_{\mathbf{u}'^2 \mathbf{u}^{-2}}(\cdot))$. For $p \mid w_{k \ell}^L, w_{k \ell}^R$ we again use Proposition \ref{Prop if p divides 2 coefficients} to find that 
\begin{equation*}
    \begin{split}
        &\text{inv}_p(\mathcal{A}_{\mathbf{u}'^2 \mathbf{u}^{-2}}(\cdot)) = \text{inv}_p(\mathcal{A}(\cdot)) + \frac{\Jacobi{u'_k u_k u'_\ell u_\ell}{p} - 1}{4} \text{ if } p \mid  w_{k \ell}^L\\
        &\text{inv}_p(\mathcal{A}_{\mathbf{u}'^2 \mathbf{u}^{-2}}(\cdot)) = \text{inv}_p(\mathcal{A}(\cdot)) + \frac{\Jacobi{u'_m u_m u'_n u_n}{p} - 1}{4}
        \text{ if } p \mid  w_{k \ell}^R.
    \end{split}
\end{equation*}
Summing this over all $p$ we find that 
\begin{equation*}
    \sum_p \text{inv}_p(\mathcal{A}_{\mathbf{u}'^2 \mathbf{u}^{-2}}(\cdot)) = \sum_p \text{inv}_p(\mathcal{A}(\cdot)) + \frac{\prod_{\{k ,\ell\}}\Jacobi{u'_k u_k u'_\ell u_\ell}{w_{k \ell}^L}\Jacobi{u'_m u_m u'_n u_n}{w_{k \ell}^R}- 1}{4}.
\end{equation*}

There thus exists a $\delta := \delta(\mathbf{A}, \mathbf{M}, \mathbf{v}^{\bm{\zeta}}, \mathbf{w}^L, \mathbf{w}^R) \in \{1,-1\}$, such that the surface $X_{\mathbf{a}}$ has a Brauer-Manin obstruction if and only if 
\begin{equation*}
    \prod_{\{k, \ell\} \in \mathcal{P}_2(4)} (\frac{u_k u_\ell}{w^L_{k \ell}}) (\frac{u_m u_n}{w^R_{k \ell}}) = \delta.
\end{equation*}
The indicator function of having a Brauer-Manin obstruction is then given by
\begin{equation*}
    \frac{1}{2} - \frac{\delta}{2} \prod_{\{k, \ell\} \in \mathcal{P}_2(4)} (\frac{u_k u_\ell}{w^L_{k \ell}}) (\frac{u_m u_n}{w^L_{k \ell}}).
\end{equation*}
Hence, to prove the lemma it suffices to bound
\begin{equation*}
    \mathop{\sum_{\mathbf{v}^{\zeta}} \sum_{\mathbf{w}^L} \sum_{\mathbf{w}^R} \sum_{\mathbf{u}}}_{(\mathbf{u}, \mathbf{v}^{\zeta}, \mathbf{w}^L, \mathbf{w}^R) \in N^{\text{loc}'}_{\mathbf{A}, \mathbf{M}}(T)} \delta \prod_{\{k, \ell \}} (\frac{u_k u_\ell}{w^L_{k \ell}}) (\frac{u_m u_n}{w^R_{k \ell}}).
\end{equation*}
For each possible $4$-tuple $\mathbf{U}$ of cosets of $(\Z/8 m_{\mathbf{A}} \Z)^{\times 4}$ in $(\Z/8 m_{\mathbf{A}} \Z)^{\times}$ consider the subsum 
\begin{equation*}
    \mathop{\sum_{\mathbf{v}^{\zeta}} \sum_{\mathbf{w}^L} \sum_{\mathbf{w}^R} \sum_{\mathbf{u}}}_{\substack{(\mathbf{u}, \mathbf{v}^{\zeta}, \mathbf{w}^L, \mathbf{w}^R) \in N^{\text{loc}'}_{\mathbf{A}, \mathbf{M}}(T) \\ \mathbf{u} \in \mathbf{U}}} \delta \prod_{\{k, \ell \}} (\frac{u_k u_\ell}{w^L_{k \ell}}) (\frac{u_m u_n}{w^R_{k \ell}}).
\end{equation*}
We will bound this using Theorem~\ref{Theorem on linked variables} with $q_{\text{osc}} = O(1),  A = 6$. As can be seen in the definion \eqref{An expansion of the set N^loc in which the conditions u_k and w_k l are separated} the only conditions induced by $N^{\text{loc}'}_{\mathbf{A}, \mathbf{M}}(T)$ and $\delta$ for which $\mathbf{w}^L, \mathbf{w}^R$ interact with $\mathbf{u}$ is the coprimality condition and the modulo condition coming from $\mathbf{M}$. The coprimality condition is also contained in the Jacobi symbol and the interaction coming from $\mathbf{M}$ has been removed by the presence of $\mathbf{U}$. The variable $w_{k \ell}^L$, respectively $w_{k \ell}^R$, can thus be linked to $u_k$ and $u_\ell$, respectively to $u_m$ and $u_n$, by Lemma \ref{Examples of oscillating bilinear characters}. 

The downward-closed set we are summing over is defined by the inequalities
\begin{equation*}
    |A_k|u_k^2 \prod_{ \ell \neq k} (\prod_{\bm{\zeta}} v_{k \ell}^{\bm{\zeta}}) w^L_{k \ell} w^R_{k \ell} \leq T.
\end{equation*}
Taking the product of these inequalities shows that $L_U \leq |\theta_{\mathbf{A}}|^{-\frac{1}{2}} T^2$ which gives an error term in Theorem \ref{Theorem on linked variables} of $\ll |\theta_{\mathbf{A}}|^{- \frac{1}{2}} T^2$. To get the total error term we sum over the possible $\mathbf{U}$, there are $\ll \tau_4(\theta_{\mathbf{A}})^4 \ll_{\varepsilon} |\theta_{\mathbf{A}}|^{\varepsilon}$ of them due to the divisor bound.

After applying the theorem we can get rid of $\delta$ and the Jacobi symbols with a trivial bound. We can then sum back over the possible $\mathbf{U}$ and rewrite $v_{k \ell} = \prod_{\bm{\zeta} \in (\Z/4 \Z)^4} v_{k \ell}^{\bm{\zeta}}$. What remains is a sum which counts the elements of a certain subset of $N^{\text{loc}}_{\mathbf{A}, \mathbf{M}}(T)$. This subset is contained in the union of the two sets which are defined respectively by the following two conditions, the $C$ comes from the application of Theorem~\ref{Theorem on linked variables}.
\begin{enumerate}
    \item The first is when $u_k \leq (\log |\theta_{\mathbf{A}}|^{- \frac{1}{2}} T^2)^C$ for at least two $k$. In this case we can rewrite $w_{k \ell} = w^L_{k \ell} w^R_{k \ell}$ and get a subset of the left-hand side of \eqref{Contribution from most u_k small}.
    \item The second is when $u_k > (\log |\theta_{\mathbf{A}}|^{- \frac{1}{2}} T^2)^C$ for at least three $k$. We must then have $w^L_{k \ell},w^R_{k \ell} \leq (\log |\theta_{\mathbf{A}}|^{- \frac{1}{2}} T^2)^C$ for all $\{k, \ell\}$ by Theorem \ref{Theorem on linked variables}. We may rewrite $w_{k \ell} = w^L_{k \ell} w^R_{k \ell}$ and attempt to bound the size of the larger set defined by $w_{k \ell} \leq (\log T^2)^{2C}$. This set is exactly the left-hand side of \eqref{Contribtution from w_kl small}.
\end{enumerate}
\end{proof}

In the following sections we will prove the following lemma
\begin{lemma}
If $\theta_{\mathbf{A}} \not \in \Q^{\times 2}$ then there exists a constant $0 < Q_{\mathbf{A}} \ll_{\varepsilon} |\theta_{\mathbf{A}}|^{\varepsilon}$ for all $\varepsilon > 0$ such that for $T > 2$ we have
\begin{equation*}
    \# N^{\emph{loc}}_{\mathbf{A}, \mathbf{M}}(T) = \frac{T^2 (\log T)^{6 M_{\mathbf{A}}} }{|\theta_{\mathbf{A}}|^{\frac{1}{2}}}( Q_{\mathbf{A}} + O_{ \varepsilon}(|\theta_{\mathbf{A}}|^{\varepsilon}(\log T)^{- M_{\mathbf{A}}}).
\end{equation*}
\label{Lemma for asymptotic formula of N^loc}
\end{lemma}

We can now prove the main result. 
\begin{proof}[Proof of Theorem~\ref{Main theorem split into cases}]
Lemmas \ref{Lemma for splitting N^Br into subsums} and \ref{Bound for N^Br when theta is a square} together imply that 
\begin{equation*}
        \# N_{= \square}^{\Br}(T) = \sum_{\substack{(\mathbf{A}, \mathbf{M}) \in \Psi(T) \\ \theta_{\mathbf{A}} \in \Q^{\times 2}}} O\Big(\frac{T^2 (\log T)^2}{|\theta_{\mathbf{A}}|^{\frac{1}{2}}}\Big).
\end{equation*}
Combining the Lemmas \ref{Lemma for splitting N^Br into subsums}, \ref{Lemma to show that N^Br is asymptotically equal to N^A}, \ref{Lemma to show that half of locally everywhere not prolific have a Brauer-Manin obstruction} and \ref{Lemma for asymptotic formula of N^loc} and recalling that $M_{\mathbf{A}}$ was defined to be $\frac{11}{16}$ if $\theta_{\mathbf{A}} \in - \Q^{\times 2}$ and $\frac{15}{32}$ if $\theta_{\mathbf{A}} \not \in \pm \Q^{\times 2}$ shows that
\begin{equation*}
    \begin{split}
        \# N_{= -\square}^{\Br}(T) &= T^2 (\log T)^{\frac{33}{8}}\sum_{\substack{(\mathbf{A}, \mathbf{M}) \in \Psi(T) \\ \theta_{\mathbf{A}} \in -\Q^{\times 2}}} \frac{\eta(\mathbf{A}, \mathbf{M})}{2 |\theta_{\mathbf{A}}|^{\frac{1}{2}}}\left(Q_{\mathbf{A}} + O_{\varepsilon}\Big(\frac{|\theta_{\mathbf{A}}|^{\varepsilon}}{(\log T)^{ \frac{11}{16}}}\Big)\right)   \\
        \# N_{\neq \pm \square}^{\Br}(T) &= T^2(\log T)^{\frac{45}{16}}\sum_{\substack{(\mathbf{A}, \mathbf{M}) \in \Psi(T) \\ \theta_{\mathbf{A}} \not \in \pm \Q^{\times 2}}} \frac{\eta(\mathbf{A}, \mathbf{M})}{2 |\theta_{\mathbf{A}}|^{\frac{1}{2}}}\left(Q_{\mathbf{A}} + O_{\varepsilon}\Big(\frac{|\theta_{\mathbf{A}}|^{\varepsilon}}{(\log T)^{ \frac{15}{32}}}\Big)\right).
    \end{split}
\end{equation*}
To prove Theorem~\ref{Main theorem split into cases} it remains to show that the above three sums converge, and to a non-zero constant for the latter two. To see that the constant is non-zero it suffices to give a single example of a surface in which $\mathcal{A}$ is locally not-profilic at all places, for example because it induces a Brauer-Manin obstruction. Because in that case the corresponding $\eta(\mathbf{A}, \mathbf{M})$ is non-zero. For $\theta_{\mathbf{a}}\not \in \pm \Q^{\times 2}$ such an example is \cite[Proposition 3.3]{bright2011brauer} and for $\theta_{\mathbf{a}}\in - \Q^{\times 2}$ Lemma~\ref{Example which is locally not prolific with theta = -1} provides the required example.

We show that the sums converge. By taking upper bounds it suffices to bound
\begin{equation*}
    \sum_{\substack{(\mathbf{A}, \mathbf{M}) \in \Psi \\ |\theta_{\mathbf{A}}| \geq T}} |\theta_{\mathbf{A}}|^{-\frac{1}{2} + \varepsilon}
\end{equation*}
for sufficiently small $\varepsilon$. The amount of $\mathbf{M}$ associated to each $\mathbf{A}$ is $\ll \tau_4(\theta_{\mathbf{A}})^4 \ll_{\varepsilon} |\theta_{\mathbf{A}}|^{\varepsilon}$ by the divisor bound, so it suffices to bound
\begin{equation*}
    \sum_{\substack{\mathbf{A} \in \Phi \\ |\theta_{\mathbf{A}}| > T}} |\theta_{\mathbf{A}}|^{-\frac{1}{2} + 2\varepsilon}.
\end{equation*}
All the tuples $\mathbf{A} \in \Phi$ are such that $\theta_{\mathbf{A}}$ can be written (non-uniquely) as $g h_3^3 h_4^4 h_5^5$ such that for all $p \mid g$ we have $p \in S$ and $v_p(g) \leq 2$. There are thus only a finite amount of possibilities for $g$. For every $\theta$ there exist only $\tau_4(\theta) \ll_{\varepsilon} |\theta|^{\varepsilon}$ tuples $\mathbf{A} \in \Phi$ such that $\theta_{\mathbf{A}} = \theta$, so we get an upper bound
\begin{equation*}
    \ll \sum_{\substack{g, h_3, h_4, h_5 \\ |g h_3^3 h_4^4 h_5^5| > T}} |g h_3^3 h_4^4 h_5^5|^{-\frac{1}{2} + 3\varepsilon} \ll_{\varepsilon} T^{-\frac{1}{6} + 3\varepsilon} \sum_{h_4, h_5} h_4^{-2 + 12 \varepsilon + \frac{2}{3} - 12 \varepsilon} h_5^{- \frac{5}{2} + 15 \varepsilon + \frac{5}{6} - 15 \varepsilon} \ll T^{-\frac{1}{6} + 3\varepsilon}.
\end{equation*}
By taking $\varepsilon < 1/18$ we are done.
\end{proof}
\subsection{A sum of linked variables}
The goal of this subsection will be to simplify $\#N^{\text{loc}}_{\mathbf{A}, \mathbf{M}}(T)$. For simplicity we will assume that $\theta_{\mathbf{A}} \not \in \Q^{\times 2}$. This will also allow us to deal with the error terms in the lemmas of the previous subsection. For this fix a downward-closed subset $U$ of the set defined by the inequalities.
\begin{equation}
    |A_k| u_k^2 \prod_{\ell \neq k} v_{k \ell} w_{k \ell} \leq T.
    \label{Definition of relevant region}
\end{equation}
By taking the product of the inequalities we see that $L_U \leq \frac{T^2}{|A_0 A_1 A_2 A_3|^{\frac{1}{2}}}$. Define $N^{\text{loc}, U}_{\mathbf{A}, \mathbf{M}}(T) := N^{\text{loc}}_{\mathbf{A}, \mathbf{M}}(T) \cap U$. For a first reading it will be simplest to assume that $U$ is defined by \eqref{Definition of relevant region} and thus that $N^{\text{loc}, U}_{\mathbf{A}, \mathbf{M}}(T) = N^{\text{loc}}_{\mathbf{A}, \mathbf{M}}(T)$.

Let
\begin{equation*}
    \Gamma_{\mathbf{A}} := (\Z/ 8 m_\mathbf{A})^{\times}/(\Z/ 8 m_\mathbf{A})^{\times 4}.
    \end{equation*}
The condition induced by $\mathbf{M}$ is then encoded in the standard way as a sum of characters
\begin{equation*}
    \frac{1}{|\Gamma_{\mathbf{A}}|^4}\prod_{k}\sum_{\chi_k \in \Gamma_{\mathbf{A}}^{\vee}} \overline{\chi}_k(M_k) \chi_k(u_k^2\prod_{\ell \neq k} v_{k \ell} w_{k \ell}).
\end{equation*}
Write $\bm{\chi} = (\chi_0, \chi_1, \chi_2, \chi_3)$. Substituting this into the main sum and switching the order of summation shows that 
\begin{equation}
    \#N^{\text{loc}, U}_{\mathbf{A}, \mathbf{M}}(T) = \frac{1}{|\Gamma_{\mathbf{A}}|^4} \sum_{\bm{\chi} \in (\Gamma_{\mathbf{A}}^{\vee})^4} \prod_{k} \overline{\chi}_k(M_k) S_{U}(\bm{\chi}, T) 
    \label{Formula for N loc as sum of S(chi, T)}
\end{equation}
where 
\begin{equation}
   S^{U}_{\mathbf{A}}(\bm{\chi}, T) := \mathop{\sum \sum \sum}_{(\mathbf{u}, \mathbf{v}, \mathbf{w}) \in \sqcup_{\mathbf{N}}N^{\text{loc}, U}_{\mathbf{A}, \mathbf{N}}(T)} \chi_k(u_k^2\prod_{\ell \neq k} v_{k \ell} w_{k \ell}).
   \label{Definition of S_U(chi,T)}
\end{equation}
If $N^{\text{loc}, U}_{\mathbf{A}, \mathbf{M}}(T) = N^{\text{loc}}_{\mathbf{A}, \mathbf{M}}(T)$ then we will just write $ S_{\mathbf{A}}(\bm{\chi}, T) := S^{U}_{\mathbf{A}}(\bm{\chi}, T)$. The goal of this section is to prove the following lemma.

\begin{lemma}
There exist frobenian multiplicative functions $\alpha(\cdot ; \theta_{\mathbf{A}}), \beta(\cdot, \theta_{\mathbf{A}})$ of conductor $\leq 16 |\theta_{\mathbf{A}}|^2$ such that for $T > 2$ and all $\mathbf{A}, \bm{\chi}, U$ we have
\begin{equation}
    \begin{split}
        S^{U}_{\mathbf{A}}(\bm{\chi}, T) =& \mathop{\sum \sum \sum}_{\substack{(\mathbf{u}, \mathbf{d}, \mathbf{w}) \in U \\ \mu(\prod_k u_k \prod_{\{k, \ell\}} d_{k \ell} w_{k \ell})^2 = 1}} \prod_k \chi_k(u_k)^2 \prod_{\{k, \ell\}}\chi_k \chi_\ell(w_{k \ell} d_{k \ell})\alpha(w_{k \ell}; \theta_{\mathbf{A}}) \beta( d_{k \ell}; \theta_{\mathbf{A}}) \\ &+ O_{\varepsilon}\left(\frac{T^2(\log T)^{5 M_{\mathbf{A}}}}{|\theta_{\mathbf{A}}|^{\frac{1}{2} - \varepsilon}}\right).
    \end{split}
    \label{Result after removing highly linked variables}
\end{equation}
The functions $\alpha(\cdot ; \theta_{\mathbf{A}}), \beta(\cdot, \theta_{\mathbf{A}})$ respectively have means $\frac{1}{2}, \frac{3}{16}$ if $\theta_{\mathbf{A}} \in - \Q^{\times 2}$ and means $\frac{1}{4}, \frac{7}{32}$ if $\theta_{\mathbf{A}} \not \in \pm \Q^{\times 2}$.
\label{Lemma which describes S_U as sum over alpha and beta}
\end{lemma}
Before proving this lemma we will do some preparations. To encode the condition that local points have to exist we will apply Lemma~\ref{Local solubility}. From now on assume that $k < \ell$ and $m < n$. If $p \nmid m_\mathbf{A}\prod_{\{k, \ell\}} v_{k \ell} w_{k \ell}$ then there are always local solutions. If $p \mid v_{k \ell} w_{k \ell}$ then there local solutions if and only if $-\frac{a_k}{a_\ell} \in \Z_p^{\times 4}$ or $-\frac{a_m}{a_n} \in \Z_p^{\times 4}$.

If $p \mid w_{k \ell}$ then this follows from the other conditions on $w_{k \ell}$. Because in this case $p \equiv 3 \pmod{4}$ so $\Z_p^{\times 4} = \Z_p^{\times 2}$. But also $ \theta_{\mathbf{a}}\not \in \Q_p^{\times 2}$ so $-\frac{a_k}{a_\ell} \not \equiv -\frac{a_m}{a_n} \mod \Z_p^{\times 2}$. Hence at least one of them must be a square. To encode the other conditions on $w_{k \ell}$ we will utilize the function $\alpha(w; \theta)$ which is the indicator function of the set 
\begin{equation*}
    \{ w : w \text{ square-free }, p \mid w \Rightarrow p \equiv 3 \Mod{4}, \theta \not \in \Q_p^{\times 2} \}.
\end{equation*}
This is a $\{2, p \mid \theta\}$-frobenian multiplicative function defined by $\Q(\sqrt{-1}, \sqrt{\theta_{\mathbf{a}}})$ and thus of conductor $\leq 16 |\theta|^2$. It has mean $\frac{1}{2}$ if $-\theta_{\mathbf{a}}\in \Q^{\times 2}$ and mean $\frac{1}{4}$ if $\theta_{\mathbf{a}}\not \in \pm \Q^{2 \times}$.

To encode the properties for $v_{k \ell}$ we will first introduce some functions. Given $x, y, z \in \Q$ define $\varpi(v;x, y, z)$ as the indicator function of the set
\begin{equation*}
    \{ v \in \Z : \mu(v)^2 = 1, v \text{ coprime with } x, y,z,  p \mid v \Rightarrow z \in \Z_p^{\times 2} \text{ and } x \text{ or } x y^2 z \in \Q_p^{\times 4} \}.
\end{equation*}
This is a multiplicative function in $v$ which encodes the condition on $v_{k \ell}$ as 
\begin{align*}
    v &= v_{k \ell} & x &= -\frac{a_k}{a_\ell} = -\frac{A_k v_{k m} w_{k m} v_{k n} w_{k n} u_k^2}{A_\ell v_{\ell m} w_{\ell m} v_{\ell n} w_{\ell n} u_\ell^2} \\
    y &= y_{k n}^{\ell m} := \frac{v_{\ell m} w_{\ell m} u_\ell u_m}{v_{k n} w_{k n} u_k u_n}  & z &= z_{k n}^{\ell m} := \frac{A_\ell A_m}{A_k A_n}
\end{align*} due to \eqref{Writing -a_k/a_l out}. We will think of $z$ as being constant. Using these functions we have the following formula for $S_U(\bm{\chi}, T)$
\begin{equation}
    \mathop{\sum_{\mathbf{v}} \sum_{\mathbf{w}} \sum_{\mathbf{u}}}_{\substack{(\mathbf{u}, \mathbf{v}, \mathbf{w}) \in U \\ \mu(\prod_k u_k \prod_{\{k, \ell\}} v_{k \ell} w_{k \ell})^2 = 1}} \prod_k \chi_k(u_k)^2 \prod_{k < \ell}\alpha(w_{k \ell} ; \theta_{\mathbf{A}}) \varpi(v_{k \ell} ; -\frac{a_k}{a_{\ell}}, y_{k n}^{\ell m}, z_{k n}^{\ell m}) \chi_k \chi_{\ell}(v_{k \ell} w_{k \ell}).
    \label{Description of S_U in terms of varpi}
\end{equation}

We will now give a description of $\varpi(v_{k \ell} ; -\frac{a_k}{a_{\ell}}, y_{k n}^{\ell m}, z_{k n}^{\ell m})$ as a convolution of simpler functions. Let $\beta(x;z), \gamma(x;z), \gamma^c(x; z), \delta(x;z)$ be the multiplicative functions which are $0$ if $2x$ is not square-free and such that for an odd prime $p$ they take the value described in the following table
\begin{center} 
\begin{tabular}{ | c | c | c | c | c  |}
\hline
& $\beta(p; z)$ & $\gamma(p;z)$  & $\gamma^c(p; z)$ & $\delta(p;z)$ \\ 
\hline
$p \equiv 1 \pmod{4}, z \in \Q_p^{\times 2} \setminus \Q_{p}^{\times 4}$ & $3/8$ & $1/8$ & $-1/8$ & $1/8$ \\ \hline
$p \equiv 1 \pmod{4}, z \in \Q_{p}^{\times 4}$                              &
$3/8$ & $-1/8$ & $1/8$ & $1/8$ \\ \hline
$p \equiv 3 \pmod{4}, z \in \Q_p^{\times 2} = \Q_p^{\times 4}$              &
$1/2$ & $0$ & $0$ & $0$ \\ \hline
Other $p$                                                                   &
$0$ & $0$ & $0$ & $0$ \\ \hline
Mean if $-z \in \Q^{\times 2}$                                              &
$3/16$ & $0$ & $0$ & $1/16$ \\ \hline
Mean if $\pm z \not \in \Q^{\times 2}$                                      &
$7/32$ & $0$ & $0$ & $1/32$ \\ \hline
\end{tabular}
\end{center}
These are $\{p : p = 2, \text{ or } v_p(z) \neq 0\}$-frobenian multiplicative functions of conductor at most the discriminant $\Q(i,\sqrt[4]{z})$ which is $\leq 2^{24}\prod_{v_p(z) \neq 0} p^6$. The functions $\beta(\cdot,z), \delta(\cdot, z)$ are defined by the extension $\Q(\sqrt{-1}, z)$ and thus have conductor at most $16 \prod_{v_p(z) \neq 0} p^2$. The means of these functions in the relevant cases are given in the table. The functions $\gamma$ and $\gamma^c$ are bounded by $\delta$.
\begin{lemma}
The following expression is equal to $\varpi(v;x, y,z) $
\begin{equation}
    \begin{split}
     \mu(v)^2 \sum_{d e f f' \Norm(\mathfrak{g})\Norm(\mathfrak{g}') = v}  &\beta(d; z) \beta(e; z) \gamma(f; z) \gamma(f' ; z) \delta(\Norm(\mathfrak{g}); z) \gamma^c(\Norm(\mathfrak{g}') ; z)   \\ &\Jacobi{x}{e} \Jacobi{y}{f} \Jacobi{x y}{f'}  \PowerRes{x}{\mathfrak{g}}{4} \PowerRes{x y^2}{\mathfrak{g}'}{4}.
    \end{split}
    \label{Alternative expression for varpi}
\end{equation}
Where $d, e,f,f'$ range over $\N$ and $\mathfrak{g},\mathfrak{g}' $ range over $I_{\Q[i]}$.
\end{lemma}
\begin{proof}
Both sides are multiplicative in $v$ and $0$ at higher prime powers so it suffices to check this equality at primes. This follows from a case analysis. If $p \equiv 1 \pmod{4}$ then it factors in $\Z[i]$ as $\mathfrak{p} \overline{\mathfrak{p}}$. The four characters of $\F_p^{\times}/\F_p^{\times 4}$ are $1, \Jacobi{\cdot}{p}, \PowerRes{\cdot}{\mathfrak{p}}{4}, \PowerRes{\cdot}{\overline{\mathfrak{p}}}{4}$ and for any fixed choice of $z \in \Z_p^{\times 2}/\Z_p^{\times 4}$ one can check that the expression \eqref{Alternative expression for varpi} for $p = v$ is exactly the representation of $\varpi(p; x,y,z)$ as a sum of characters. The same is true for $p \equiv 3$ but the characters for $\F_p^{\times}/\F_p^{\times 4}$ in that case are $1, \Jacobi{\cdot}{p}$.
\end{proof}

We can now prove Lemma~\ref{Lemma which describes S_U as sum over alpha and beta}. We remark that the frobenian multiplicative functions $\alpha(\cdot, \theta), \beta(\cdot, \theta)$ are the same ones that appear in the statement of Lemma~\ref{Lemma which describes S_U as sum over alpha and beta}.
\begin{proof}
Substituting \eqref{Alternative expression for varpi} in the sum \eqref{Description of S_U in terms of varpi} and switching the order of summation gives a sum of $46$ variables over a downward-closed set $V$. To describe this sum we will rename some of the variables and associate to each of them a factor. The sum is over the product of these factors, together with the condition that these variables are pairwise coprime and square-free.
\begin{enumerate}
    \item The $u_k$ variables are given a factor $\chi_k(u_k)^2$. This is bounded by $1$.
    \item The $w_{k \ell}$ have a factor $\alpha( w_{k \ell}; \theta) \chi_k \chi_{\ell}(w_{k \ell})$ which is bounded by $\alpha( w_{k \ell}; \theta)$.
    \item Denote the $d$ variable associated to $v_{k \ell}$ by $d_{k \ell}$. It is given a factor 
    \begin{equation*}
        \beta(d_{k \ell} ; z_{k n}^{\ell m}) \chi_k \chi_{\ell}(d_{k \ell}).
    \end{equation*} This is bounded by $\beta( w_{k \ell}; z_{k n}^{\ell m})$.
    \item The $e$ variable associated to $v_{k \ell}$ will be denoted by $e_{k \ell}$. It has a factor
    \begin{equation*}
        \beta(e_{k \ell}, z_{k n}^{\ell m}) \Jacobi{-A_k A_\ell t_{k m} t_{k n} t_{\ell m} t_{\ell n}}{e_{k \ell}} \chi_k \chi_{\ell}(e_{k \ell}).
    \end{equation*}
    This is bounded by $\beta( w_{k \ell}; z_{k n}^{\ell m})$.
    \item Denote the $f$ variable associated to $v_{k \ell}$ by $f_{k \ell}^{n m}$. The associated factor is
    \begin{equation*}
         \gamma(f_{k \ell}^{m n}, z_{kn}^{\ell m}) \Jacobi{ -A_k A_\ell u_k u_\ell u_m u_n t_{k n} t_{\ell m}  }{f_{k \ell}^{n m}} \chi_k \chi_{\ell}(f_{k \ell}^{m n}).
    \end{equation*}
    Note that $t_{k n}$ has the first index of the top and bottom of $f_{k \ell}^{n m}$ and $t_{\ell m}$ has the second index of the top and bottom. This is bounded by $\delta(f_{k \ell}^{n m}, z_{kn}^{\ell m})$.
    \item We will denote the $f'$ variable associated to $v_{k \ell}$ by $f_{k \ell}^{m n}$. Its factor is
    \begin{equation*}
         \gamma(f_{k \ell}^{m n}, z_{kn}^{\ell m}) \Jacobi{ -A_k A_\ell u_k u_\ell u_m u_n t_{k m}  t_{\ell n}}{f_{k \ell}^{m n}} \chi_k \chi_{\ell}(f_{k \ell}^{m n}).
    \end{equation*}
    Again note that $t_{k m}$ has the first index of the top and bottom of $f_{k \ell}^{m n}$ and $t_{\ell n}$ has the second index of the top and bottom. It is bounded by $\delta(f_{k \ell}^{m n}, z_{kn}^{\ell m})$
    \item We rename the $\mathfrak{g}$ variable associated to $v_{k \ell}$ by $\mathfrak{g}_{k \ell}$. It has a factor
    \begin{equation*}
        \delta(\Norm(\mathfrak{g}_{k \ell}); z_{kn}^{\ell m}) \PowerRes{-A_k t_{k m}  t_{k n} u_k^2}{\mathfrak{g}_{k \ell}}{4} \PowerRes{A_\ell t_{\ell m} t_{\ell n} u_\ell^2}{ \mathfrak{g}_{k \ell}}{4}^{-1} \chi_{k} \chi_{\ell}(\Norm(\mathfrak{g}_{k \ell})).
    \end{equation*}
    Its factor is bounded by $\delta(\Norm(\mathfrak{g}_{k \ell}); z_{kn}^{\ell m})$. 
    \item The $\mathfrak{g}'$ associated to $v_{k \ell}$ will be denoted $\mathfrak{g}_{k \ell}^{c}$, (the $c$ standing for complement).
    \begin{equation*}
        \gamma^c(\Norm(\mathfrak{g}^{c}_{k \ell}); z_{kn}^{\ell m}) \PowerRes{-A_k t_{k m}  t_{\ell m} u_m^2}{\mathfrak{g}_{k \ell}^c}{4} \PowerRes{A_\ell t_{k n} t_{\ell n} u_n^2}{ \mathfrak{g}_{k \ell}^c}{4}^{-1} \chi_{k} \chi_{\ell}(\Norm(\mathfrak{g}_{k \ell}^c)). 
    \end{equation*}
    This factor is bounded by $\delta(\Norm(\mathfrak{g}^{c}_{k \ell}); z_{kn}^{\ell m})$.
\end{enumerate}
We will now stop assuming that $k < \ell, m < n$ and write $d_{k \ell} = d_{\ell k}, f_{k \ell}^{m n} = f_{\ell k}^{n m}, \dots$. In this way we can treat $f_{k \ell}^{n m}$ and $f_{k \ell}^{m n}$ uniformly. With this relabeling the downward-closed set $V$ is equal to
\begin{equation*}
    \{ \mathbf{u}, \mathbf{d}, \mathbf{e}, \mathbf{f}, \mathbf{g}, \mathbf{g}^c : (\mathbf{u}, (e_{k \ell} d_{k \ell} f_{k \ell}^{m n} f_{k \ell}^{n m} \Norm(\mathfrak{g}_{k \ell}) \Norm(\mathfrak{g}^c_{k \ell}))_{\{k, \ell\}}, \mathbf{w}) \in U \}.
\end{equation*}
It is then clear that the length of $V$ is $L_V = L_U \leq |\theta_{\mathbf{A}}|^{-\frac{1}{2}}T^2$.

We will now apply Theorem~\ref{Theorem on linked variables} with $A = 6, q_{\text{osc}} = O(1), q_{\text{frob}} = O(1)|\theta_{\mathbf{A}}|^{O(1)} $. The resulting error term is $O_{\varepsilon}(|\theta_{\mathbf{A}}|^{-\frac{1}{2} + \varepsilon} T^2)$.
We use Lemma~\ref{Examples of oscillating bilinear characters} to show that the following pairs of variables can be linked. Note that the coprimality condition between the linked variables is contained in the Jacobi or power residue symbols.
\begin{enumerate}
    \item The variable $u_i$ is linked to $f_{k \ell}^{m n}$ and to $\mathfrak{g}_{k \ell}$ if $i = k, \ell$ and to $f_{k \ell}^{m n}$ and $\mathfrak{g}_{k \ell}^c$ if $i = m, n$.
    \item The variables $w_{i j}$ and $d_{i j}$ are both linked to $e_{k \ell}, \mathfrak{g}_{k \ell}$ and $\mathfrak{g}_{k \ell}^c$ if $\{i, j\} \neq \{k, \ell\}, \{m, n\}$. It is also linked to $f_{k \ell}^{m n}$ if $\{i, j\} = \{k, m\}, \{\ell, n\}$.
    \item The variable $e_{i j}$ is linked to $f_{k \ell}^{m n}$ if $\{i, j\} = \{k, n\}, \{\ell, m \}$. It is also linked to $\mathfrak{g}_{k \ell}, \mathfrak{g}_{k \ell}^c$ if $\{i, j\} \neq \{k, \ell\}, \{m, n\}$.
    \item The variable $f_{k \ell}^{m n}$ is linked to $f_{k m}^{n \ell}, f_{\ell n}^{m k}$, $f_{k n}^{\ell m}$ and $f_{\ell m}^{k n}$. It is also linked to $\mathfrak{g}_{ij}$ and $\mathfrak{g}_{ij}^c$ as long as $\{i, j\} \neq \{k, \ell \}, \{m, n\}$. 
    \item The variables $\mathfrak{g}_{i j}$ and $\mathfrak{g}_{i j}^c$ are both linked to $\mathfrak{g}_{k \ell}$ and $\mathfrak{g}_{k \ell}^c$ as long as $\{i, j\} \neq \{k, \ell\}, \{\ell, m\}$.
\end{enumerate}

All the variables are also frobenian. Indeed, if we fix all but one variable then the product of the factors is a product of the following quantities.
\begin{enumerate}
    \item A constant dependent on the other variables, whose absolute value is bounded by $1$.
    \item Coprimality conditions which will only increase the set of bad primes of the frobenian multiplicative function.
    \item Factors of the form $\Jacobi{n}{m} \Jacobi{m}{n}$. These are Hecke characters modulo $4 $ by quadratic reciprocity.
    \item A frobenian multiplicative function of conductor at most $O(|\theta_{\mathbf{A}}|^6)$ independent of the other variables.
    \item The factor $\chi^2_k$ or $\chi_k \chi_\ell$ which is a Dirichlet character modulo $8 m_{\mathbf{A}}$ independent of the other variables.
    \item A product of Jacobi or quartic residue characters of the same type as defined in Lemma~\ref{Examples of oscillating bilinear characters} corresponding to a linked variable $y$. The norm of the modulus of this Hecke character is $O(\Norm(y)^6)$ by that lemma.
\end{enumerate}
It is thus a frobenian multiplicative of conductor $O(1) |\theta_{\mathbf{A}}|^{O(1)}$ by Lemma~\ref{Products and convolutions of frobenian multiplicative functions are frobenian multiplicative}.

Assume that there is a linked variable $y \neq 1$. Then the mean of this frobenian multiplicative function is $0$ by Lemma~\ref{Mean after twisting by character}. Indeed, the product of the above factors except for the Hecke character corresponding to $y$ is frobenian multiplicative and defined by a field with discriminant coprime to $\Norm(y)$ by the coprimality conditions. The Hecke character corresponding to $y$ cannot be defined by such a field since $\Norm(y)$ is square-free and thus divides the norm of the conductor of the corresponding Hecke character by Lemma~\ref{Examples of oscillating bilinear characters}.

Let us now consider the case when one of the (frobenian) variables $x$ is not equal to $1$. By Theorem~\ref{Theorem on linked variables} we may assume that all the linked variables are $\leq e^{(2\log T)^{\varepsilon}}$. Note that $\alpha(\cdot; z), \beta(\cdot; z)$ and $\delta(\cdot; z)$ depend only on the class of $z$ in $\Q^{\times}/\Q^{\times 2}$. So we may replace the $z_{kn}^{\ell m}$ by $\theta_{\mathbf{A}}$. Using upper bounds, including to get rid of the dependence of $U$, we see that the contribution of this is at most
\begin{equation}
    \begin{split}
        \mathop{\sum_{\mathbf{w}}  \sum_{\mathbf{d}} \sum_{\mathbf{e}} \sum_{\mathbf{f}} \sum_{\mathbf{g}} \sum_{\mathbf{g}^c}  \sum_{\mathbf{u}}}_{\substack{|A_k| u_k^2 \prod_{\ell \neq k} t_{k \ell} \leq T \\ \Norm(y) \leq e^{(\log T^2)^{\varepsilon}} \text{ for } y \text{ linked to } x}}& \prod_{\{k, \ell\}}\alpha(w_{k \ell}; \theta_{\mathbf{A}}) \beta(d_{k \ell}; \theta_{\mathbf{A}}) \beta(e_{k \ell}; \theta_{\mathbf{A}}) \\ &\delta(\Norm(\mathfrak{g}_{k \ell}); \theta_{\mathbf{A}}) \delta(\Norm(\mathfrak{g}^c_{k \ell}); \theta_{\mathbf{A}}) \prod_{m,n \neq k, \ell} \delta(f_{k \ell}^{m n}; \theta_{\mathbf{A}}).
     \end{split}
     \label{Upper bound for sum of linked variables with one variable not equal to 1}
\end{equation}

We can sum over the $u_k$ and enlarge the sum to get an upper bound
\begin{equation*}
    \begin{split}
        \frac{T^2}{|\theta_{\mathbf{A}}|^{\frac{1}{2}}} \mathop{\sum_{\mathbf{w}}  \sum_{\mathbf{d}} \sum_{\mathbf{e}} \sum_{\mathbf{f}} \sum_{\mathbf{g}} \sum_{\mathbf{g}^c}}_{\substack{\Norm(y) \leq e^{(\log T^2)^{\varepsilon}} \text{ for } y \text{ linked to } x \\ \Norm(z) \leq T \text{ for } z \text{ not linked to } x}} & \prod_{\{k, \ell\}} \frac{\alpha(w_{k \ell}; \theta_{\mathbf{A}})}{w_{k \ell}} \frac{\beta(d_{k \ell}; \theta_{\mathbf{A}})}{d_{k \ell}} \frac{\beta(e_{k \ell}; \theta_{\mathbf{A}})}{e_{k \ell} } \\ & \frac{\delta(\Norm(\mathfrak{g}_{k \ell}); \theta_{\mathbf{A}})}{\Norm(\mathfrak{g}_{k \ell})} \frac{\delta(\Norm(\mathfrak{g}^c_{k \ell}); \theta_{\mathbf{A}})}{\Norm(\mathfrak{g}^c_{k \ell})} \prod_{m,n \neq k, \ell} \frac{\delta(f_{k \ell}^{m n}; \theta_{\mathbf{A}})}{f_{k \ell}^{m n}}.
     \end{split}
\end{equation*}

The functions $\alpha(\cdot; \theta_{\mathbf{A}}), \beta(\cdot; \theta_{\mathbf{A}}), \delta(\cdot; \theta_{\mathbf{A}}), \delta(\Norm(\cdot), \theta_{\mathbf{A}})$ are frobenian multiplicative of known means and conductor $O(|\theta_{\mathbf{A}}|^2)$. The mean of $\delta(\Norm(\cdot), \theta_{\mathbf{A}})$ is twice that of $\delta(\cdot, \theta_{\mathbf{A}})$ because a prime $p$ splits in $\Q(i)$ if and only if $p \equiv 1 \pmod{4}$. Applying Proposition~\ref{Propostion for frobenian function with small conductor}, partial summation and using that all of the frobenian multiplicative functions are bounded by $1$ we get an upper bound
\begin{equation*}
    \ll_{\varepsilon} \frac{T^2 (\log T)^{W}}{|\theta_{\mathbf{A}}|^{\frac{1}{2} - \varepsilon}} \sum_{\substack{y \text{ linked to } x \\ \Norm(y) \leq e^{(\log T^2)^{\varepsilon}}}}  \frac{1}{\prod_y \Norm(y)} \ll_{\varepsilon} \frac{T^2 (\log T)^{W + N \varepsilon}}{|\theta_{\mathbf{A}}|^{\frac{1}{2} - (N + 1) \varepsilon}}.
\end{equation*}

Where $N = 46$ is the total amount of variables and $W$ is the sum of the means of the frobenian multiplicative functions corresponding to the unlinked variables, including $x$. We will now compute $W$ for $x = \mathfrak{g}_{k \ell}, \mathfrak{g}_{k \ell}^c$ in the following table. By the notation $2 \cdot \frac{1}{2}$ we mean that there are two unlinked variables of that type and that the corresponding frobenian multiplicative function has mean $\frac{1}{2}$.

\begin{center}
\begin{tabular}{ | c | c | c | c | c | c | c | c |}
\hline
Case & $w_{k \ell}$ & $d_{k \ell}$  & $e_{k \ell}$ & $f_{k \ell}^{m n}$ & $\mathfrak{g}_{k \ell}$ & $\mathfrak{g}^c_{k \ell}$ & $W$\\ 
\hline
$ - \theta_{\mathbf{A}} \in \Q^{\times 2}$  & $2 \cdot 1/2$ & $2 \cdot 3/16$ & $2 \cdot 3/16$ & $4 \cdot 1/16$ & $2 \cdot 1/8$ & $2 \cdot 1/8$ & $5/2$\\ \hline 
$ \pm \theta_{\mathbf{A}} \not \in \Q^{\times 2}$ & $2 \cdot 1/4$ & $2 \cdot 7/32$ & $2 \cdot 7/32$ & $4 \cdot 1/32$ & $2 \cdot 1/16$ & $2 \cdot 1/16$ & $7/4$\\ \hline
\end{tabular}
\end{center}

We can now assume that $\mathfrak{g}_{k \ell}, \mathfrak{g}_{k \ell}^c = 1$. Consider now the case that $x = f_{k \ell}^{m n} \neq 1$. All the $u_i$ are linked to $x$, by \eqref{Upper bound for sum of linked variables with one variable not equal to 1} we can thus give an upper bound for this contribution of
\begin{equation*}
    \mathop{\sum_{\mathbf{w}}  \sum_{\mathbf{d}} \sum_{\mathbf{e}} \sum_{\mathbf{f}}  \sum_{\mathbf{u}}}_{\substack{|A_k| u_k^2 \prod_{\ell \neq k} w_{k \ell} d_{k \ell} e_{k \ell}  f_{k \ell}^{m n} f_{k \ell}^{n m} \leq T \\ u_k \leq e^{(\log T^2)^{\varepsilon}}}} \prod_{\{k, \ell\}}\alpha(w_{k \ell}; \theta_{\mathbf{A}}) \beta(d_{k \ell}; \theta_{\mathbf{A}}) \beta(e_{k \ell}; \theta_{\mathbf{A}}) \prod_{m,n \neq k, \ell} \delta(f_{k \ell}^{m n}; \theta_{\mathbf{A}}).
\end{equation*}

Consider the Dirichlet convolution $\rho := \alpha * \beta * \beta * \delta * \delta$. One can check by looking at prime powers that $|\rho(\cdot; \theta_{\mathbf{A}})| \leq 1$. Writing $t_{k \ell} = w_{k \ell} d_{k \ell} e_{k \ell}  f_{k \ell}^{m n} f_{k \ell}^{n m}$ we get an upper bound
\begin{equation*}
    \sum_{u_k \leq e^{(\log T^2)^{\varepsilon}}} \sum_{|A_k| u_k^2 \prod_{\ell \neq k} t_{k \ell} \leq T} \rho(t_{k \ell}; \theta_{\mathbf{A}}) \leq \sum_{u_k \leq e^{(\log T^2)^{\varepsilon}}} \sum_{|A_k| u_k^2 \prod_{\ell \neq k} t_{k \ell} \leq T} 1.
\end{equation*}

The inner sum was bounded in the proof of Lemma~\ref{Bound for N^Br when theta is a square}. To be precise, \eqref{Integral for the t region.} gives an upper bound
\begin{equation*}
    \ll \sum_{u_k \leq e^{(\log T^2)^{\varepsilon}}} \frac{T^2 (\log T)^2}{\prod_k u_k |\theta_{\mathbf{A}}|^{ \frac{1}{2}}} \ll  \frac{T^2( \log T)^{2 + 4 \varepsilon}}{ |\theta_{\mathbf{A}}|^{ \frac{1}{2}}}.
\end{equation*}

We can thus assume that $f_{k \ell}^{m n} = 1$. Consider now the case that $x = e_{k \ell} \neq 1$. We can apply the same method as for $\mathfrak{g}_{k \ell}, \mathfrak{g}_{k \ell}^c$, but we may assume that $\mathfrak{g}_{k \ell} = \mathfrak{g}_{k \ell}^c = f_{k \ell}^{m n} = 1$ so they will not contribute to $W$. We compute $W$ in the following table.

\begin{center}
\begin{tabular}{|c | c | c | c | c |}
\hline
Case & $w_{k \ell}$ & $d_{k \ell}$  & $e_{k \ell}$ & $W$\\ 
\hline
$ - \theta_{\mathbf{A}} \in \Q^{\times 2}$ & $2 \cdot 1/2$ & $2 \cdot 3/16$ & $6 \cdot 3/16$ & $5/2$\\ \hline 
$ \pm \theta_{\mathbf{A}} \not \in \Q^{\times 2}$ & $2 \cdot 1/4$ & $2 \cdot 7/32$ & $6 \cdot 7/32$ & $9/4$\\ \hline
\end{tabular}
\end{center}

We recall that $M_{\mathbf{A}} = \frac{11}{16}$ if $\theta_{\mathbf{A}} \in - \Q^{\times 2}$ and $M_{\mathbf{A}} = \frac{15}{32}$ if $\theta_{\mathbf{A}} = \pm \Q^{\times 2}$. So in the above cases $W$ is always strictly smaller than $5 M_{\mathbf{A}}$. Combining everything and changing our $\varepsilon$ we find that
\begin{equation*}
    \begin{split}
        S^{U}_{\mathbf{A}}(\bm{\chi}, T) =& \mathop{\sum \sum \sum}_{\substack{(\mathbf{u}, \mathbf{d}, \mathbf{w}) \in U \\ \mu(\prod_k u_k \prod_{\{k, \ell\}} d_{k \ell} w_{k \ell})^2 = 1}} \prod_k \chi_k(u_k)^2 \prod_{\{k, \ell\}}\chi_k \chi_\ell(w_{k \ell} d_{k \ell})\alpha(w_{k \ell}; \theta_{\mathbf{A}}) \beta( d_{k \ell}; \theta_{\mathbf{A}}) \\ &+ O_{\varepsilon}\left(\frac{T^2(\log T)^{5 M_{\mathbf{A}}}}{|\theta_{\mathbf{A}}|^{\frac{1}{2} - \varepsilon}}\right)
    \end{split}
\end{equation*}
as desired.
\end{proof}
\subsection{The error terms}
\label{Subsection for error terms}
\begin{proof}[Proof of Lemma \ref{Upper bounds for when either u_k or w_kl is small}]
By combining \eqref{Formula for N loc as sum of S(chi, T)} and \eqref{Result after removing highly linked variables} we have for all downward-closed sets $U$
\begin{equation}
    \begin{split}
        \#N^{\text{loc}, U}_{\mathbf{A}, \mathbf{M}}(T) \ll_{\varepsilon} &\mathop{\sum \sum \sum}_{(\mathbf{u}, \mathbf{d}, \mathbf{w}) \in U} \prod_{\{k, \ell\}} \alpha(w_{k \ell}; \theta_{\mathbf{A}}) \beta( d_{k \ell}; \theta_{\mathbf{A}}) + \frac{T^2 (\log T)^{5 M_{\mathbf{A}}}}{|\theta_{\mathbf{A}}|^{\frac{1}{2} - \varepsilon}}.
    \end{split}
\end{equation}
We have gotten rid of the coprimality and square-freeness with a trivial bound. 

There are two sets $U$ we have to consider.
\begin{enumerate}
    \item The $U$ we have to consider for \eqref{Contribtution from w_kl small} is defined by the inequalities $w_{k \ell} \leq (\log T^2)^{C}$ for all $\{k , \ell\}$.
    Summing over the $u_k$ and enlarging the sum gives
    \begin{equation*}
        \ll_{\varepsilon} \frac{T^2(\log T)^{\varepsilon}}{|\theta_{\mathbf{A}}|^{\frac{1}{2}}} \sum_{w_{k \ell} \leq (\log T^2)^{C}} \prod_{\{k, \ell\}}\frac{\alpha( w_{k \ell}; \theta_{\mathbf{A}})}{w_{k \ell}} \sum_{d_{k \ell} \leq T} \prod_{\{k, \ell\}}\frac{\beta( d_{k \ell}; \theta_{\mathbf{A}})}{d_{k \ell}}.
    \end{equation*}
    The functions $\alpha(\cdot ; \theta_{\mathbf{A}}), \beta(\cdot ; \theta_{\mathbf{A}})$ are frobenian multiplicative functions of conductor $\leq 16 |\theta_{\mathbf{A}}|^2$ and mean given in Lemma \ref{Lemma which describes S_U as sum over alpha and beta}. By Theorem~\ref{Propostion for frobenian function with small conductor} and partial summation this is
    \begin{equation*}
        \ll_{\varepsilon} |\theta_{\mathbf{A}}|^{12 \varepsilon- \frac{1}{2}} T^2 (\log T)^{W} (C \log \log T)^6 \ll_C |\theta_{\mathbf{A}}|^{12 \varepsilon- \frac{1}{2}} T^2 (\log T)^{5 M_{\mathbf{A}}}.
    \end{equation*}
        Here $W = 6 \cdot \frac{3}{16} < \frac{55}{16} = 5 M_{\mathbf{A}}$ if $\theta_{\mathbf{A}} \in - \Q^{\times 2}$ and $W = 6 \cdot \frac{7}{32} < \frac{75}{32} = 5M_{\mathbf{A}}$ if $\theta_{\mathbf{A}} \not \in \pm \Q^{\times 2}$.
    \item We may assume that $i = 0, j = 1$ when proving \eqref{Contribution from most u_k small}, the condition defining $U$ is thus $u_0, u_1 \leq (\log T^2)^C$. Summing first over $u_2, u_3$ and then over the $w_{k \ell} \leq T/|A_k| u_k^2  d_{k \ell} \prod_{m \neq k, \ell} d_{k m} w_{k m}$ for all $(k, \ell) = (0, 2), (1, 3)$ together with the trivial bound $\alpha(w_{k \ell}; \theta_{\mathbf{A}})\leq 1$ gives the following upper bound after enlarging the sum
    \begin{equation*}
        \ll \frac{T^2}{|\theta_{\mathbf{A}}|^{\frac{1}{2}}} \sum_{u_0, u_1 \leq (\log T^2)^C} \frac{1}{u_0 u_1}\sum_{d_{k \ell}, w_{k \ell} \leq T} \prod_{\{k, \ell \} \neq \{0,2\}} \frac{\alpha(w_{k \ell}; \theta_{\mathbf{A}})}{ w_{k \ell}} \prod_{\{k, \ell\}} \frac{\beta( d_{k \ell}; \theta_{\mathbf{A}})}{d_{k \ell}}.
    \end{equation*}
    Again, the functions $\alpha, \beta$ are frobenian of known means so using Theorem~\ref{Propostion for frobenian function with small conductor} and partial summation we get
    \begin{equation*}
          \ll |\theta_{\mathbf{A}}|^{10 \varepsilon- \frac{1}{2}} T^2 (\log T)^{W} (C\log \log T)^2 \ll |\theta_{\mathbf{A}}|^{6 \varepsilon- \frac{1}{2}} T^2 (\log T)^{5 M_{\mathbf{A}} }.
    \end{equation*}
    Here $W = 4 \cdot \frac{1}{2} + 6 \cdot \frac{3}{16} < \frac{55}{16} = 5 M_{\mathbf{A}}$ if $\theta_{\mathbf{A}} \in - \Q^{\times 2}$ and $W = 4 \cdot \frac{1}{4} + 6 \cdot \frac{7}{32} < \frac{75}{32} = 5M_{\mathbf{A}}$ if $\theta_{\mathbf{A}} \not \in \pm \Q^{\times 2}$.
\end{enumerate}
\end{proof}
We will now look at $ S_{\mathbf{A}}(\bm{\chi}, T)$. Let $\lambda(\cdot; \theta)$ be the $S \cup \{p \mid \theta\}$-frobenian multiplicative function
\begin{equation*}
    \lambda(t; \theta) := \mu(t)^2 \sum_{\substack{wd = t \\(\theta \prod_{p \in S} p , t) = 1}} \alpha(w; \theta) \beta(d; \theta).
    \label{Definition of lambda}
\end{equation*}
To see that it is frobenian multiplicative note that it is bounded by the convolution $\alpha(\cdot; \theta) *\beta(\cdot; \theta)$, which is frobenian multiplicative due to Lemma~\ref{Products and convolutions of frobenian multiplicative functions are frobenian multiplicative}, and equal to that convolution at all but finitely many primes $t$. Its mean $M_{\theta}$ is equal to the sum of the means of $\alpha(\cdot; \theta), \beta(\cdot; \theta)$. This is $11/16$ if $-\theta \in \Q^{\times 2}$ and $15/32$ if $\pm\theta \not \in \Q^{\times 2}$. We remark that $M_{\theta_{\mathbf{A}}} = M_{\mathbf{A}}$.

It is bounded by $1$ since there exists at most one pair $d,w$ such that the term in the sum is non-zero. The extension $\Q(\sqrt{-1}, \sqrt{\theta})$ defines $\lambda(\cdot; \theta)$ so it has conductor $\leq 16 |\theta|^2$.

Note that $\chi_k(x) = 0$ if $(\theta_{\mathbf{A}} \prod_{p \in S} S, x) \neq 1$ so substituting the definition \eqref{Definition of lambda} into \eqref{Result after removing highly linked variables} shows that
\begin{equation}
    \begin{split}
         S_{\mathbf{A}}(\bm{\chi}, T) = &\mathop{\sum_{\mathbf{u}} \sum_{\mathbf{t}}}_{ |A_k| u_k^2 \prod_{\ell \neq k} t_{k \ell} \leq T} \mu(\prod_k u_k \prod_{\{k, \ell\}} t_{k \ell})^2\prod_k\chi_k(u_k^2) \prod_{\{k,\ell\}} \chi_k \chi_\ell(t_{k \ell})\lambda(t_{k \ell}; \theta_{\mathbf{A}}) \\
        &+ O_{\varepsilon}\left(\frac{T^2(\log T)^{5 M_{\mathbf{A}}}}{|\theta_{\mathbf{A}}|^{\frac{1}{2} - \varepsilon}}\right).
    \end{split}
    \label{Formula for S(chi, T) as sum over u and t}
\end{equation}

We will first deal with the $\mu(\prod_k u_k \prod_{\{k, \ell\}} t_{k \ell})^2$ factor via M\"obius inversion. Let $\kappa$ be the $10$-variable multiplicative function with Dirichlet series
\begin{equation}
    \begin{split}
    F(\mathbf{s}, \mathbf{s}') =& \sum_{\mathbf{u}} \sum_{\mathbf{t}} \frac{\mu(\prod_k u_k \prod_{\{k, \ell\}} t_{k \ell})^2}{\prod_{k} u_k^{s_k} \prod_{\{k, \ell\} } t_{k \ell}^{s'_{k \ell}}} \prod_k \zeta^{-1}(s_k) \prod_{\{k, \ell\}} \zeta(s'_{k \ell})^{-1} \\ =& \prod_p(1 + \sum_k p^{-s_k} + \sum_{\{k, \ell \}} p^{-s'_{k \ell}}) \prod_k(1 -  p^{-s_k}) \prod_{\{k, \ell\}} (1 - p^{-s'_{k \ell}}).
    \end{split}
\label{Dirichlet series for omega}
\end{equation}
This converges absolutely as long as $s_i + s_k, s_i + s'_{k \ell}, s'_{i j} + s'_{k \ell} > 1$ for all $i,j,k,\ell$. By the correspondence between arithmetic functions and Dirichlet series we have
\begin{equation*}
    \mu(\prod_k u_k \prod_{\{k, \ell\}} t_{k \ell})^2 = \sum_{\substack{\mathbf{f} \\ f_{k \ell}| t_{k \ell}}} \sum_{\substack{\mathbf{g} \\ g_{k}| u_{k}}} \kappa(\mathbf{f}, \mathbf{g}).
\end{equation*}
Substituting this into \eqref{Formula for S(chi, T) as sum over u and t}, rewriting $t_{k \ell} = f_{k \ell} t_{k \ell}, u_k = g_k u_k$, using that $\lambda(\cdot; \theta_{\mathbf{A}})$ is multiplicative and supported on square-free integers and switching the order of summation shows that the main term of $ S_{\mathbf{A}}(\bm{\chi}, T)$ is
\begin{equation}
    \mathop{\sum_{\mathbf{f}} \sum_{\mathbf{g}} \sum_{\mathbf{u}} \sum_{\mathbf{t}}}_{\substack{|A_k| g_k^2 u_k^2 \prod_{\ell \neq k} f_{k \ell} t_{k \ell} \leq T \\ (f_{k \ell}, t_{k \ell}) = 1}} \kappa(\mathbf{f}, \mathbf{g}) \prod_k\chi_k(g_k^2 u_k^2) \prod_{\{k, \ell\}}
    \chi_k \chi_\ell(f_{k \ell} t_{k \ell}) \lambda(f_{k \ell} ; \theta_{\mathbf{A}})\lambda(t_{k \ell}; \theta_{\mathbf{A}}).
    \label{Formula for S(chi, T) after mobius inversion}
\end{equation}

We will prove the following lemma to deal with the $\chi_k$.
\begin{lemma}
If there exists a character $\chi \in \Gamma_{\mathbf{A}}^\vee[2]$ such that $\chi_k = \chi$ for all $k$ then
\begin{equation}
    \begin{split}
     S_{\mathbf{A}}(\bm{\chi}, T) =  S_{\mathbf{A}}(\bm{1}, T) = &\mathop{\sum_{\mathbf{f}} \sum_{\mathbf{g}} \sum_{\mathbf{u}} \sum_{\mathbf{t}}}_{\substack{|A_k| g_k^2 u_k^2 \prod_{\ell \neq k} f_{k \ell} t_{k \ell} \leq T \\(m_{\mathbf{A}}, \prod_k u_k g_k)= (f_{k \ell}, t_{k \ell}) = 1 }} \kappa(\mathbf{f}, \mathbf{g}) \prod_{\{k, \ell\}}
    \lambda(f_{k \ell} ; \theta_{\mathbf{A}})\lambda(t_{k \ell}; \theta_{\mathbf{A}}) \\ &+ O_{\varepsilon}\left(\frac{T^2(\log T)^{5 M_{\mathbf{A}}}}{|\theta_{\mathbf{A}}|^{\frac{1}{2} - \varepsilon}}\right).
    \end{split}
    \label{The formula for S(1,T)}
\end{equation}
For all other $\bm{\chi}$ we have
\begin{equation*}
      S_{\mathbf{A}}(\bm{\chi}, T) = O_{\varepsilon}\left(\frac{T^2(\log T)^{5 M_{\mathbf{A}}}}{|\theta_{\mathbf{A}}|^{\frac{1}{2} - \varepsilon}}\right).
\end{equation*}
\label{Lemma which deals with the characters}
\end{lemma}
\begin{proof}
If there exists such a character $\chi$ then $\chi_k(g_k^2 u_k^2)= \chi^2(g_k u_k)$ and $\chi_k \chi_\ell(f_{k \ell} t_{k \ell}) = \chi^2(f_{k \ell} t_{k \ell})$. Since $\chi^2$ is principal it is the indicator function of the set $\{x \in \Z: (x,m_{\mathbf{A}}) =~1\}$. The equality \eqref{The formula for S(1,T)} follows from \eqref{Formula for S(chi, T) after mobius inversion} after noting that the definition \eqref{Definition of lambda} of $\lambda$ already includes the necessary coprimality condition.

Assume now that no such $\chi$ exists. We claim that there exist $k \neq \ell$ such that $\chi_k \chi_{\ell}$ is non-principal. To prove the claim assume that for all $k \neq \ell$ the character $\chi_k \chi_{\ell}$ is principal. Then for all pairwise distinct $k, \ell, m$ we have $\chi_k = \chi_{\ell}^{-1} = \chi_m$. Hence $\chi_k = \chi_m = \chi_k^{-1}$ for all $k, m$.

We may assume that $\chi_0 \chi_1$ is non-principal. First sum over $t_{0 1}$ in \eqref{Formula for S(chi, T) after mobius inversion}. The function $\chi_0 \chi_1(\cdot) \lambda(\cdot, \theta_{\mathbf{A}})$ is a $\{p \mid m_{\mathbf{A}} f_{0 1} \}$-frobenian multiplicative function of conductor $\ll |\theta_{\mathbf{A}}|^{O(1)}$. The function $\lambda(t, \theta)$ is strictly positive if $t$ is odd squarefree so $\chi_0 \chi_1(\cdot) \lambda(\cdot, \theta_{\mathbf{A}})$ has mean zero by \cite[Lemma~2.4]{loughran2019frobenian}. Corollary \ref{Corollary of bound for frobenian function of mean 0 invariant in the conductor} and the divisor bound thus imply that this sum is
\begin{equation*}
    \begin{split}
        &\ll_{C, \varepsilon} |\theta_{\mathbf{A}}|^{\varepsilon} e^{\omega(f_{0 1})} T \min\left( \frac{\left(\log \frac{2T}{|A_0| u_0^2 t_{0 2} t_{0 3} f_{0 1} f_{0 2} f_{0 3}}\right)^{-2C}}{|A_0| u_0^2 t_{0 2} t_{0 3} f_{0 1} f_{0 2} f_{0 3}}, \frac{\left(\log \frac{2T}{|A_1| u_1^2 t_{1 2} t_{1 3} f_{0 1} f_{1 2} f_{1 3}}\right)^{-2C}}{|A_1| u_1^2 t_{1 2} t_{1 3} f_{0 1} f_{1 2} f_{1 3}}\right) \\ &\ll_{C, \varepsilon} |\theta_{\mathbf{A}} f_{0 1}|^{ \varepsilon} \frac{T \left(\log \frac{2T}{u_0^2 |A_0| t_{0 2} t_{0 3} f_{0 1} f_{0 2} f_{0 3}}\right)^{-C} \left(\log \frac{2T}{u_1^2 |A_1| t_{1 2} t_{1 3} f_{0 1} f_{1 2} f_{1 3}}\right)^{-C}}{u_0 u_1 f_{0 1}(|A_0 A_1| t_{0 2} t_{0 3}  f_{0 2} f_{0 3} t_{1 2} t_{1 3} f_{1 2} f_{1 3})^{\frac{1}{2}}}.
    \end{split}
\end{equation*}

Choose $C = 2$ and use the inequality $|\chi_k(\cdot)| \leq 1$. The sum over $u_i$ for $i = 0,1$ is bounded by the integral
\begin{equation*}
    \int_{x \geq 0}^{x^2 \leq \frac{T}{A_i t_{i 2} t_{i 3} f_{i j} f_{i 2} f_{i 3}}} \frac{\left(\log \frac{2T}{x^2 A_i t_{i 2} t_{i 3} f_{i j} f_{i 2} f_{i 3}}\right)^{-2}}{x} dx = \int_{y^2 \geq 2}^{\infty} \frac{1}{ y(2 \log y)^2} dy \ll 1.
\end{equation*}
Where we made the change of variables $y = \frac{(2T)^{\frac{1}{2}}}{x (A_i t_{i 2} t_{i 3} f_{i j} f_{i 2} f_{i 3})^{\frac{1}{2}}}.$ Summing over the $u_k$ for $k = 2,3$ and using the trivial bound $\lambda(f_{k \ell}, \theta_{\mathbf{A}}) \leq 1$ gives a total upper bound
\begin{equation*}
    \ll_{\varepsilon} \frac{T^2}{|\theta_{\mathbf{A}}|^{\frac{1}{2} - \varepsilon}} \Big( \sum_{\mathbf{f}, \mathbf{g}}|\kappa(\mathbf{f}, \mathbf{g})| f_{0 1}^{\varepsilon - 1} \sum_{t_{k \ell} \leq T} \prod_{\{k, \ell\} \neq \{0, 1\}}\frac{\lambda(t_{k \ell}; \theta_{\mathbf{A}})}{t_{k \ell} f_{k \ell}} + (\log T)^{5 M_{\mathbf{A}}} \Big).
\end{equation*}
The sum over $\mathbf{f}$ and $\mathbf{g}$ is bounded by a value of the series gotten by taking absolute values in the series $ F(\bm{1}, \mathbf{s}')$, which was defined in \eqref{Dirichlet series for omega}, with $s'_{01} = 1 - \varepsilon$ and $s'_{k \ell} = 1$ for $\{k, \ell\} \neq \{0, 1\}$. This is $O(1)$ for $\varepsilon \leq \frac{1}{4}$ since the series is absolutely convergent there. By applying Theorem~\ref{Propostion for frobenian function with small conductor} to $\lambda(\cdot; \theta_{\mathbf{A}})$ and using partial summation we can conclude that
\begin{equation*}
     S_{\mathbf{A}}(\bm{\chi}, T) \ll_{\varepsilon} \frac{T^2 (\log T)^{5 M_{\mathbf{A}}}}{|\theta_{\mathbf{A}}|^{\frac{1}{2} - \varepsilon}}.
\end{equation*}
\end{proof}

Lemma~\ref{Lemma which deals with the characters} implies that
\begin{equation}
    \begin{split}
        N^{\text{loc}}_{\mathbf{A}, \mathbf{M}}(T) &= \frac{1}{|\Gamma_{\mathbf{A}}|^4} \sum_{\chi \in \Gamma_{\mathbf{A}}^{\vee}[2]} \prod_{k} \overline{\chi}(M_k)  S_{\mathbf{A}}(\bm{1}, T) + O_{\varepsilon}\left(\frac{T^2 (\log T)^{5 M_{\mathbf{A}}}}{|\theta_{\mathbf{A}}|^{\frac{1}{2} - \varepsilon}}\right) \\ &= \frac{|\Gamma_{\mathbf{A}}^{\vee}[2]|}{|\Gamma_{\mathbf{A}}|^4}  S_{\mathbf{A}}(\bm{1}, T) + O_{\varepsilon}\left(\frac{T^2 (\log T)^{5 M_{\mathbf{A}}}}{|\theta_{\mathbf{A}}|^{\frac{1}{2} - \varepsilon}}\right).
    \end{split}
    \label{Formula for N loc as constant times S(1, T)}
\end{equation}
The second equality is because $\prod_k M_k$ has to be a square so $\chi(\prod_k M_k) = 1$ if $\chi \in \Gamma_{\mathbf{A}}^{\vee}[2]$.

\subsection{The main term}
It remains to evaluate $ S_{\mathbf{A}}(\bm{1}, T)$. By Theorem~\ref{Propostion for frobenian function with small conductor} we have the following asymptotic formula for all $\varepsilon > 0$
\begin{equation*}
    \sum_{\substack{t \leq x \\ (t,f) = 1}} \lambda(t; \theta) = \left(K_{\theta, f} M_{\theta} + e^{O(f)}O_{\epsilon}\big(|\theta|^{\varepsilon}(\log x)^{-1} \big)\right) x (\log x)^{M_{\theta} - 1}.
\end{equation*}
The constant $K_{\theta, f}$ is given by the Euler product
\begin{equation*}
    K_{\theta, f} := \frac{1}{\Gamma(M_\theta) M_{\theta}}\prod_{p \not \in S, p \nmid f \theta} (1 + \frac{\lambda(p, \theta)}{p}) \prod_p (1 - \frac{1}{p})^{M_{\theta}}.
\end{equation*}
In particular $K_{\theta, f} \leq K_{\theta, 1}$. By applying Theorem~\ref{Propostion for frobenian function with small conductor} with $N = -1$ we see that $K_{\theta, 1} \ll_{\varepsilon} |\theta|^{\varepsilon}$.

It follows from the divisor bound that $e^{O(\omega(f))} = O_{\varepsilon}(f^{\varepsilon})$. Applying partial summations and noting that $M_{\theta} < 1$ shows that
\begin{equation}
    \sum_{\substack{t \leq x \\ (t,f) = 1}} \frac{\lambda(t; \theta)}{t} = K_{\theta, f}(\log x)^{M_{\theta}} + O_{\varepsilon}\big((|\theta| f)^{\varepsilon}\big).
    \label{Asymptotic formula for lambda/t}
\end{equation}
We will use the shortened notation $M_{\mathbf{A}} := M_{\theta_{\mathbf{A}}}, K_{\mathbf{A}, f} := K_{\theta_{\mathbf{A}}, f}$ which is consistent with our earlier definition of $M_{\mathbf{A}}$.

We will now prove an asymptotic formula for the main term
\begin{lemma}
There exists a positive constant $Q'_{\mathbf{A}} \ll_{\varepsilon} |\theta_{\mathbf{A}}|^{\varepsilon}$ for all $\varepsilon$ such that for $T > 2$ we have
\begin{equation}
    S_{\mathbf{A}}(\mathbf{1}, T) = \left(Q'_{\mathbf{A}} + O_{\varepsilon}\left(\frac{|\theta_{\mathbf{A}}|^{\varepsilon}}{(\log T)^{M_{\mathbf{A}}}}\right)\right) \frac{T^{2} (\log T)^{6 M_{\mathbf{A}}}}{|\theta_{\mathbf{A}}|^{\frac{1}{2}}}.
    \label{Formula for S(1, T)}
\end{equation}
\label{Lemma for the value of S(1, T)}
\end{lemma}
\begin{proof}
We have to evaluate the right-hand side of \eqref{The formula for S(1,T)}
\begin{equation}
    \mathop{\sum_{\mathbf{f}} \sum_{\mathbf{g}} \sum_{\mathbf{u}} \sum_{\mathbf{t}}}_{\substack{|A_k| g_k^2 u_k^2 \prod_{\ell \neq k} f_{k \ell} t_{k \ell} \leq T \\ (m_{\mathbf{A}}, \prod_k u_k g_k) = (f_{k \ell}, t_{k \ell}) = 1}} \kappa(\mathbf{f}, \mathbf{g}) \prod_{\{k, \ell\}}
    \lambda(f_{k \ell} ; \theta_{\mathbf{A}})\lambda(t_{k \ell}; \theta_{\mathbf{A}}) + O_{\varepsilon}\left(\frac{T^2 (\log T)^{5 M_{\mathbf{A}}}}{|\theta_{\mathbf{A}}|^{\frac{1}{2} - \varepsilon}}\right).
    \label{Main term of S}
\end{equation}

We first sum over the $u_k$, this sum is $\frac{\varphi(m_{\mathbf{A}})}{m_{\mathbf{A}}} \frac{T}{|A_k|^{\frac{1}{2}} g_k \prod_{\ell \neq k} f_{k \ell} t_{k \ell}} + O(\tau(m_{\mathbf{A}}))$, it is also $\ll \frac{T}{|A_k|^{\frac{1}{2}} g_k \prod_{\ell \neq k} f_{k \ell} t_{k \ell}}$. To deal with the resulting error terms it thus suffices to bound
\begin{equation*}
    \frac{\tau(m_{\mathbf{A}})T^{\frac{3}{2}}}{|A_1 A_2 A_3|^{\frac{1}{2}}}\mathop{\sum_{\mathbf{f}} \sum_{\mathbf{g}} \sum_{\mathbf{t}}}_{\substack{|A_k| g_k^2 \prod_{\ell \neq k} f_{k \ell} t_{k \ell} \leq T \\ (m_{\mathbf{A}}, \prod_k g_k) = (f_{k \ell}, t_{k \ell}) = 1}} \frac{|\kappa(\mathbf{f}, \mathbf{g})| g_0 \prod_{i \neq 0} \sqrt{f_{0 i} t_{0 i}}\prod_{\{k, \ell\}}
    \lambda( f_{k \ell}; \theta_{\mathbf{A}})\lambda( t_{k \ell}; \theta_{\mathbf{A}})}{\prod_k g_{k} \prod_{\{k, \ell\}} f_{k \ell} t_{k \ell}}.
\end{equation*}
Applying the trivial inequality $\lambda( \cdot ; \theta_{\mathbf{A}}) \leq 1$, getting rid of the coprimality conditions by an upper bound, summing over $t_{0 1} \leq T/ g_0 t_{0 2} t_{03}\prod_{i \neq 0} f_{0 i}$ and enlarging the sum gives an upper bound
\begin{equation*}
    \ll \frac{\tau(m_{\mathbf{A}})T^{2}}{|\theta_{\mathbf{A}}|^{\frac{1}{2}}}\mathop{\sum_{\mathbf{f}} \sum_{\mathbf{g}} \sum_{\mathbf{t} \setminus \{t_{0 1}\}}}_{ t_{k \ell} \leq T} \frac{|\kappa(\mathbf{f}, \mathbf{g})| \prod_{\{k, \ell\}}
    \lambda( f_{k \ell}; \theta_{\mathbf{A}}) \lambda( t_{k \ell}; \theta_{\mathbf{A}})}{\prod_k g_{k} \prod_{\{k, \ell\}} f_{k \ell} t_{k \ell}}.
\end{equation*}

Using the bound $\lambda( f_{k \ell}; \theta_{\mathbf{A}}) \leq 1$ we see that the sums over $\mathbf{f}$ and $\mathbf{g}$ are bounded by the series gotten by taking absolute values in the series defining $F(\mathbf{1}, \mathbf{1})$ which was defined in \eqref{Dirichlet series for omega}. This series converges since $F$ is absolutely convergent at $(\mathbf{1}, \mathbf{1})$. The divisor bound gives $\tau(m_{\mathbf{A}}) \ll_{\varepsilon} |\theta_{\mathbf{A}}|^{\varepsilon}$. And the sum over each $t_{k \ell}$ gives a contribution $\ll_{\varepsilon} |\theta_{\mathbf{A}}|^{\varepsilon}(\log T)^{M_{\mathbf{A}}}$ by \eqref{Asymptotic formula for lambda/t}. The total error term is thus, after changing $\varepsilon$, given by $ \ll_{\varepsilon} T^{2}(\log T)^{5 M_{\mathbf{A}}}/ |\theta_{\mathbf{A}}|^{\frac{1}{2} - \varepsilon}$. We find that $S_{\mathbf{A}}(\mathbf{1}, T)$ is equal to
\begin{equation}
    \begin{split}
        &\left(\frac{\varphi(m_{\mathbf{A}})}{m_{\mathbf{A}}}\right)^4\frac{T^{2}}{|\theta_{\mathbf{A}}|^{\frac{1}{2}}} \mathop{\sum_{\mathbf{f}} \sum_{\mathbf{g}} \sum_{\mathbf{t}}}_{\substack{|A_k| g_k^2 \prod_{\ell \neq k} f_{k \ell} t_{k \ell} \leq T \\ (m_{\mathbf{A}}, \prod_k g_k) = (f_{k \ell}, t_{k \ell}) = 1}} \frac{\kappa(\mathbf{f}, \mathbf{g})\prod_{\{k, \ell\}}
        \lambda( f_{k \ell}; \theta_{\mathbf{A}}) \lambda( t_{k \ell}; \theta_{\mathbf{A}})}{\prod_k g_{k} \prod_{\{k, \ell\}} f_{k \ell} t_{k \ell}} \\ &+ O_{\varepsilon}\left(|\theta_{\mathbf{A}}|^{\varepsilon - \frac{1}{2}}T^2(\log T)^{5M_{\mathbf{A}}} \right).  
    \label{The sum after separating f and t in lambda}
    \end{split}
\end{equation}
Denote the sum over $\mathbf{t}$ by $R_{\mathbf{A}}(T; \mathbf{f}, \mathbf{g})$. We will first compare this sum with the slightly different sum
\begin{equation*}
    R'_{\mathbf{A}}(T; \mathbf{f}) := \sum_{\substack{\mathbf{t}\\ \prod_{\ell \neq k} t_{k \ell} \leq T \\ (f_{k \ell}, t_{k \ell}) = 1}} \prod_{\{k, \ell\}}\frac{
    \lambda( t_{k \ell}; \theta_{\mathbf{A}})}{t_{k \ell}}.
\end{equation*}
The difference between these two sums can be bounded by the sum over $i \in \{0, 1, 2, 3\}$ of the subsum of $R'_{\mathbf{A}}(T, \mathbf{1})$ defined by the condition $T/|A_i| g_{i}^2 \prod_{j \neq i} f_{i j} \leq \prod_{\ell \neq i} t_{i \ell} \leq T$. Choose $j \neq i$ and sum over $t_{i j}$. Using the trivial inequality $\lambda(t_{i j}; \theta_{\mathbf{A}}) \leq 1$ and ignoring the coprimality condition with an upper bound shows that this gives a contribution of
\begin{equation*}
    \ll \log \frac{T}{\prod_{\ell \neq i,j} t_{i \ell}} - \log \frac{T}{|A_i| g_{i}^2 \prod_{j \neq i} f_{i j} \prod_{\ell \neq i,j} t_{i \ell}} = \log |A_i|  g_{i}^2 \prod_{j \neq i} f_{i j}.
\end{equation*}
Summing over the other $t_{k \ell} \leq T$ and applying \eqref{Asymptotic formula for lambda/t} shows that $|R_{\mathbf{A}}(T; \mathbf{f}, \mathbf{g}) -  R'_{\mathbf{A}}(T; \mathbf{f})|$ is
\begin{equation*}
    \ll_{\varepsilon}|\theta_{\mathbf{A}}|^{5 \varepsilon}\log( |\theta_{\mathbf{A}}| \prod_{k} g_{k} \prod_{\{k, \ell\}} f_{k \ell})(\log T)^{5 M_{\mathbf{A}}} \ll_{\varepsilon} (|\theta_{\mathbf{A}}| \prod_{k} g_{k} \prod_{\{k, \ell\}} f_{k \ell})^{6 \varepsilon} (\log T)^{5 M_{\mathbf{A}}}.
\end{equation*}

To evaluate $R'_{\mathbf{A}}(T; \mathbf{f})$ we will apply Lemma~\ref{Lemma that sums can be approximated by integrals for downward-closed sets}. By \eqref{Asymptotic formula for lambda/t} we may take 
\begin{equation*}
    \begin{split}
    &g_{k \ell}(t_{k \ell}) := K_{\mathbf{A}, f_{k \ell}}\frac{d}{dt_{k \ell}} (\log t_{k \ell})^{M_{\mathbf{A}}} \text{ if } t_{k \ell} \geq 1, 0 \text{ else} \\ &h_{k \ell}(t_{k \ell}) := O_{\varepsilon}\left((|\theta_{\mathbf{A}}| f_{k \ell}\right)^{\varepsilon}) \text{ if } \frac{1}{2} \leq t_{k \ell} \leq 1, 0 \text{ else.}
    \end{split}
\end{equation*} 

By enlarging the integrals we see that the error term is bounded by the sum over $J \subsetneq \mathcal{P}_2(4)$ of the integrals
\begin{equation*}
    \prod_{\{k, \ell\} \in J} K_{\mathbf{A}, f_{k \ell}} \int_{1}^{4 T} \frac{d}{d t_{k \ell}}(\log t_{k \ell})^{M_{\mathbf{A}}} d t_{k \ell} \prod_{\{m, n\} \not \in J} O_{\varepsilon}\left((|\theta_{\mathbf{A}}| f_{m n})^{\varepsilon}\right) \int_{\frac{1}{2}}^{1}d t_{m n}.
\end{equation*}
The contribution for $\{m, n\} \not \in J$ is $O_{\varepsilon}\big((|\theta_{\mathbf{A}}| \prod_{\{m, n\}}f_{m n})^{(6- |J|)\varepsilon}\big)$ and the contribution of $\{k, \ell\} \in J$ is $O(|\theta_{\mathbf{A}}|^{|J|\varepsilon} (\log T)^{|J|M_{\mathbf{A}}})$. The total error term is thus 
\begin{equation*}
    O_{\varepsilon}\Big((|\theta_{\mathbf{A}}| \prod_{\{k, \ell\}}f_{k \ell})^{6 \varepsilon} (\log T)^{5 M_{\mathbf{A}}}\Big).
\end{equation*}

The main term is given by the integral
\begin{equation*}
    \prod_{\{k, \ell\}} K_{\mathbf{A}, f_{k \ell}}\int_{1 \leq t_{k \ell}}^{\prod_{\ell \neq k} t_{k \ell} \leq T} \prod_{\{k, \ell\}} \frac{d}{d t_{k \ell}}(\log t_{k \ell})^{M_{\mathbf{A}}} d t_{k \ell}.
\end{equation*}
By making the change of variables $s_{k \ell} = ( \log t_{k \ell}/ \log T)^{M_\mathbf{A}}$ we see that this is equal to
\begin{equation*}
     \prod_{\{k, \ell\}} K_{\mathbf{A}, f_{k \ell}} (\log T)^{6 M_{\mathbf{A}}} \int_{0 \leq s_{k \ell}}^{\sum_{\ell \neq k} s_{k \ell}^{1/M_{\mathbf{A}}} \leq 1} \prod_{\{k, \ell \}} d s_{k \ell}.
\end{equation*}
The integral is equal to a constant $0 < L_{\mathbf{A}} \leq 1$ depending on $M_\mathbf{A}$. To summarize 
\begin{equation*}
    R_{\mathbf{A}}(T; \mathbf{f}, \mathbf{g})=  \prod_{\{k, \ell\}} K_{\mathbf{A}, f_{k \ell}} L_{\mathbf{A}} (\log T)^{6 M_{\mathbf{A}}} +     O_{\varepsilon}\Big((|\theta_{\mathbf{A}}| \prod_{\{k, \ell\}}f_{k \ell} \prod_k g_k)^{\varepsilon} (\log T)^{5 M_{\mathbf{A}}}\Big).
\end{equation*}

We claim that the expression one gets after filling this into \eqref{The sum after separating f and t in lambda} converges as $T \to \infty$. Using the inequalities $K_{\mathbf{A}, f_{k \ell}} \leq  K_{\mathbf{A}, 1}, \lambda( f_{k \ell}; \theta_{\mathbf{A}}) \leq 1$ it suffices to bound the sums
\begin{equation*}
    K_{\mathbf{A},1}^6\mathop{\sum_{\mathbf{f}} \sum_{\mathbf{g}}}_{|A_k| g_k^2 \prod_{\ell \neq k} f_{k \ell} > T \text{ for some } k} \frac{|\kappa(\mathbf{f}, \mathbf{g})|}{\prod_k g_k^{1 - \varepsilon} \prod_{\{k, \ell\}} f_{k \ell}^{1 - \varepsilon}}. 
\end{equation*}
for sufficiently small $\varepsilon$. By enlarging the sum it suffices to bound the same sum over either the region defined by $g_k >  (T/ |\theta_{\mathbf{A}}|)^{\frac{1}{5}}$ for some $k$ or the region defined by $f_{k \ell} > (T/ |\theta_{\mathbf{A}}|)^{\frac{1}{5}}$ for some $\{k, \ell\}$. This sum is exactly the remainder of the series gotten by taking absolute values of the coefficients of $F(\mathbf{1} - \bm{\varepsilon}, \mathbf{1} - \bm{\varepsilon})$. This Dirichlet series converges absolutely in a large enough half open plane so that the sum is $O(\frac{|\theta_{\mathbf{A}}|^{\varepsilon}}{T^{\varepsilon}})$ for sufficiently small $\varepsilon$. We also recall that $ K_{\mathbf{A},1} \ll_{\varepsilon} |\theta_{\mathbf{A}}|^{\varepsilon}$. Hence, there exists a constant $0 < F_{\mathbf{A}}(\mathbf{1}, \mathbf{1}) \ll_{\varepsilon} |\theta_{\mathbf{A}}|^{\varepsilon}$ such that
\begin{equation*}
    \mathop{\sum_{\mathbf{f}} \sum_{\mathbf{g}}}_{\substack{|A_k| g_k^2 \prod_{\ell \neq k} f_{k \ell}\leq T }} R_{\mathbf{A}}(T; \mathbf{f}, \mathbf{g}) = F_{\mathbf{A}}(\mathbf{1},\mathbf{1}) L_{\mathbf{A}} (\log T)^{6 M_{\mathbf{A}}} + O_{\varepsilon}( |\theta_{\mathbf{A}}|^{\varepsilon} (\log T)^{5 M_{\mathbf{A}}} ).
\end{equation*}

Combining everything we find that 
\begin{equation*}
    S_{\mathbf{A}}(\mathbf{1}, T) = \left(\left(\frac{\varphi(m_{\mathbf{A}})}{m_{\mathbf{A}}}\right)^4 F_{\mathbf{A}}(\mathbf{1}, \mathbf{1}) L_{\mathbf{A}} + O_{\varepsilon}\left(\frac{|\theta_{\mathbf{A}}|^{\varepsilon}}{(\log T)^{M_{\mathbf{A}}}}\right)\right ) \frac{T^{2} (\log T)^{6 M_{\mathbf{A}}}}{|\theta_{\mathbf{A}}|^{\frac{1}{2}}}.
\end{equation*}
Since $\left(\frac{\varphi(m_{\mathbf{A}})}{m_{\mathbf{A}}}\right)^4 \leq 1,  0 < F_{\mathbf{A}}(\mathbf{1}, \mathbf{1}) \ll_{\varepsilon} |\theta_{\mathbf{A}}|^{\varepsilon}$ and $0 < L_{\mathbf{A}} \leq 1$ we are done.
\end{proof}
Substituting \eqref{Formula for S(1, T)} into \eqref{Formula for N loc as constant times S(1, T)} shows that 
\begin{equation*}
    N^{\text{loc}}_{\mathbf{A}, \mathbf{M}}(T) = \frac{|\Gamma_{\mathbf{A}}^{\vee}[2]|}{|\Gamma_\mathbf{A}|^4}\left(Q'_{\mathbf{A}} + O_{\varepsilon}\left(\frac{|\theta_{\mathbf{A}}|^{\varepsilon}}{(\log T)^{M_{\mathbf{A}}}}\right)\right ) \frac{T^{2} (\log T)^{6 M_{\mathbf{A}}}}{{|\theta_{\mathbf{A}}|^{\frac{1}{2}}}}.
\end{equation*}
This proves Lemma~\ref{Lemma for asymptotic formula of N^loc} since $|\Gamma_{\mathbf{A}}^{\vee}[2]| \leq |\Gamma_{\mathbf{A}}|$.
\section{Uniform formula}
\label{Section uniform formula}
One of the difficulties in finding upper and especially lower bounds of $N^{\Br}(T)$ was that we do not have an uniform description of $\mathcal{A}$ as this algebra depends on choosing a rational point on a quadric. We will prove in this section that no such uniform description can exist.

We will follow the ideas of \cite{uematsu2014cubic}. Let $k$ be a field of characteristic $0$, $\pi: \mathfrak{X} \to \mathbb{P}_k^3$ the universal diagonal quartic, i.e. $\mathfrak{X} \subset \mathbb{P}_k^3 \times \mathbb{P}_k^3$ is given by the equation
\begin{equation*}
    a_0 X_0^4 + a_1 X_1^4 + a_2 X_2^4 + a_3 X_3^4 = 0.
\end{equation*}
Let $K := k(\mathbb{P}_k^3)$. Write $\lambda = -\frac{a_1}{a_0}, \mu = -\frac{a_2}{a_0}, \nu = \frac{a_0 a_3}{a_1 a_2}$. We see that $K = k(\lambda,\mu,\nu)$ and the generic fiber $\mathfrak{X}_K$ of $\pi$ is the smooth $K$-variety defined by the equation
\begin{equation*}
    X_0^4 - \lambda X_1^4 - \mu X_2^4 + \lambda \mu \nu X_3^4 = 0. 
\end{equation*}
The $\lambda, \mu, \nu$ were chosen to be compatible with \cite[\S 2]{uematsu2016quadrics}. 

Uematsu defined a specialization map \cite[\S 2]{uematsu2014cubic} whose definition in this particular case we will recall. Let $A \in \Br(\mathfrak{X}_K)$. Then there exists an open subscheme $U \subset \mathbb{P}^{3}_k$ and an element $\Tilde{A} \in \Br(\mathfrak{X}_U)$ such that $\mathfrak{X}_U$ is smooth over $U$ and $A$ is the image of $\Tilde{A}$ under the injection $\Br(\mathfrak{X}_U) \to \Br(\mathfrak{X}_K)$. For every $P \in U(k)$ we define the specialization of $A$ at $P$ as
\begin{equation*}
    \text{sp}(A ; P) := \Tilde{A}|_{\mathfrak{X}_P} \in \Br(\mathfrak{X}_P).
\end{equation*} 
We will deal with the algebraic and transcendental Brauer groups separately.
\begin{proposition}
If $\sqrt[4]{2}\not \in k(i, \sqrt{2})$ then $\Br(\mathfrak{X}_K)/\Br_1(\mathfrak{X}_K) = 0.$
\label{Transcendental Brauer group of family}
\end{proposition}
\begin{proof}
For the even torsion we apply \cite[Theorem 3.8]{gvirtz2021cohomology}, the surface $\mathfrak{X}_K$ is split by the extension $K(i, \sqrt{2}, \sqrt[4]{-\lambda},\sqrt[4]{-\mu},\sqrt[4]{\nu})$ which does not contain $\sqrt[4]{2}$.

The odd torsion follows from \cite[Remark 3.10]{gvirtz2021cohomology}. To spell out the details, let $p$ be an odd prime. The remark says that the action of $\text{Gal}(\overline{K}/\overline{k}(\lambda, \mu, \nu))$ on $\Br(X_{\overline{K}})[p] \cong \Z[\sqrt{-1}]/p$ is given by multiplication with the character $\text{Gal}(\overline{K}/\overline{k}(\lambda, \mu, \nu))~\to~\mu_n$ corresponding to the quartic extension given by adjoining the fourth root of $\lambda^2 \mu^2 \nu.$ The fixed elements of this action are thus fixed by multiplication with $\sqrt{-1}$. The only such element is $0$, so 
\begin{equation*}
    \Br(X_{\overline{K}})[p]^{\text{Gal}(\overline{K}/K)} \subset \Br(X_{\overline{K}})[p]^{\text{Gal}(\overline{K}/\overline{k}(\lambda, \nu, \mu))} = 0.
\end{equation*} and $\Br(\mathfrak{X}_K)/\Br_1(\mathfrak{X}_K)[p]$ by definition injects into this group.
\end{proof}
\begin{theorem}
We have $\Br_1(\mathfrak{X}_K) /\Br_0(\mathfrak{X}_K) = 0$ but $H^1(K, \emph{Pic}(\mathfrak{X}_{\overline{K}})) \cong \Z/2\Z$. 
\label{Brauer group of generic fiber}
\end{theorem}
\begin{proof}
The Hochschild-Serre spectral sequence \cite[Theorem 2.20]{milne1980etale}  \\ $H^p(K, H^q(\mathfrak{X}_{\overline{K}}, \mathbb{G}_m)) \Rightarrow H^{p+q}(\mathfrak{X}_K, \mathbb{G}_m)$ induces an exact sequence
\begin{equation*}
    \Br(K) \to \Br_1(\mathfrak{X}_K) \to \HH^1(K, \text{Pic}(\mathfrak{X}_{\overline{K}})) \xrightarrow{d^{1,1}_{\mathfrak{X}_K}} H^3(K, \mathbb{G}_m).
\end{equation*}
It suffices to prove that the map $d^{1,1}_{\mathfrak{X}_K}$ is injective. The group $H^1(K, \text{Pic}(\mathfrak{X}_{\overline{K}}))$ has been computed by Bright \cite[Chapter 3]{bright2002computations}. He only mentions number fields, but the same arguments work over general fields of characteristic $0$. The relevant case in Appendix A is either $A222, B60, C56, D56$ or $E18$ depending on $K \cap \Q( \sqrt[4]{-1}) = k \cap \Q( \sqrt[4]{-1})$. In all cases $H^1(K, \text{Pic}(\mathfrak{X}_{\overline{K}})) \cong \Z/2\Z.$ 

Consider the open subscheme $V := \{ X_3 \neq 0\} \subset \mathfrak{X}_K$. It is the affine quartic surface defined by $x_0^4 - \lambda x_1^4 - \mu x_2^4 + \lambda \mu \nu = 0$. Let $U$ be the affine quadric surface defined by $y_0^2 - \lambda y_1^2 - \mu y_2^2 + \lambda \mu \nu = 0$ and $f: V \to U$ the map $y_i = x_i^2$. The functoriality of the Hochschild-Serre spectral sequence gives us the following commutative diagram
\begin{equation*}
    \begin{tikzcd}
{H^1(K, \text{Pic}(\mathfrak{X}_{\overline{K}}))} \arrow[rd, "d^{1,1}_{\mathfrak{X}_K}"'] \arrow[r, "\cdot|_V"] & {H^1(K, \text{Pic}(V_{\overline{K}}))} \arrow[d] & {H^1(K,\text{Pic}(U_{\overline{K}}))} \arrow[ld,"d^{1,1}_{U}" ] \arrow[l, "f^*"'] \\
                                                                       & {H^3(K, \mathbb{G}_m)}                           &                                                           
\end{tikzcd}
\end{equation*}
We have $H^1(K,\text{Pic}(U_{\overline{K}})) \cong \Z/2\Z$ and the map $d^{1,1}_{U}$ is injective \cite[Proposition~2.2, Theorem~3.1]{uematsu2016quadrics}. Let $\phi \in \HH^1(K,\text{Pic}(U_{\overline{K}}))$ be the generator. It suffices to find an element $\psi \in \HH^1(K, \text{Pic}(\mathfrak{X}_{\overline{K}}))$ such that $\psi|_V = f^* \phi$ since then $d^{1,1}_{\mathfrak{X}_K} \psi = d^{1,1}_{U} \phi \neq 0$ which implies that $\psi$ generates $H^1(K,\text{Pic}(U_{\overline{K}})) \cong \Z/2\Z$.

Let $\alpha := \sqrt{\lambda}, \alpha' := \sqrt{\mu}, \gamma := \sqrt{\nu}, \beta := \alpha \gamma$.
We define the following divisors of $\mathfrak{X}_K$.
\begin{equation*}
    \begin{split}
        & L_1: X_0^2 = \alpha X_1^2, X_2^2 = \beta X_3^2 \\
        & L_2: X_0^2 = \alpha X_1^2, X_2^2 = - \beta X_3^2 \\
        & L'_1: X_0^2 = -\alpha X_1^2, X_2^2 = \beta X_3^2 \\
        & L'_2: X_0^2 = -\alpha X_1^2, X_2^2 = - \beta X_3^2.
    \end{split}
\end{equation*}
Let $H \in \text{Pic}(\mathfrak{X}_{\overline{K}})$ be the hyperplane class.

By \cite[Corollary 2.3]{uematsu2016quadrics} $f^* \phi$ is represented by the cocycle
\begin{equation*}
    f^* \phi: \text{Gal}(K(\gamma)/K) \to \text{Pic}(V_{\overline{K}}): \text{id} \to 0, \sigma \to [L_1 \cap V].
\end{equation*}
Where $\sigma$ is the generator of $ \text{Gal}(K(\gamma)/K)$.

We claim that the following cochain is a cocycle
\begin{equation*}
    \psi: \text{Gal}(K(\gamma)/K) \to \text{Pic}(\mathfrak{X}_{\overline{K}}): \text{id} \to 0, \sigma \to [L_1] - H.
\end{equation*}
If the claim holds then $\psi|_V = f^* \phi$ as desired.

The claim consists of 2 parts, first that $[L_1]$ is $\text{Gal}(\overline{K}/K(\gamma))$-invariant. The Galois conjugates of $L_1$ are $L_1$ and $L'_2$ so this is equivalent to $[L_1] - [L'_2] = 0$. This follows from 
\begin{equation*}
\begin{split}
    & \text{div}(X_0^2 - \alpha X_1^2 - \alpha'X_2^2  +\alpha' \beta X_3^2) = L_1 + \{X_0^2  = \alpha' X_2^2, \alpha X_1^2 = \alpha' \beta X_3^2 \} \\
    & \text{div}(X_0^2 + \alpha X_1^2 - \alpha'X_2^2  - \alpha' \beta X_3^2) = L'_2 + \{X_0^2  = \alpha' X_2^2, \alpha X_1^2 = \alpha' \beta X_3^2 \}.
\end{split}
\end{equation*} 
That the left-hand side contains the right-hand side is clear. They are thus equal because they both have degree $8$. The second part of the claim is that $[L_1] + \sigma[L_1] -2H = 0$. This follows from 
$[L_1] + \sigma [L_1] = [L_1] + [L_2] = [\text{div}(X_0^2 - \alpha X_1^2)] = 2H.$
\end{proof}
This theorem allows us to prove in a precise sense that there exists no uniform formula for the $\mathcal{A}$. Let $k = \Q, K =\Q(\lambda, \mu, \nu)$.
\begin{corollary}
There exists no $A \in \Br(\mathfrak{X}_K)$ such that
\begin{equation*}
    \emph{sp}(A; P) = \mathcal{A}
\end{equation*}
 for all $P \in \mathbb{P}^3_\Q(\Q)$ for which this specialization is defined.
\label{No uniform formula}
\end{corollary}
\begin{proof}
By Proposition \ref{Transcendental Brauer group of family} and Theorem~\ref{Brauer group of generic fiber} it suffices to show that the subset of $ \mathbb{P}^3_\Q(\Q)$ on which $\mathcal{A}$ is not constant is Zariski-dense. Let $p$ be an odd prime. By Lemma~\ref{Prop if p divides 2 coefficients} the Zariski-dense set $\{ [p a_0: pa_1: a_2: a_3] :p \nmid a_0 a_1 a_2 a_3 \}$ is contained in this subset.

\subsection*{Acknowledgements}
First and foremost I want to thank Daniel Loughran for suggesting I study this question and for his extensive help with the problem and while writing this paper. 

I would also like to thank Martin Bright and Nick Rome for answering some of my questions. I'm grateful to Alexei Skorobogatov for the proof of Proposition~\ref{Transcendental Brauer group of family}.

I'm grateful to the anonymous reviewer for his helpful comments.
\end{proof}
\bibliographystyle{amsalpha}
\bibliography{references}
\end{document}